\documentclass{amsart}

\usepackage{amscd,amssymb}

\newcommand{\I}{\mathbf 1}
\newcommand{\C}{\mathbf C}
\newcommand{\F}{\mathbf F}

\newcommand{\Q}{\mathbf Q}

\newcommand{\Z}{\mathbf Z}

\newcommand{\bG}{\mathbf G}

\newcommand{\sA}{\mathcal A}
\newcommand{\sB}{\mathcal B}
\newcommand{\sC}{\mathcal C}
\newcommand{\sD}{\mathcal D}
\newcommand{\sE}{\mathcal E}
\newcommand{\sF}{\mathcal F}

\newcommand{\sH}{\mathcal H}
\newcommand{\sI}{\mathcal I}
\newcommand{\sJ}{\mathcal J}
\newcommand{\sK}{\mathcal K}
\newcommand{\sL}{\mathcal L}
\newcommand{\sM}{\mathcal M}
\newcommand{\sN}{\mathcal N}
\newcommand{\sO}{\mathcal O}
\newcommand{\sP}{\mathcal P}

\newcommand{\sR}{\mathcal R}
\newcommand{\sS}{\mathcal S}

\newcommand{\sV}{\mathcal V}

\newcommand{\iso}{\xrightarrow{\sim}}

\newcommand{\nd}{\nobreakdash-\hspace{0pt}}

\renewcommand{\theenumi}{(\roman{enumi})}

\DeclareMathOperator{\Alb}{Alb}

\DeclareMathOperator*{\colim}{colim}
\DeclareMathOperator{\End}{End}

\DeclareMathOperator{\Fr}{Fr}

\DeclareMathOperator{\Hom}{Hom}
\DeclareMathOperator{\Id}{Id}
\DeclareMathOperator{\id}{id}

\DeclareMathOperator{\Ind}{Ind}

\DeclareMathOperator{\Ker}{Ker}

\DeclareMathOperator{\Ob}{Ob}

\DeclareMathOperator{\Pic}{Pic}
\DeclareMathOperator{\pr}{pr}

\DeclareMathOperator{\Rad}{Rad}

\DeclareMathOperator{\Rep}{Rep}
\DeclareMathOperator{\REP}{REP}
\DeclareMathOperator{\Spec}{Spec}
\DeclareMathOperator{\Sym}{Sym}
\DeclareMathOperator{\Tor}{Tor}
\DeclareMathOperator{\tr}{tr}

\newtheorem{thm}{Theorem}[subsection]
\newtheorem{cor}[thm]{Corollary}
\newtheorem{lem}[thm]{Lemma}
\newtheorem{prop}[thm]{Proposition}

\theoremstyle{definition}
\newtheorem{defn}[thm]{Definition}

\newtheorem*{thm*}{Theorem}
\newtheorem*{defn*}{Definition}

\numberwithin{equation}{subsection}

\begin{document}

\title{Algebraic cycles on an abelian variety}

\author{Peter O'Sullivan}
\address{School of Mathematics and Statistics F07\\
University of Sydney NSW 2006 \\
Australia}
\email{petero@maths.usyd.edu.au}
\thanks{The author was supported by ARC Discovery Project grant DP0774133.}



\date{}

\dedicatory{}

\begin{abstract}
It is shown that to every $\Q$\nd linear cycle $\overline{\alpha}$
modulo numerical equivalence on an abelian variety $A$ there is canonically
associated a $\Q$\nd linear cycle $\alpha$ modulo rational equivalence on $A$
lying above $\overline{\alpha}$, characterised by a condition on
the spaces of cycles generated by $\alpha$ on products of $A$ with itself.
The assignment $\overline{\alpha} \mapsto \alpha$ respects
the algebraic operations and pullback and push forward along homomorphisms
of abelian varieties.
\end{abstract}

\maketitle

\tableofcontents

\section{Introduction}

Let $A$ be an abelian variety, equipped with a base point, over a field $F$.
Denote by $CH(A)_\Q$ the graded $\Q$\nd algebra of $\Q$\nd linear cycles
on $A$ modulo rational equivalence, and by $\overline{CH}(A)_\Q$ its
quotient consisting of the $\Q$\nd linear cycles modulo numerical equivalence.
Given an $\overline{\alpha}$ in $\overline{CH}{}^i(A)_\Q$, Beauville
has shown (\cite{Bea}, Proposition~5) that there is an $\alpha$ in
$CH^i(A)_\Q$ lying above
$\overline{\alpha}$ such that $(n_A)^*(\alpha) = n^{2i} \alpha$ for every integer
$n$, where $n_A:A \to A$ is multiplication by $n$ on $A$.
He also raised the question whether such an $\alpha$ is unique.
The uniqueness of such an $\alpha$ would imply that every
$\overline{\alpha}$ in $\overline{CH}(A)_\Q$
has a canonical lift $\alpha$ in $CH(A)_\Q$, such that the assignment
$\overline{\alpha} \mapsto \alpha$ respects the algebraic operations and
pullback and push forward along homomorphisms of abelian varieties.
We show in this paper that such a canonical lift $\alpha$ indeed exists.
It is characterised by the property that $\alpha$ is the unique cycle lying
above $\overline{\alpha}$ which is symmetrically distinguished in the sense of
the definition below.
The endomorphism $n_A$ of $A$ is used in this definition only for $n = -1$.
It is thus necessary to consider also cycles generated by $\alpha$
on products $A^m$ of $A$ with itself,
because for $i > 1$ an $\alpha$ with $((-1)_A)^*\alpha = \alpha$
above a given $\overline{\alpha}$ is in general not unique.
In the definition
$\alpha^{r_1} \otimes \alpha^{r_2} \otimes \dots \otimes \alpha^{r_n}$ denotes the product in
$CH(A^n)_\Q$ of the $(\pr_j)^*(\alpha^{r_j})$, where $\pr_j:A^n \to A$ is the $j$th projection.

\begin{defn*}
Let $\alpha$ be a cycle in $CH^i(A)_\Q$.
For each integer $m \ge 0$, denote by $V_m(\alpha)$
the $\Q$-vector subspace of $CH(A^m)_\Q$ generated by elements of the form
\[
p_*(\alpha^{r_1} \otimes \alpha^{r_2} \otimes \dots \otimes \alpha^{r_n}),
\]
where $n \le m$, the $r_j$ are integers $\ge 0$,
and $p:A^n \rightarrow A^m$ is a closed immersion
with each component $A^n \rightarrow A$ either a projection
or the composite of a projection with $(-1)_A:A \rightarrow A$.
Then $\alpha$ will be called \emph{symmetrically distinguished}
if  for every $m$ the restriction of the projection
$CH(A^m)_\Q \rightarrow \overline{CH}(A^m)_\Q$
to $V_m(\alpha)$ is injective.
\end{defn*}

The main result is now the following (cf.\ Corollary~\ref{c:symdist}).

\begin{thm*}
\mbox{}
\begin{enumerate}
\item\label{i:sdexist}
Above every cycle in $\overline{CH}{}^i(A)_\Q$ there lies a unique symmetrically
distinguished cycle in $CH^i(A)_\Q$.
\item\label{i:sdsubalgebra}
The symmetrically distinguished cycles in $CH^i(A)_\Q$ form a $\Q$\nd vector subspace,
and the product of  symmetrically distinguished cycles in $CH^i(A)_\Q$ and $CH^j(A)_\Q$
is symmetrically distinguished in $CH^{i+j}(A)_\Q$.
\item\label{i:sdpushpull}
For any homomorphism of abelian varieties $f:A \to A'$, the pullback $f^*$ and push forward
$f_*$ along $f$ preserve symmetrically distinguished cycles.
\end{enumerate}
\end{thm*}

It follows immediately from this result that the assignment to each
cycle in $\overline{CH}{}^i(A)_\Q$ of the unique symmetrically distinguished cycle
in $CH^i(A)_\Q$ lying above it respects the algebraic operations and pullback and push forward
along homomorphisms of abelian varieties.
In particular the subspace of symmetrically distinguished cycles in $CH^i(A)_\Q$
is contained in the subspace of those cycles $\alpha$ for which
$(n_A)^*(\alpha) = n^{2i} \alpha$ for every $n$, and coincides with it
if there is a unique such $\alpha$ lying above a given element of $\overline{CH}{}^i(A)_\Q$.

The above Theorem will be proved in Section~\ref{s:symdist},
after preparation in the preceding four sections.
A fundamental role in the proof is played by the properties of Kimura categories
in the sense of \cite{AndKah},~\S~9, and by the relations
between motives and algebraic cycles.

To describe the proof, we first introduce some terminology.
Let $\sV$ be category with finite products.
We assume that $\sV$ is equipped with an assignment
to each object $X$ in $\sV$ of an integer $\dim X$, depending only on the isomorphism
class of $X$, such that $\dim X \times Y = \dim X + \dim Y$.
For example we may take for $\sV$ the category $\sA\sS_F$ of abelian varieties over $F$,
or of smooth projective varieties over $F$ which are non-empty and equidimensional,
with $\dim$ the usual dimension.
Then a Chow theory $C$ with source $\sV$ is an assignment to each object
$X$ in $\sV$ of a $\Z$\nd graded commutative $\Q$\nd algebra $C(X)$ and
to each $f:X \to Y$ in $\sV$ of a homomorphism $f^*:C(Y) \to C(X)$ of graded $\Q$\nd algebras
and a homomorphism \mbox{$f_*:C(X) \to C(Y)$} of degree $\dim Y - \dim X$ of graded
$\Q$\nd vector spaces, satisfying appropriate compatibilities.
A morphism $C \to C'$ of such Chow theories is a homomorphism $C(X) \to C'(X)$ of
graded $\Q$\nd algebras for every $X$ in $\sV$, compatible with the homomorphisms
$f^*$ and $f_*$.

Consider in particular the Chow theories $C$ and $\overline{C}$ with source $\sA\sS_F$
which send $A$ respectively to $CH(A)_\Q$ and $\overline{CH}(A)_\Q$, with $f^*$
and $f_*$ the usual pullback and push forward.
The projections  $CH(A)_\Q \to \overline{CH}(A)_\Q$ define a projection morphism
$C \to \overline{C}$.
The Theorem now reduces to the following two statements.
\begin{enumerate}
\renewcommand{\theenumi}{(\arabic{enumi})}
\item\label{i:sdrinverse}
The projection $C \to \overline{C}$ has a right inverse.
\item\label{i:sdulift}
Any two symmetrically distinguished cycles in $CH^i(A)_\Q$ lying above the
same cycle in $\overline{CH}{}^i(A)_\Q$ coincide.
\end{enumerate}
Indeed if  $\tau$ is right inverse to $C \to \overline{C}$, \ref{i:sdexist} will follow because
$\tau(\overline{\alpha})$
is the unique symmetrically distinguished cycle lying above $\overline{\alpha}$,
and \ref{i:sdsubalgebra} and \ref{i:sdpushpull} will then follow because
$\tau$ is a morphism of Chow theories.
We may reformulate \ref{i:sdulift} as follows.
Consider the (non-full) subcategory $\sE_A$ of $\sA\sS_F$  consisting of the powers of
$A^m$ of $A$ and
those morphisms $A^m \to A^n$ for which every component $A^m  \to A$ is either a projection
or the composite of a projection with $(-1)_A$.
By restricting $C$ and $\overline{C}$ to $\sE_A$, we obtain Chow theories
$C_A$ and $\overline{C}_A$ with source $\sE_A$, and a projection
$C_A \to \overline{C}_A$.
Given $\alpha$ in $CH^i(A)_\Q = (C_A)^i(A)$ above $\overline{\alpha}$ in
$\overline{CH}{}^i(A)_\Q$, write
$C_\alpha$ and $\overline{C}_{\overline{\alpha}}$ for the Chow subtheories
of $C_A$ and $\overline{C}_A$ generated by $\alpha$ and $\overline{\alpha}$.
Then $C_A \to \overline{C}_A$ defines
a morphism $C_\alpha \to \overline{C}_{\overline{\alpha}}$.
The space $V_m(\alpha)$ in the above Definition is contained in $C_\alpha(A^m)$
but does not in general coincide with it.
However, it can be shown (cf.\ Lemma~\ref{l:symdistequiv}) that $\alpha$ is
symmetrically distinguished if and only if
$C_\alpha \to \overline{C}_{\overline{\alpha}}$ is an isomorphism.
Then \ref{i:sdulift} is equivalent to the following statement.
\begin{enumerate}
\item[\ref{i:sdulift}$'$]
For $\overline{\alpha}$ in $\overline{CH}{}^i(A)_\Q$,
any two liftings $\overline{C}_{\overline{\alpha}} \to C_A$ of the embedding
$\overline{C}_{\overline{\alpha}} \to \overline{C}_A$ along the projection
$C_A \to \overline{C}_A$ coincide.
\end{enumerate}
Indeed by composing the inverse of $C_\alpha \to \overline{C}_{\overline{\alpha}}$
with the embedding $C_\alpha \to C_A$, we may identify a symmetrically
distinguished $\alpha$ lying above $\overline{\alpha}$ with a lifting
$\overline{C}_{\overline{\alpha}} \to C_A$ of
$\overline{C}_{\overline{\alpha}} \to \overline{C}_A$ along $C_A \to \overline{C}_A$.

The statements \ref{i:sdrinverse} and \ref{i:sdulift}$'$ are proved
by first embedding the category of Chow theories into an
appropriate category of Poincar\'e duality theories, by essentially
the same construction as for the usual categories of motives.
To describe this, we begin with some definitions.

A $\Q$\nd pretensor category is a $\Q$\nd linear category equipped with
a compatible structure of symmetric monoidal category for which the unit $\I$ is strict,
and a $\Q$\nd tensor category is a pseudo-abelian $\Q$\nd pretensor category.
A $\Q$\nd tensor functor between $\Q$\nd pretensor categories is a (strong) symmetric monoidal
functor which strictly preserves the unit $\I$ and is $\Q$\nd linear.
A $\Q$\nd pretensor category is said to be rigid if each of its objects has a dual.
If $\sC$ is a rigid $\Q$\nd pretensor category with $\End(\I) = \Q$, then $\sC$ has a unique
maximal tensor ideal.
We write $\overline{\sC}$ for the quotient of $\sC$
by this ideal.

Define a Tate $\Q$\nd pretensor category as a $\Q$\nd pretensor category $\sC$ equipped with a
family of objects $\I(i)$ for $i$ in $\Z$, such that $\I(0) = \I$ and
$\I(i) \otimes \I(j) = \I(i+j)$ for every $i$ and $j$, and such that those
associativity or commutativity constraints in $\sC$ which involve an object $\I(i)$
are identities.
If $M$ is an object in such a $\sC$, we write $M(i)$ for $M \otimes \I(i)$.
A Tate $\Q$\nd tensor functor is a $\Q$\nd tensor functor $T:\sC \to \sC'$ such that
$T(\I(i)) = \I(i)$ and the isomorphism from $T(M) \otimes T(N)$ to $T(M \otimes N)$
defining the monoidal structure of $T$ is the identity when $M$ or $N$ is $\I(i)$.

A Poincar\'e duality theory with source $\sV$ is a triple $(\sM,h,\nu)$, where $\sM$
a Tate $\Q$\nd pretensor category, $h$ is a symmetric monoidal functor from
$\sV^\mathrm{op}$ to $\sM$,
and $\nu$ is an assignment to each object $X$ of $\sV$ of a morphism
$\nu_X:h(X)(\dim X) \to \I$ in $\sM$.
It is required that the composite of the twist by $\dim X$ of multiplication
of the algebra $h(X)$ with $\nu_X$ should define a duality pairing between $h(X)$
and $h(X)(\dim X)$, and that the $\nu_X$ be compatible with the isomorphisms and
products in $\sV$.
A morphism from $(\sM,h,\nu)$ to $(\sM',h',\nu')$ is a Tate $\Q$\nd tensor functor
$T:\sM \to \sM'$ such that
$Th = h'$ and $T(\nu_X) = \nu'_{}\!_X$ for every $X$.
To a Poincar\'e duality theory $(\sM,h,\nu)$ with source $\sV$ we assign
a Chow theory $\Hom_\sM(\I,h(-)(\cdot))$ with source $\sV$, which sends $X$
to the graded $\Q$\nd algebra with component $\Hom_\sM(\I,h(X)(i))$
of degree $i$.
This assignment extends in an evident way to a functor from
Poincar\'e duality theories to Chow theories with source $\sV$.
The functor so defined has a fully faithful left adjoint, which sends a Chow
theory with source $\sV$ to its associated Poincar\'e duality theory.
If $(\sM,h,\nu)$ is the Poincar\'e duality theory associated to a Chow theory,
then each object of $\sM$ can be written uniquely in the form $h(X)(i)$.
In particular $\sM$ is rigid.
It is however in general not pseudo-abelian:
the usual categories of motives
are obtained by passing to a pseudo-abelian hull.

The targets of the Poincar\'e duality theories associated to the Chow
theories in \ref{i:sdrinverse} and \ref{i:sdulift}$'$ are closely related
to the category $\sM_{F,\Q}^\mathrm{ab}$ of $\Q$\nd linear abelian Chow motives,
generated by the motives abelian varieties over $F$.
The quotient $\overline{\sM}{}_{F,\Q}^\mathrm{ab}$ of $\sM_{F,\Q}^\mathrm{ab}$ is
the category of $\Q$\nd linear abelian motives modulo numerical equivalence.
We have Poincar\'e duality theories $(\sM_{F,\Q}^\mathrm{ab},h{}^\mathrm{ab},\nu{}^\mathrm{ab})$
and $(\overline{\sM}{}_{F,\Q}^\mathrm{ab},\overline{h}{}^\mathrm{ab},\overline{\nu}{}^\mathrm{ab})$
with source $\sA\sS_F$, where $h{}^\mathrm{ab}$ and $\overline{h}{}^\mathrm{ab}$ are the
usual cohomology functors sending an abelian variety to its motive.
The projection $\sM_{F,\Q}^\mathrm{ab} \to \overline{\sM}{}_{F,\Q}^\mathrm{ab}$ is then
a morphism of Poincar\'e duality theories.
For every $A$, the involution $(-1)_A$ of $A$ induces an involution on the
algebras $h{}^\mathrm{ab}(A)$ and $\overline{h}{}^\mathrm{ab}(A)$.
The proofs of both \ref{i:sdrinverse} and \ref{i:sdulift}$'$ rely on the following fact
(cf.\ Theorem~\ref{t:alginvunique}):
if $R$ and $R'$ are commutative algebras with an involution in $\sM_{F,\Q}^\mathrm{ab}$
lying respectively above $\overline{h}{}^\mathrm{ab}(A)$ and $\overline{h}{}^\mathrm{ab}(A')$
in $\overline{\sM}{}_{F,\Q}^\mathrm{ab}$, then above every morphism
$\overline{h}{}^\mathrm{ab}(A) \to \overline{h}{}^\mathrm{ab}(A')$ of algebras with
involution there lies a unique morphism of algebras with involution $R \to R'$.

It is known that $\sM_{F,\Q}^\mathrm{ab}$ is a Kimura category.
A $\Q$\nd tensor category $\sC$ is said to be a Kimura category if it is rigid
with $\End(\I) = \Q$,
and if each object of $\sC$ is a direct sum of an object
for which some exterior power is $0$ and an object for which some symmetric power is $0$.
If $\sC$ is such a category, then $\overline{\sC}$ is a Kimura category which is semisimple
abelian.
The following splitting theorem proved in \cite{AndKah}, (cf.\ Theorem~\ref{t:splitting})
is fundamental for the proof \ref{i:sdrinverse}:
if $\sC$ is a Kimura $\Q$\nd tensor category, then the projection
$\sC \to \overline{\sC}$ has a right inverse.
The following unique lifting theorem (cf.\ Theorem~\ref{t:uniquelift}),
which will be proved in Section~\ref{s:Kimura},
is fundamental for the proof of \ref{i:sdulift}$'$:
if $\sD$ and $\sC$ are Kimura $\Q$\nd tensor categories,
then between any two liftings $\sD \to \sC$ along the projection $\sC \to \overline{\sC}$
of a faithful $\Q$\nd tensor functor $\sD \to \overline{\sC}$ there exists
a tensor isomorphism lying above the identity.

Once the above abstract machinery has been constructed, the proofs of \ref{i:sdrinverse}
and \ref{i:sdulift}$'$ are straightforward.
In the case of \ref{i:sdrinverse} (cf.\ Theorem~\ref{t:Chowsplit}),
we consider the Poincar\'e duality
theory $(\sM,h,\nu)$ associated to $C$.
Then the Poincar\'e duality associated to $\overline{C}$
is of the form $(\overline{\sM},\overline{h},\overline{\nu})$,
and the morphism $P$ from $(\sM,h,\nu)$ to $(\overline{\sM},\overline{h},\overline{\nu})$
induced by the projection $C \to \overline{C}$ is given by
the projection $\sM \to \overline{\sM}$.
Since the associated Poincar\'e duality theory functor is fully faithful,
it is enough to show that $P$ has a right inverse.
Now by the universal property of $(\sM,h,\nu)$ we have an embedding of $\sM$
into $\sM_{F,\Q}^\mathrm{ab}$.
In particular the pseudo-abelian hull of $\sM$ is a Kimura category.
The splitting theorem for such categories then shows that there is a $\Q$\nd tensor
functor $T$ right inverse to $P$.
In general, such a $T$ will not be a morphism from
$(\overline{\sM},\overline{h},\overline{\nu})$ to $(\sM,h,\nu)$.
However the theorem on algebras with involution associated to abelian varieties
gives a canonical isomorphism \mbox{$\xi:h \iso T\overline{h}$}.
We can use $\xi$ to modify $T$ so that the condition $h = T\overline{h}$ is
satisfied.
We can then modify $T$ again so that in addition $T$ is a Tate
$\Q$\nd tensor functor.
Such a $T$ is necessarily compatible with $\overline{\nu}$ and $\nu$,
and hence gives the required right inverse.

In the case of \ref{i:sdulift} (cf.\ Theorem~\ref{t:Chowunique}),
we consider the Poincar\'e duality theories
$(\sM_0,h_0,\nu_0)$ and $(\sM_1,h_1,\nu_1)$  associated
respectively to $\overline{C}_{\overline{\alpha}}$ and $C_A$.
The morphism $P$
induced by the  projection
$C_A \to \overline{C}_A$ is given by the projection
$\sM_1 \to \overline{\sM}_1$,
and the morphism $K$ induced by the embedding
$\overline{C}_{\overline{\alpha}} \to \overline{C}_A$
is given by a faithful $\Q$\nd tensor functor $\sM_0 \to \overline{\sM}_1$.
It is enough to show that any two liftings of $K'$ and $K''$ of $K$
along $P$ coincide.
The pseudo-abelian hulls of both $\sM_1$ and
(since $\sE_A$ contains the $(-1)_{A^n}$) $\sM_0$ are Kimura categories.
By the unique lifting theorem for such categories there is a
tensor isomorphism $\varphi:K' \iso K''$ lying above the identity of $K$.
Since $K'$ and $K''$ are morphisms of Poincar\'e duality theories,
$K'h_0$ and $K''h_0$ coincide with $h_1$.
Thus $\varphi h_0$ is an automorphism of $h_1$ above the identity of $h_0$.
The theorem on algebras with involution associated to abelian varieties
then shows that every $\varphi_{h_0(A^n)} = (\varphi h_0)_{A^n}$ is the identity.
It is immediate that every $\varphi_{\I(i)}$ is the identity.
Since $K'$ and $K''$ are Tate $k$\nd tensor functors, and every object of
$\sM_0$ is of the form $h_0(A^n)(i)$, it follows that $\varphi$ is the identity
and $K' = K''$.

\medskip

This paper is organised as follows.
In Section~\ref{s:tensor}, we recall what is needed
concerning tensor categories and rigid categories.
The only new results are Lemma~\ref{l:frmod}
and those in \ref{ss:Tate}.
Section~\ref{s:ChowPoin} gives the definitions for Chow and
Poincar\'e duality theories
and the construction of the associated Poincar\'e duality theory.
Section~\ref{s:Kimura} contains what is needed concerning Kimura categories.
The main new results are the characterisation of symmetric Hopf algebras in a Kimura
category in \ref{ss:Hopf}, and the unique lifting theorem for Kimura categories
in \ref{ss:splitlift}.
Motives are considered in Section~\ref{s:mot}, particularly from the point of
view of Kimura categories.
We work over a regular noetherian base, subject to some
technical conditions, and with cycles linear over an arbitrary field of characteristic
zero.
The main result  needed for the proof of the
theorems on symmetrically distinguished cycles
is Theorem~\ref{t:abeliansym},
giving the structure of the motive of an abelian scheme.
This is essentially due to K\"unnemann~\cite{KunAbSch},
but is here obtained as an application of the Hopf theorem of \ref{ss:Hopf}.
It is used in \ref{ss:Abmot} to prove Theorem~\ref{t:alginvunique}
on algebras with involution associated to abelian schemes,
on which Theorems~\ref{t:Chowsplit} and \ref{t:Chowunique} depend.
A brief account of the motivic algebra, which is a commutative algebra in the
category of ind-objects of abelian motives modulo numerical equivalence
describing the structure of abelian Chow motives, is given in \ref{ss:motalg}.
This will not be required in Section~\ref{s:symdist}.
Theorems~\ref{t:Chowsplit} and \ref{t:Chowunique}
are proved in Section~\ref{s:symdist} by the method outlined above.
We work with Chow groups of abelian schemes with coefficients in a field of
characteristic $0$, modulo an equivalence relation.
The main remaining difficulty is the proof of the criterion for a symmetrically
distinguished cycle given in Lemma~\ref{l:symdistequiv}.
In \ref{ss:conclrem} some potential applications of the methods
used here to varieties other than abelian varieties are discussed.

Throughout this paper, $k$ denotes a commutative ring, which
from Section~\ref{s:Kimura} on is a field of characteristic zero.
If $\sC$ is a category, we also write $\sC(M,N)$ for $\Hom_\sC(M,N)$,
and if $T:\sC \to \sC'$ is a functor we write $T_{M,N}$ for
the map from $\sC(M,N)$ to $\sC'(T(M),T(N))$ defined by $T$.
A category will be called essentially small if it is equivalent to a
small category.
A morphism in a category will be called a section if it has a left inverse
and a retraction if it has a right inverse.

\section{Tensor categories}\label{s:tensor}

\subsection{Pretensor categories}\label{ss:pretensor}

Recall (see for example \cite{Mac}, XI.1) that
a symmetric monoidal category is a category $\sC$ equipped
with a with a functor $\otimes:\sC \times \sC \to \sC$, the tensor product,
an object $\I$, the unit, natural isomorphisms from
$(L \otimes M) \otimes N$ to $L \otimes (M \otimes N)$, the associativities,
from $M \otimes N$ to $N \otimes M$, the symmetries,
and from $\I \otimes M$ to $M$ and $M \otimes \I$ to $M$,
subject to appropriate compatibility conditions.
If $\sC$ and $\sC'$ symmetric monoidal categories, a symmetric monoidal functor
(called a braided strong monoidal functor in \cite{Mac}, XI.2)
from $\sC$ to $\sC'$ is a functor $T$ from $\sC$ to $\sC'$
equipped with natural isomorphisms,
the structural isomorphisms,
from $\I$ to $T(\I)$ and $T(M) \otimes T(N)$ to $T(M \otimes N)$,
subject to appropriate compatibility conditions with the symmetric monoidal
structures of $\sC$ and $\sC'$.
If $T$ and $T'$ are symmetric monoidal functors from $\sC$ to $\sC'$, then
a monoidal natural transformation from $T$ to $T'$ is a natural transformation
$\varphi$ from $T$ to $T'$ such that, modulo the structural isomorphisms,
$\varphi_\I$ coincides with $1_\I$ and $\varphi_{M \otimes N}$ with
$\varphi_M \otimes \varphi_N$.

Let $\sC$ be a symmetric monoidal category.
Given for each pair of objects $M$ and $N$ of $\sC$
an isomorphism $\zeta_{M,N}$ in $\sC$ with target $M \otimes N$,
there exists a unique symmetric monoidal category $\widetilde{\sC}$ with
the same underlying category as $\sC$ such that
we have a symmetric monoidal functor $\sC \to \widetilde{\sC}$ which is
the identity on underlying categories
and whose structural morphisms are $1_\I$ and the $\zeta_{M,N}$.
We say that $\widetilde{\sC}$ is obtained by modifying the
tensor product of $\sC$ according to the family $(\zeta_{M,N})$.
Let $T$ be a symmetric monoidal functor from $\sC$ to $\sC'$.
Given for each object $M$ of $\sC$ an isomorphism $\xi_M$ in $\sC'$ with
target $T(M)$, there exists a unique symmetric monoidal functor
$\widetilde{T}$ from $\sC$ to $\sC'$ such that the $\xi_M$ are the
components of a monoidal isomorphism from $\widetilde{T}$ to $T$.
We say that $\widetilde{T}$ is obtained by modifying $T$ according
to the family $(\xi_M)$.

The unit $\I$ of a symmetric monoidal category $\sC$ is said to be strict if
the isomorphisms $\I \otimes M \iso M$ and $M \otimes \I \iso M$
are identities.
The symmetric monoidal functor $T$ from $\sC$ to $\sC'$ is said to strictly
preserve the units if $\I \iso T(\I)$ is the identity.
If the units of $\sC$ and $\sC'$ are strict and $T$ strictly preserves them,
then the structural isomorphism $T(M) \otimes T(N) \iso T(M \otimes N)$
of $T$ is the identity whenever $M$ or $N$ is $\I$.
We obtain from any symmetric monoidal category $\sC$ a symmetric monoidal
category whose unit is strict by modifying the tensor product of $\sC$
according to the family $(\zeta_{M,N})$ with $\zeta_{M,N}$ the inverse
of $\I \otimes N \iso N$ when $M = \I$,  the inverse
of $M \otimes \I \iso M$ when $N = \I$, and the identity otherwise.
Similarly we obtain from any symmetric monoidal functor $T$ a symmetric
monoidal functor which strictly preserves the units by modifying $T$
according to the family $(\xi_M)$ with $\xi_M$ the structural
isomorphism $\I \iso T(\I)$ when $M = \I$ and the identity otherwise.

\emph{In what follows, it will always be assumed unless the contrary is stated
that the units of any symmetric monoidal categories that occur are strict
and that any symmetric monoidal functors that occur strictly preserve
the units}.
When modifying the tensor product of a symmetric monoidal category according
to a family $(\zeta_{M,N})$, it will then be necessary to require that $\zeta_{M,N}$
be the identity when either $M$ or $N$ is $\I$.
Similarly when modifying a symmetric monoidal functor according to a family
$(\xi_M)$ it will be necessary to require that $\xi_\I$ be the identity.

A symmetric monoidal category is said to be strict if each of
its associativities is the identity,
and a symmetric monoidal functor is said to be strict if each of its
structural isomorphisms is the identity.
In a strict symmetric monoidal category multiple tensor products may be written
without brackets.
For any symmetric monoidal category $\sC$, there is a strict symmetric monoidal
category $\sC'$ and a strict symmetric monoidal functor from $\sC$ to $\sC'$
which is an equivalence (\cite{Mac}, Theorem~XI.2.1).
Any symmetric monoidal functor $T$ has a factorisation $T = T''T'$, essentially unique,
with $T'$ a strict symmetric monoidal functor which is bijective on objects and $T''$
fully faithful.

Even if a symmetric monoidal category is not strict, we still often write
multiple tensor products without brackets when it is of no importance which
bracketing is chosen.
The tensor product of $n$ copies of an object $M$ will then be written
$M^{\otimes n}$, and similarly for morphisms.
Using the appropriate symmetries and associativities, we obtain an action
of the symmetric group $\mathfrak{S}_n$ of degree $n$ on $M^{\otimes n}$
by permuting the factors.

Let $\sV$ be a category in which finite products exist.
To choose a final object $\I$ and a specific product and projections
for every pair of objects in $\sV$, such that the projections
$X \times \I \to X$ and $\I \times X \to X$
are identities, is the same as to choose a symmetric monoidal structure
on $\sV$ such that $\I$ is a final object
of $\sV$ and for every $X$ and $Y$ the morphisms
$X \otimes Y \to X$ and $X \otimes Y \to Y$
given by tensoring $X$ with $Y \to \I$ and
$X \to \I$ with $Y$ are the projections of a product.
When such a choice has been made, $\sV$
is said to be a \emph{cartesian monoidal category}.
Let $\sV$ and $\sV'$ be cartesian monoidal categories
and $E$ be a functor from $\sV$ to $\sV'$.
Then $E$ will be called product-preserving
if $E(\I) = \I$ and the canonical morphism from $E(X \times Y)$ to
$E(X) \times E(Y)$ is an isomorphism for every $X$ and $Y$ in $\sV$.
The functor $E$ has a symmetric monoidal structure if and only if
it is product-preserving, and when this is so the
symmetric monoidal structure is unique.

Let $k$ be a commutative ring.
Then a $k$\nd linear category is a category with small hom-sets
equipped with a structure
of $k$\nd module on each hom-set such that composition is $k$\nd bilinear.
A $k$\nd linear functor of $k$\nd linear categories is a functor which
is $k$\nd linear on hom-sets.
Recall that idempotents can be lifted along any a surjective homomorphism
$R \to R'$ of (not necessarily commutative) rings whose kernel consists of nilpotent
elements: replacing $R$ by its subring generated by any lifting of an idempotent
in $R'$, we may suppose that $R$ is commutative.
It follows that idempotent endomorphisms can be lifted along any full
$k$\nd linear functor $T$ for which every endomorphism annulled by $T$
is nilpotent.
Similarly such a $T$ reflects isomorphisms, sections and retractions.
The tensor product $\sC_1 \otimes_k \sC_2$ over $k$ of two $k$\nd linear
categories $\sC_1$ and $\sC_2$ is defined as follows.
Its set of objects is $\Ob (\sC_1) \times \Ob (\sC_2)$, and
\[
(\sC_1 \otimes_k \sC_2)((M_1,M_2),(N_1,N_2)) =
\sC_1(M_1,N_1) \otimes_k \sC_2(M_2,N_2).
\]
The identities and  composition of $\sC_1 \otimes_k \sC_2$ are defined
component-wise, so that for example the composite
of $a_1 \otimes_k a_2$ with $b_1 \otimes_k b_2$ is
$(b_1 \circ a_1) \otimes_k (b_2 \circ a_2)$.

We define a \emph{$k$\nd pretensor category} as a $k$\nd linear category
$\sC$ equipped with a structure of symmetric monoidal category
such that the tensor product is $k$\nd bilinear on hom-sets.
The tensor product of $\sC$ may then be regarded as
a $k$\nd linear functor $\sC \otimes_k \sC \to \sC$.
If $\sC$ and $\sC'$ are $k$\nd pretensor categories, then a $k$\nd tensor
functor from $\sC$ to $\sC'$ is a symmetric monoidal functor from
$\sC$ to $\sC'$ whose underlying functor is $k$\nd linear.
A monoidal isomorphism between $k$\nd tensor functors
will also be called a tensor isomorphism.
As with symmetric monoidal categories, it will be assumed unless the contrary
is stated that
\emph{the units of any $k$\nd pretensor categories are strict,
and any $k$\nd tensor functors strictly preserve them}.

A \emph{tensor ideal} $\sJ$ in a $k$\nd pretensor category $\sC$
is an assignment to every pair of objects $M$ and $N$ in $\sC$
of a $k$\nd submodule $\sJ(M,N)$ of $\sC(M,N)$,
such that the composite of any morphism in $\sJ$ on the left or right with
a morphism of $\sC$ lies in $\sJ$ and the tensor product of any morphism of
$\sJ$ on the left or right with a morphism of $\sC$,
or equivalently with an object of $\sC$, lies in $\sJ$.
Factoring out the $\sJ(M,N)$, we obtain the quotient $k$\nd pretensor
category $\sC/\sJ$ of $\sC$ by $\sJ$.
The projection $\sC \to \sC/\sJ$ onto $\sC/\sJ$ is then
a strict $k$\nd tensor functor with kernel $\sJ$.
If $\sC_0$ is a full $k$\nd pretensor of $\sC$, then any tensor ideal
$\sJ_0$ can be extended to a tensor ideal $\sJ$ of $\sC$.
The smallest such $\sJ$ is given by taking as $\sJ(M,N)$ the $k$\nd submodule
of $\sC(M,N)$ generated by the $M \to N$ of the form
$h \circ g \circ f$ with $g$ in $\sJ_0$.
A morphism $f$ of $\sC$ is said to be \emph{tensor nilpotent} if
$f^{\otimes n} = 0$ for some $n$.
The tensor nilpotent morphisms form a tensor ideal of $\sC$.
We denote by $\sC_{\mathrm{red}}$ the quotient of $\sC$ by this tensor ideal.

Let $\sC_1$ and $\sC_2$ be $k$\nd pretensor categories.
Then $\sC_1 \otimes_k \sC_2$ has a canonical structure of $k$\nd pretensor category,
with tensor product, unit, associativities and symmetries
defined component-wise.
There are canonical strict $k$\nd tensor functors
$I_1 = (-,\I)$ from $\sC_1$ to $\sC_1 \otimes_k \sC_2$ and
$I_2 = (\I,-)$ from $\sC_2$ to $\sC_1 \otimes_k \sC_2$.
Given $k$\nd tensor functors $P_1$ from $\sC_1$ to $\sD_1$ and
$P_2$ from $\sC_2$ to $\sD_2$, we define component-wise the $k$\nd tensor functor
$P_1 \otimes_k P_2$ from $\sC_1 \otimes_k \sC_2$ to $\sD_1 \otimes_k \sD_2$.
If $\sD_1 = \sD_2 = \sC$, then there is a $k$\nd tensor functor
$P:\sC_1 \otimes_k \sC_2 \to \sC$ such that $PI_1 = P_1$ and $PI_2 = P_2$.
Indeed by composing with $P_1 \otimes_k P_2$ we reduce to the case where
$\sC_1 = \sC = \sC_2$ and $P_1 = \Id_{\sC} = P_2$.
We may then take as $P$ the functor $\sC \otimes_k \sC \to \sC$ that
sends the object $(M_1,M_2)$ to $M_1 \otimes M_2$
and the morphism $f_1 \otimes_k f_2$ to $f_1 \otimes f_2$,
with structural isomorphisms defined using the appropriate
symmetries and associativities in $\sC$.
Given also $P':\sC_1 \otimes_k \sC_2 \to \sC$ with $P'I_1 = P_1$
and $P'I_2 = P_2$, there is a unique tensor isomorphism $\varphi$
from $P$ to $P'$ with $\varphi I_1$ and $\varphi I_2$ the identity.

An \emph{algebra} in a $k$\nd pretensor category $\sC$ is an object
$R$ in $\sC$ together with morphisms $\iota:\I \to R$, the unit, and
$\mu:R \otimes R \to R$, the multiplication, such that $\mu$ satisfies the usual
associativity condition, and $\iota$ is a left an right identity for $\mu$.
If $R$ and $R'$ are algebras in $\sC$ with units $\iota$ and $\iota'$ and
multiplications $\mu$ and $\mu'$, their tensor product $R \otimes R'$ has a
structure of algebra, with unit $\iota \otimes \iota'$ and multiplication the composite
of the appropriate symmetries and associativities with $\mu \otimes \mu'$.
The algebra $R$ is said to be \emph{commutative} if composing the symmetry
interchanging the two factors in $R \otimes R$ with the multiplication
$\mu$ leaves $\mu$ unchanged.

A $k$\nd linear category is said to be \emph{pseudo-abelian} if it has a zero object
and direct sums, and if every idempotent endomorphism has an image.
We define a \emph{$k$\nd tensor category} as a $k$\nd pretensor category whose
underlying $k$\nd linear category is pseudo-abelian.

Given a commutative monoid $Z$,
a \emph{$Z$\nd graded $k$\nd tensor category} is a $k$\nd tensor category $\sC$
together with a strictly full subcategory $\sC_i$ of $\sC$ for each $i \in Z$,
such that every object of $\sC$ is a coproduct of objects in the $\sC_i$,
and such that $\I$ lies in $\sC_0$, the tensor product of an object in
$\sC_i$ with an object in $\sC_j$ lies in $\sC_{i+j}$,
and $\sC(M,N) = 0$ for $M$ in $\sC_i$ and $N$ in $\sC_j$ with $i \ne j$.
To any $\Z/2$\nd graded $k$\nd tensor category $\sC$ there is associated
a $k$\nd tensor category $\sC^\dagger$, obtained by modifying the symmetry of $\sC$
according to the $\Z/2$\nd grading.
The underlying $k$\nd linear category, tensor product and associativities
of $\sC^\dagger$ are the same as those of $\sC$, but the symmetry
$M \otimes N \iso N \otimes M$ in $\sC^\dagger$ is given by multiplying that in $\sC$
by $(-1)^{ij}$ when $M$ lies in $\sC_i$ and $N$ in $\sC_j$,
and then extending by linearity.

\subsection{Pseudo-abelian hulls and ind-completions}\label{ss:ind}

A \emph{pseudo-abelian hull} of a $k$\nd linear category $\sC$ is a pseudo-abelian
$k$\nd linear category $\sC'$ containing $\sC$ as a full subcategory,
such that every object of $\sC'$ is a direct summand of a direct sum of
objects in $\sC$.
A pseudo-abelian hull exists for any $\sC$ and is unique up to $k$\nd linear
equivalence.
We have in fact the following more precise form of the uniqueness.
Let $\sC$ be a full subcategory of a $k$\nd linear category $\sC'$
such that every object of $\sC'$ is a direct summand of a direct sum of
objects in $\sC$.
Then given two $k$\nd linear functors $\sC' \to \sC''$, any natural transformation
between their restrictions to $\sC$
extends uniquely to a natural transformation between the functors themselves.
Further if $\sC''$ is pseudo-abelian, then any $k$\nd linear functor $\sC \to \sC''$
extends to a $k$\nd linear functor $\sC' \to \sC''$.

A pseudo-abelian hull of a $k$\nd pretensor category $\sC$ is a
$k$\nd pretensor category $\sC'$
containing $\sC$ as a full $k$\nd pretensor subcategory,
such that the underlying $k$\nd linear category of $\sC'$
is a pseudo-abelian hull of the
underlying $k$\nd linear category of $\sC$.
To see that it exists, start with a pseudo-abelian hull $\sC'$ of the
underlying $k$\nd linear category of $\sC$.
Then every object of $\sC' \otimes_k \sC'$ is a direct summand of a
direct sum of objects of $\sC \otimes_k \sC$.
The tensor product $\otimes$ of $\sC$ composed with the embedding $\sC \to \sC'$
thus extends to a $k$\nd linear functor $\otimes'$ from $\sC' \otimes_k \sC'$
to $\sC'$, and we may assume that $\I \otimes' -$ and $- \otimes' \I$ are the
identity of $\sC'$.
The associativities and symmetries for $\otimes$ then extend uniquely
to associativities and symmetries for $\otimes'$, and the required compatibilities
follow from those for $\otimes$.
A similar argument shows that if $\sC'$ is a pseudo-abelian hull of the
$k$\nd pretensor category $\sC$,
then any $k$\nd tensor functor $\sC \to \sC''$ with $\sC''$ pseudo-abelian
extends uniquely up to tensor isomorphism to a $k$\nd tensor functor
$\sC' \to \sC''$.
In particular a pseudo-abelian hull of a $k$\nd pretensor category is unique up to
$k$\nd tensor equivalence.

Let $\sC$ be a category with small hom-sets.
By a filtered system in $\sC$ we mean a functor to $\sC$ from a small filtered
category.
An \emph{ind-completion} of $\sC$ is a
category $\sC'$ with small hom-sets containing $\sC$ as a full subcategory,
such that the colimit
of every filtered system in $\sC$ exists in $\sC'$,
every object of $\sC'$ is the colimit a filtered system in $\sC$,
and $\sC'(M,-)$ preserves colimits of filtered systems in $\sC$
for every $M$ in $\sC$.
The category $\Ind(\sC)$ of ind-objects of $\sC$ (\cite{SGA4-1}, 8.2.4, 8.2.5),
where $\sC$ is identified with the full subcategory of constant ind-objects, is
an ind-completion of $\sC$.
An ind-completion is unique up to equivalence.
We have in fact the following more precise form of the uniqueness.
Let $\sC$ be a full subcategory of a category $\sC'$ with small hom-sets,
such that every object of $\sC'$ is the colimit of a filtered system in $\sC$
and $\sC'(M,-)$ preserves colimits of filtered systems in $\sC$ for every
$M$ in $\sC$.
Then given two functors $\sC' \to \sC''$ which preserve colimits of filtered systems
in $\sC$,
any natural transformation  between their restriction to $\sC$
extends uniquely to a natural transformation between the functors themselves.
Further any functor $\sC \to \sC''$ for which the image in $\sC''$ of every
filtered system in $\sC$ has a colimit
extends to a functor $\sC' \to \sC''$ which preserves colimits of filtered systems
in $\sC$.
Also if $\sC$ has a $k$\nd linear structure it extends uniquely to $\sC'$,
and a functor from $\sC'$ to a $k$\nd linear
category which preserves colimits of filtered systems in $\sC$
is $k$\nd linear when its restriction to $\sC$ is.

An ind-completion of a $k$\nd pretensor category $\sC$ is a $k$\nd pretensor
category $\sC'$ containing $\sC$ as a full $k$\nd pretensor subcategory,
such that the underlying category of $\sC'$ is an ind-completion of the
underlying category of $\sC$, and such that the tensor product with any object of $\sC'$
preserves colimits of filtered systems in $\sC$.
To see that it exists, start with an ind-completion $\sC'$ of the
underlying category of $\sC$.
Then $\sC' \times \sC'$ is an ind-completion of the category $\sC \times \sC$.
Now define the tensor product
$\otimes':\sC' \times \sC' \to \sC'$, associativities and symmetries of $\sC'$
in a similar way to the case of a pseudo-abelian hull
of a $k$\nd pretensor category above.
When $\sC'$ is equipped with the unique $k$\nd linear structure extending that
of $\sC$, it is then easily checked that $\otimes'$ is $k$\nd bilinear on hom-spaces.
A similar argument shows that if $\sC'$ is an ind-completion of the
$k$\nd pretensor category $\sC$,
then any $k$\nd tensor functor $\sC \to \sC''$ for which the image in $\sC''$
of every filtered system in $\sC$ has a colimit
extends uniquely up to tensor isomorphism to a $k$\nd tensor functor
$\sC' \to \sC''$ which preserves colimits of filtered systems in $\sC$.
In particular an ind-completion of a $k$\nd pretensor category is unique up to
$k$\nd tensor equivalence.

Let $\sC$ be a $k$\nd linear category
and $\widehat{\sC}$ be an ind-completion of $\sC$,
equipped with the unique $k$\nd linear structure extending that of $\sC$.
Then applying $(-)^\mathrm{op}$ shows that restriction from $\widehat{\sC}^\mathrm{op}$ to
$\sC^\mathrm{op}$ defines an equivalence between the category $\sH$ of those
$k$\nd linear functors from $\widehat{\sC}^\mathrm{op}$ to $k$\nd modules which preserve the limit
of $\sI^\mathrm{op} \to \sC^\mathrm{op}$ for every filtered system $\sI \to \sC$ in $\sC$
and the category of all $k$\nd linear functors from $\sC^\mathrm{op}$ to $k$\nd modules.
Suppose now that $\sC$ is essentially small and semisimple abelian.
Then every $H$ in $\sH$ is representable.
Indeed define as follows an object $L$
of $\widehat{\sC}$ and an element $l$ of $H(L)$ which
represent $H$.
Choose a small set $\sC_0$ of representatives for isomorphism classes of
objects of $\sC$.
Then we have a small category $\sP$ whose objects are pairs
$(M,m)$ with $M \in \sC_0$ and $m \in H(M)$, where a morphism
$(M,m) \to (M',m')$ is a morphism $f:M \to M'$ in $\sC$ such that
$H(f)$ sends $m'$ to $m$.
There is an evident forgetful functor $\sP \to \sC$.
Since $\sC$ is semisimple abelian, the category $\sC^\mathrm{op}$ has finite
limits, and the restriction of $H$ to
$\sC^\mathrm{op}$ preserves them.
Thus given a finite diagram $\sD \to \sP^\mathrm{op}$, any limiting cone in
$\sC^\mathrm{op}$ with base
$\sD \to \sP^\mathrm{op} \to \sC^\mathrm{op}$
and vertex in $\sC_0$ lifts uniquely to a cone in $\sP^\mathrm{op}$  with base
$\sD \to \sP^\mathrm{op}$, as follows by taking the limit of the diagram of
$k$\nd modules obtained by composing $\sD \to \sC^\mathrm{op}$ with $H$.
Hence $\sP$ is filtered.
Now take as $L$ the colimit $\colim_{(M,m) \in \sP} M$ in $\widehat{\sC}$,
and as $l$ the element of $H(L) = \lim_{(M,m) \in \sP} H(M)$
with component $m$ at $(M,m) \in \sP$.
Then $l$ defines an isomorphism $\widehat{\sC}(M,L) \to H(M)$
for every $M$ in $\sC_0$ and hence for every $M$ in $\widehat{\sC}$.

\subsection{Rigid categories}\label{ss:rigid}

Let $\sC$ be a $k$\nd pretensor category.
A \emph{duality pairing} in $\sC$ is a quadruple $(M,M^\vee,\eta,\varepsilon)$,
consisting of objects $M$ and $M^\vee$ of $\sC$ and a unit
$\eta:\I \to M^\vee \otimes M$ and a counit $\varepsilon:M \otimes M^\vee \to \I$
satisfying triangular identities analogous to those for an adjunction (\cite{Mac}, p.~85).
Explicitly,
$(\varepsilon \otimes M) \circ \alpha^{-1}  \circ (M \otimes \eta)$
is required to be the identity of $M$ and
$(M^\vee \otimes \varepsilon) \circ \widetilde{\alpha}
\circ (\eta \otimes M^\vee)$ the identity
of $M^\vee$, where $\alpha$ and $\widetilde{\alpha}$ are the associativities.
When such a duality pairing exists, $M$ is said to be \emph{dualisable},
and $(M,M^\vee,\eta,\varepsilon)$ to be a duality pairing for $M$ and $M^\vee$ to be a
\emph{dual} for $M$.

Given duality pairings $(M,M^\vee,\eta,\varepsilon)$ for $M$ and
$(M',M'{}^\vee,\eta',\varepsilon')$ for $M'$,
any morphism $f:M \to M'$ has a transpose
$f^\vee:M'{}^\vee \to M^\vee$, characterised by the condition
\begin{equation}\label{e:transpose}
\varepsilon \circ (M \otimes f^\vee) =
\varepsilon' \circ (f \otimes M'{}^\vee),
\end{equation}
or by a similar condition using $\eta$ and $\eta'$.
Explicitly, $f^\vee$ is the composite of $\eta \otimes M'{}^\vee$, the appropriate
associativity, $M^\vee \otimes (f \otimes M'{}^\vee)$,
and $M^\vee \otimes \varepsilon'$.
It follows in particular from \eqref{e:transpose} and the similar condition
using the units that a duality pairing for $M$
is determined uniquely up to unique isomorphism, and that the unit of
a duality pairing is determined uniquely when the counit is given,
and conversely.
Also $(1_M)^\vee = 1_{M^\vee}$ and \mbox{$(f' \circ f)^\vee = f^\vee \circ f'{}^\vee$}.

The identity duality pairing in $\sC$ is $(\I,\I,1_\I,1_\I)$.
The tensor product of the duality pairings $(M',M'{}^\vee,\eta',\varepsilon')$ and
$(M'',M''{}^\vee,\eta'',\varepsilon'')$ is
$(M' \otimes M'',M'{}^\vee \otimes M''{}^\vee,\eta,\varepsilon)$ with $\eta$
and $\varepsilon$ obtained from $\eta \otimes \eta'$ and $\varepsilon \otimes \varepsilon'$
by composing with the appropriate symmetries.
The dual of the duality pairing $(M,M^\vee,\eta,\varepsilon)$ is
$(M^\vee,M,\widetilde{\eta},\widetilde{\varepsilon})$ with $\widetilde{\eta}$ and
$\widetilde{\varepsilon}$ obtained from $\eta$ and $\varepsilon$ by composing with the
appropriate symmetries.
The set of dualisable objects of $\sC$ is closed under the formation of
tensor products and duals, and when $\sC$ is pseudo-abelian
it is closed under the formation of direct sums and direct summands.
Any $k$\nd tensor functor sends dualisable objects to dualisable objects.
If $M$ is dualisable in $\sC$, then the component at $M$ of any monoidal natural
transformation $\varphi$ of $k$\nd tensor functors with source $\sC$ is an isomorphism,
with inverse the transpose of $\varphi_{M^\vee}$.

The \emph{trace} $\tr(f)$ of an endomorphism $f$
of a dualisable object $M$ of $\sC$
is defined as the endomorphism  $\varepsilon \circ (f \otimes M^\vee) \circ \widetilde{\eta}$
of $\I$,
where $\varepsilon$ is the counit of a duality pairing for $M$
and $\widetilde{\eta}$ is the unit of the dual pairing.
The trace does not depend on the duality pairing chosen for $M$, and it is preserved
by $k$\nd tensor functors.
We have $\tr(f^\vee) = \tr(f)$, and if $f'$ is an endomorphism of a dualisable
object then $\tr(f \otimes f') = \tr(f)\tr(f')$.
The \emph{rank} of $M$ is defined as $\tr(1_M)$.
More generally the
\emph{contraction} $N \to N'$ of a morphism $f:N \otimes M \to N' \otimes M$
with respect to $M$ is defined as the composite, modulo the appropriate associativities,
of $N \otimes \widetilde{\eta}$ with $f \otimes M^\vee$ and
$N' \otimes \varepsilon$.
Again it does not depend on the duality chosen for $M$.
For any $j:M \to M$,
the contraction of $l \otimes j$ with respect to $M$ is $\tr(j)l$.
If $M$ is the direct sum of $M_1, M_2, \dots, M_r$, then the contraction
of $N \otimes M \to N' \otimes M$ with respect to $M$ is the sum of
the contractions of the diagonal components $N \otimes M_i \to N' \otimes M_i$
with respect to the $M_i$.

If $(M,M^\vee,\eta,\varepsilon)$ is a duality pairing for $M$ in $\sC$,
then the composite of $g:N \to M$ and $h:M \to L$ is given by the composite
of $N \circ \eta$, the appropriate associativity, $(g \otimes M^\vee) \otimes h$, and
$\varepsilon \otimes L$, or equivalently by the contraction with respect to $M$
of $g \otimes h$ composed with the appropriate symmetry.
To verify this, we reduce to the case where $\sC$ is strict by applying a strict
$k$\nd tensor functor $T:\sC \to \sC'$ with $T$ an equivalence and $\sC'$ strict
and replacing $\sC$ by $\sC'$.
It then suffices to use the triangular identity
$(\varepsilon \otimes M) \circ (M \otimes \eta) = 1_M$.
It follows that if $N$ is also dualisable then $\tr(g \circ h) = \tr(h \circ g)$.
It also follows by induction on $n$ that the composite of $n$ morphisms between dualisable
objects lies in the tensor ideal generated by their tensor product.
A tensor nilpotent endomorphism of a dualisable object
is thus nilpotent.

Given objects $L$, $M$ and $N$ in $\sC$,
and a duality pairing $(M,M^\vee,\eta,\varepsilon)$ for $M$,
we have a canonical isomorphism
\begin{equation}\label{e:adjiso}
\omega_{M,\varepsilon;N,L}:\sC(N,M^\vee \otimes L) \iso
\sC(M \otimes N,L)
\end{equation}
which sends $f:N \to M^\vee \otimes L$ to the composite
\[
M \otimes N \xrightarrow{M \otimes f} M \otimes (M^\vee \otimes L) \iso
(M \otimes M^\vee) \otimes L \xrightarrow{\varepsilon \otimes L} L.
\]
Its inverse sends $g:M \otimes N \to L$ to the composite of $\eta \otimes N$,
the appropriate associativity, and $M^\vee \otimes g$.
This can be checked using the triangular identities, after first reducing to the
case where $\sC$ is strict.
When $N = \I$, we write \eqref{e:adjiso} as
\begin{equation}\label{e:adj1iso}
\omega_{M,\varepsilon;L}:\sC(\I,M^\vee \otimes L)
\iso \sC(M,L).
\end{equation}
If $f:\I \to M^{\prime \vee} \otimes M$ is a morphism in $\sC$ and
$\varepsilon:M \otimes M^\vee \to \I$ is the counit of a duality pairing for $M$ and
$\widetilde{\varepsilon}:M^\vee \otimes M \to \I$ is the counit of its dual, then
\begin{equation}\label{e:adjtrans}
\omega_{M',\varepsilon';M}(f)^\vee =
\omega_{M^\vee,\widetilde{\varepsilon};M^{\prime \vee}}(\sigma \circ f),
\end{equation}
where the transpose is defined using
$\varepsilon$ and $\varepsilon'$, and $\sigma$ is the symmetry
interchanging $M^{\prime \vee}$ and $M$.
Given $f':\I \to M^{\prime \vee} \otimes L'$ and
$f'':\I \to M^{\prime\prime \vee} \otimes L''$, we have
\begin{equation}\label{e:adjtens}
\omega_{M',\varepsilon';L'}(f') \otimes \omega_{M'',\varepsilon'',L''}(f'') =
\omega_{M' \otimes M'',\varepsilon;L' \otimes L''}(\rho \circ (f' \otimes f'')),
\end{equation}
where $\varepsilon:M^\vee \otimes M \to \I$ is the counit of the tensor product of the duality pairings
with counits $\varepsilon':M'{}^\vee \otimes M' \to \I$ and
$\varepsilon'':M''{}^\vee \otimes M'' \to \I$,
and $\rho$ is the symmetry that interchanges the middle two factors
of $(M^{\prime \vee} \otimes L') \otimes (M^{\prime\prime \vee} \otimes L'')$.
Given $f:\I \to M^\vee \otimes M'$ and $f':\I \to M^{\prime \vee} \otimes M''$, we have
\begin{equation}\label{e:adjtenscomp}
\omega_{M',\varepsilon';M''}(f') \circ \omega_{M,\varepsilon;M'}(f) =
\omega_{M,\varepsilon;M''}((M^\vee \otimes (\varepsilon' \otimes M'')) \circ \lambda
                    \circ (f \otimes f')),
\end{equation}
where
$\lambda:(M^\vee \otimes M') \otimes (M^{\prime \vee} \otimes M'') \iso
M^\vee \otimes ((M' \otimes M^{\prime \vee}) \otimes M'')$ is the associativity.
To verify \eqref{e:adjtenscomp}, we may suppose that $\sC$ is strict.
Then the left hand side of \eqref{e:adjtenscomp} is the composite of
$M \otimes \eta'$ with
$\omega_{M,\varepsilon;M'}(f) \otimes M'{}^\vee \otimes \omega_{M',\varepsilon;M''}(f')$
and $\varepsilon' \otimes M''{}^\vee$, where $\eta':\I \to M'{}^\vee \otimes M'$
is the unit,
and it suffices to use the triangular identity
$(\varepsilon' \otimes M') \circ (M' \otimes \eta') = 1_{M'}$.
Given $v:M \to M'$ and duality pairings for $M$ and $M'$ with respective counits
$\varepsilon:M^\vee \otimes M \to \I$ and $\varepsilon':M'{}^\vee \otimes M' \to \I$, we have
\begin{equation}\label{e:adjgraph}
v = \omega_{M,\varepsilon;M'}((v \otimes M'{}^\vee )^\vee \circ \varepsilon'{}^\vee),
\end{equation}
where the transposes are taken using the counits defined by
$\varepsilon$ and $\varepsilon'$.
Indeed
\[
f^\vee = \varepsilon' \circ (\omega_{M,\varepsilon;M'}(f) \otimes M'{}^\vee)
\]
for any $f:\I \to M^\vee \otimes M'$, by the definition of the counits for the dual and
tensor product.
Taking the transpose and applying $\omega_{M,\varepsilon;M'}$ now gives \eqref{e:adjgraph},
because $v =\omega_{M,\varepsilon;M'}(f)$ for some $f$.

A $k$\nd pretensor category is said to be \emph{rigid} if each
of its objects is dualisable.
In a rigid $k$\nd pretensor category, every tensor nilpotent endomorphism is
nilpotent.
In view of the isomorphism \eqref{e:adj1iso}, it is enough for many purposes
to consider only those hom-spaces in a rigid $k$\nd pretensor category $\sC$ of the form
$\sC(\I,N)$, or alternatively $\sC(N,\I)$.
For example the assignment $\sJ \mapsto \sJ(\I,-)$ defines a bijection between tensor
ideals of $\sC$ and $k$\nd linear subfunctors of $\sC(\I,-)$ from $\sC$ to
$k$\nd modules.

Let $\sC$ be a rigid $k$\nd pretensor category
with $\End(\I)$ a field.
Then $\sC$ has a unique maximal tensor ideal $\Rad(\sC)$,
corresponding to the $k$\nd linear subfunctor of $\sC(\I,-)$ that sends
$N$ to the set of those morphisms $\I \to N$ which are not sections.
The morphisms $f:M \to N$ of $\sC$ in $\Rad(\sC)$ are those for which
$\tr(f \circ g) = 0$ for every $g:N \to M$.
We write $\overline{\sC}$ for $\sC/\Rad(\sC)$.

Let $\sD$ be a rigid $k$\nd pretensor category.
If $T:\sD \to \sD'$ is a $k$\nd tensor functor such that
$T_{\I,N}:\sD(\I,N) \to \sD'(\I,T(N))$ is surjective
for every $N$ in $\sD$,
then $T$ is full.
To see this we reduce after factoring $T$ as a strict $k$\nd tensor
functor composed with a fully faithful $k$\nd tensor functor to the
case where $T$ is strict.
If $M$ is an object of $\sD$ and $\varepsilon:M \otimes M^\vee \to \I$ is the counit of
a duality pairing for $M$, then $T(\varepsilon)$ is the counit of
a duality pairing for $T(M)$.
For any $f:\I \to M^\vee \otimes N$ we have
\[
\omega_{T(M),T(\varepsilon);T(N)}(T(f)) = T(\omega_{M,\varepsilon;N}(f)).
\]
Since the omega are isomorphisms, $T_{M,N}$ is thus surjective
when $T_{\I,N \otimes M^\vee}$ is surjective.
That $T$ is faithful when the $T_{\I,N}$ can be seen in the same way,
or deduced from the remark on tensor ideals above.
Similarly if $T_1$ and $T_2$ are strict $k$\nd tensor
functors $\sD \to \sD'$ which coincide on objects of $\sD$ and on
hom-spaces of $\sD$ of the form $\sD(\I,N)$ for $N$ in $\sD$,
then $T_1 = T_2$, because $T_1(\eta) = T_2(\eta)$ for every unit $\eta$
implies $T_1(\varepsilon) = T_2(\varepsilon)$ for every counit $\varepsilon$.
Replacing $\sD$ and $\sD'$ by $\sD^\mathrm{op}$ and $\sD'{}^\mathrm{op}$
shows that similar results hold with $\sD(N,\I)$ instead of
$\sD(\I,N)$.

Let $\sC$ be a $k$\nd pretensor category
and $R$ be commutative algebra in $\sC$.
An $R$\nd module in $\sC$ is an object $M$ of $\sC$ together with a morphism
$R \otimes M \to M$, or equivalently $M \otimes R \to M$, which satisfies the
usual unit and associativity conditions.
The tensor product $M \otimes_R N$ of the $R$\nd modules $M$ and $N$ is defined
as the coequaliser of the two morphisms
\[
R \otimes (M \otimes N) \to M \otimes N
\]
obtained using the actions of $R$ on $M$ and on $N$, provided that this coequaliser
exists and is preserved by tensor product with any object of $\sC$.
If either $M$ or $N$ is free, i.e.\ isomorphic to an $R$\nd module $R \otimes L$
for some object $L$ of $\sC$, then  $M \otimes_R N$ exists:
when $N = R \otimes L$ for example it is given by $M \otimes L$.
Thus we obtain a $k$\nd pretensor category of free $R$\nd modules in $\sC$
with tensor product $\otimes_R$,
and a $k$\nd tensor functor $R \otimes -$ to it from $\sC$.

More generally, suppose that $\sC$ is a full $k$\nd pretensor subcategory of a
$k$\nd pretensor category $\widehat{\sC}$, and let $R$ be a commutative algebra
in $\widehat{\sC}$.
Then we have a $k$\nd pretensor category of free $R$\nd modules in $\widehat{\sC}$
on objects of $\sC$, consisting of the $R$\nd modules isomorphic to $R \otimes M$
with $M$ in $\sC$.
It will be convenient for what follows to consider the $k$\nd tensor equivalent
category $\sF_R$, obtained by factoring the $k$\nd tensor functor $R \otimes -$
from $\sC$ to free $R$\nd modules as a strict $k$\nd tensor functor
$F_R:\sC \to \sF_R$ which is bijective on objects
followed by a fully faithful $k$\nd tensor functor.
Explicitly, each object of $\sF_R$ can be written uniquely in the
form $F_R(M)$ with $M$ an object of $\sC$, and a morphism from $F_R(M)$
to $F_R(N)$ is a morphism of $R$\nd modules from $R \otimes M$ to $R \otimes N$,
with composition that of morphisms of $R$\nd modules.
We then have a canonical isomorphism
\begin{equation}\label{e:frmodhom}
\theta_{R;M,N}:\sF_R(F_R(M),F_R(N)) \iso \widehat{\sC}(M, R \otimes N)
\end{equation}
for every $M$ and $N$ in $\sC$, given by composing with $\iota \otimes M$,
where $\iota:\I \to R$ is the unit of $R$.
The functor $F_R$ sends $f:M \to N$ in $\sC$ to
$R \otimes f$.
The tensor product of $a:F_R(M) \to F_R(N)$ and $a':F_R(M') \to F_R(N')$
has image under $\theta_{R;M \otimes M',N \otimes N'}$
the composite of $\theta_{R;M,N}(a) \otimes \theta_{R;M',N'}(a')$ with the
symmetry interchanging the middle two factors of its target and
$\mu \otimes N \otimes N'$, where $\mu$ is the multiplication $R \otimes R \to R$.
In what follows, $\widehat{\sC}$ will always be an ind-completion of $\sC$.

Let $u:R \to R'$ be a morphism of commutative algebras in $\widehat{\sC}$.
Then we have a strict $k$\nd tensor functor $\sF_u:\sF_R \to \sF_{R'}$
which sends the object $F_R(M)$ of $\sF_R$ to $F_{R'}(M)$ and the
morphism $a:F_R(M) \to F_R(N)$ in $\sF_R$ to $a':F_{R'}(M) \to F_{R'}(N)$
with
\[
\theta_{R';M,N}(a') =(u \otimes N) \circ \theta_{R;M,N}(a).
\]
It is clear that $F_{R'} = \sF_u \circ F_R$.
Also $\sF_{v \circ u} = \sF_v \sF_u$, and $\sF_{1_R}$ is the identity of $\sF_R$.

\begin{lem}\label{l:frmod}
Let $\sC$ be a rigid $k$\nd pretensor category
and $\widehat{\sC}$ be an ind-completion of $\sC$.
\begin{enumerate}
\item\label{i:frmodes}
If $\sC$ is essentially small and semisimple abelian,
then for every $k$\nd tensor functor  $T:\sC \to \sB$ there exists
a commutative algebra $R$ in $\widehat{\sC}$ and a fully faithful
$k$\nd tensor functor $I:\sF_R \to \sB$ such that $T = IF_R$.
\item\label{i:frmodff}
If $R$ and $R'$ are commutative algebras in $\widehat{\sC}$,
then for every $k$\nd tensor functor $U:\sF_R \to \sF_{R'}$ with
$F_{R'} = UF_R$ there exists a unique morphism of algebras
$u:R \to R'$ such that $U = \sF_u$.
\end{enumerate}
\end{lem}

\begin{proof}
\ref{i:frmodes}
Let $\widehat{\sB}$ be an ind-completion of $\sB$.
Since the functors $F_R$ are bijective on objects, any factorisation of the composite $T_1$ of $T$
with the embedding $\sB \to \widehat{\sB}$ as $I_1F_R$ with $I_1$ fully faithful
gives at the same time a factorisation $IF_R$ of $T$
with $I$ fully faithful.
Hence we may suppose, by replacing $\sB$ by $\widehat{\sB}$ and $T$ by $T_1$,
that $T$ extends to a $k$\nd tensor functor $\widehat{T}:\widehat{\sC} \to \sB$
which preserves colimits of filtered systems in $\sC$.
Factoring $\widehat{T}$ as $T'T''$ with $T':\widehat{\sC} \to \sB'$ strict
and $T'':\sB' \to \sB$ fully faithful,
we may further suppose after replacing $\sB$ by $\sB'$, $\widehat{T}$ by $T'$,
and $T$ by the restriction of $T'$ to $\sC$, that $\widehat{T}$ is strict.
The $k$\nd linear functor
$\sB(\widehat{T}(-),\I)$ from
$\widehat{\sC}^\mathrm{op}$ to $k$\nd modules preserves the limit in $\widehat{\sC}^\mathrm{op}$ of
$\sI^\mathrm{op} \to \sC^\mathrm{op}$ for every filtered system $\sI \to \sC$.
Hence as was seen in \ref{ss:ind} this functor is representable,
because $\sC$ is essentially small and semisimple abelian.
Thus there is an object $R$ of $\widehat{\sC}$ and a morphism
$a:\widehat{T}(R) \to \I$ in $\sB$ such that for each object $M$ of
$\widehat{\sC}$ we obtain an isomorphism
\[
\zeta_M:\widehat{\sC}(M,R) \iso \sB(\widehat{T}(M),\I)
\]
by applying $\widehat{T}$ and then composing with $a$.
The object $R$ has a unique structure of commutative algebra in $\widehat{\sC}$
such that $a$ is a morphism of algebras from $\widehat{T}(R)$ to $\I$
in $\sB$.
Indeed the image  of the multiplication $R \otimes R \to R$ under
$\zeta_{R \otimes R}$ is $a \otimes a$ and the image of the unit
$\I \to R$ under $\zeta_\I$ is $1_\I$.

Given a morphism $f$ from $F_R(M)$ to $F_R(N)$, write
\[
\widetilde{f} = (a \otimes \widehat{T}(N)) \circ \widehat{T}(\theta_{R;M,N}(f)):
\widehat{T}(M) \to \widehat{T}(N),
\]
where $\theta_{R;M,N}$ is the isomorphism \eqref{e:frmodhom}.
Then we have a strict $k$\nd tensor functor $I:\sF_R \to \sB$ with $T = IF_R$, which
sends $F_R(M)$ to $T(M) = \widehat{T}(M)$ and $f$ to $\widetilde{f}$.
That $I$ strictly preserves tensor products of morphisms follows from the fact that
$\widehat{T}$ does and that $a$ is a morphism of algebras.
To see that $I$ preserves composites and identities, and that $T = IF_R$,
note that $\widetilde{f}$ is the unique morphism for which
\[
\widetilde{f} \circ (a \otimes \widehat{T}(M)) =
(a \otimes \widehat{T}(N)) \circ \widehat{T}(f),
\]
where $f$ on the right hand side is regarded as a morphism from $R \otimes M$
to $R \otimes N$ in $\widehat{\sC}$.
Indeed if $\widehat{T}(N)$ is regarded as a
$\widehat{T}(R)$\nd module in $\sB$ by means of $a:\widehat{T}(R) \to \I$,
then both sides are morphisms $\widehat{T}(R) \otimes \widehat{T}(M) \to \widehat{T}(N)$
of $\widehat{T}(R)$\nd modules whose composites with the tensor product
of the unit $\I \to \widehat{T}(R)$ and $\widehat{T}(M)$ coincide.
That $I$ strictly preserves associativities and symmetries is clear
because $T = IF_R$ with $F_R$ bijective on objects and $F_R$ and $T$ strict.

For every $M$ the homomorphism $I_{F_R(M),F_R(\I)}$ of hom-spaces
is an isomorphism,
because it is the composite of the isomorphism
$\theta_{R;M,\I}$ with $\zeta_M$.
Since $\sF_R$ is rigid, $I$ is thus fully faithful.

\ref{i:frmodff}
Since $F_{R'} = UF_R$ with $F_R$ bijective on objects and $F_R$ and $F_{R'}$ strict,
$U$ is a strict $k$\nd tensor functor.
The $k$\nd homomorphisms
\begin{equation}\label{e:xidef}
\xi_M = \theta_{R';M,\I} \circ U_{F_R(M),F_R(\I)} \circ (\theta_{R;M,\I})^{-1}
\end{equation}
for $M$ in $\sC$ are the components of a natural transformation $\xi$
from $\widehat{\sC}(-,R)$ to $\widehat{\sC}(-,R')$ on $\sC^\mathrm{op}$.
As was seen in \ref{ss:ind}, such a $\xi$ extends uniquely to a natural
transformation of functors on $\widehat{\sC}^\mathrm{op}$,
because the hom-functors on $\widehat{\sC}^\mathrm{op}$ preserve limits.
Thus $\xi = \widehat{\sC}(-,u)$ for a unique morphism $u:R \to R'$
in $\widehat{\sC}$.
Considering the image of $\theta_{R;\I,\I}(1)$ under $\xi_\I$
shows that $u$ sends the identity of $R$ to the identity of $R'$.
Write $\mu$ and $\mu'$ for the multiplications of $R$ and $R'$.
If $b:M \to R$ and $c:N \to R$ are morphisms in $\widehat{\sC}$ with $M$
and $N$ in $\sC$, then since $U$ strictly preserves tensor products of morphisms
we have
\[
\xi_{M \otimes N}(\mu \circ (b \otimes c)) = \mu' \circ (\xi_M(b) \otimes \xi_N(c)).
\]
It follows that the composites of $b \otimes c$
with $u \circ \mu$ and with $\mu' \circ (u \otimes u)$ coincide.
Writing $R \otimes R$ as the filtered colimit of objects $M \otimes N$
thus shows that $u \circ \mu$ and $\mu' \circ (u \otimes u)$ coincide, so that $u$ is a
morphism of algebras.

Since $\xi = \widehat{\sC}(-,u)$, the equality
\eqref{e:xidef} holds for every $M$ with $U$ replaced by $\sF_u$.
The strict $k$\nd tensor functors $U$ and $\sF_u$,
which coincide on objects of $\sF_R$, thus also coincide on hom-spaces
of the form $\sF_R(F_R(M),F_R(\I))$.
Since $\sF_R$ is rigid, it follows that $U = \sF_u$.
If $U = \sF_{u'}$ for a morphism of algebras $u':R \to R'$,
then $\xi = \widehat{\sC}(-,u')$, so that $u' = u$.
\end{proof}

\subsection{Free rigid categories}\label{ss:frrigid}

Let $(\alpha_\lambda)_{\lambda \in \Lambda}$ be a family of elements of $k$.
A \emph{free rigid $k$\nd tensor category of type
$(\alpha_\lambda)_{\lambda \in \Lambda}$} is a $k$\nd tensor
category $\sR$ together with a family $(N_\lambda)_{\lambda \in \Lambda}$
of objects in $\sR$ with $N_\lambda$ dualisable of rank $\alpha_\lambda$
(i.e.\ the rank of $M_\lambda$ is the image of $\alpha_\lambda$
under $k \to \End(\I)$), such that the
following universal property holds:
given a $k$\nd tensor category $\sC$, every family
$(M_\lambda)_{\lambda \in \Lambda}$ in $\sC$ with $M_\lambda$
dualisable of rank $\alpha_\lambda$ is the image of
$(N_\lambda)_{\lambda \in \Lambda}$ under some $k$\nd tensor functor $\sR \to \sC$,
and for every pair of $k$\nd tensor functors $T,T':\sR \to \sC$ and family
$(\varphi_\lambda)_{\lambda \in \Lambda}$ of isomorphisms
$\varphi_\lambda:T(N_\lambda) \iso T'(N_\lambda)$ there exists a unique tensor isomorphism
$\varphi:T \iso T'$ with $\varphi_{N_\lambda} = \varphi_\lambda$ for each $\lambda$.
Such an $\sR$, equipped with its universal family $(N_\lambda)_{\lambda \in \Lambda}$,
is unique up to $k$\nd tensor equivalence.
Since the identity $\sR \to \sR$
factors up to tensor isomorphism through any full $k$\nd tensor subcategory of $\sR$
containing the $N_\lambda$ and their duals,
$\sR$ is rigid, and generated as a rigid
$k$\nd tensor category by the $N_\lambda$.

Results such as Proposition~\ref{p:freerigid} below are well known.
We include a proof here because there seems to be no convenient reference giving
the form required.
For the proof we need the following notion.
A strict rigid $k$\nd pretensor category is a strict $k$\nd pretensor
category together with an assignment to each of its objects $M$
of a duality pairing
\[
P_M = (M,M^\vee,\eta_M,\varepsilon_M)
\]
for $M$, such that
$P_\I$ is the unit pairing, $P_{M \otimes M'}$ is the tensor
product of the pairings $P_M$ and $P_{M'}$, and $P_{M^\vee}$ is the dual
pairing of $P_M$.
A strict rigid $k$\nd tensor functor between two such categories
is a strict $k$\nd tensor functor which preserves the assigned pairings.
Every rigid $k$\nd pretensor category $\sC$
is $k$\nd tensor equivalent to a strict rigid $k$\nd pretensor category.
To see this we may suppose that $\sC$ is a strict $k$\nd pretensor category.
Choose for each object $M$ in $\sC$ a dual $M^\vee$.
If $\sC$ has monoid of objects $\Gamma$, then
the map from $\Gamma \amalg \Gamma$ to $\Gamma$ with components
$M \mapsto M$ and $M \mapsto M^\vee$
extends uniquely to a homomorphism $j:\Gamma' \to \Gamma$
from the free monoid $\Gamma'$ on $\Gamma \amalg \Gamma$.
By pulling back the structure of $\sC$ along $j$,
we obtain a strict $k$\nd pretensor category $\sC'$ with monoid
of objects $\Gamma'$ and a strict $k$\nd tensor functor $\sC' \to \sC$
which is an equivalence and coincides with $j$ on monoids of objects.
If $e_1$ and $e_2$ are the components
$\Gamma \to \Gamma'$ of $\Gamma \amalg \Gamma \to \Gamma'$,
then there exists for each $M$ in $\sC$ a duality pairing $P'{}\!_{e_1(M)}$
for $e_1(M)$ with dual $e_2(M)$.
Any object $M'$ of $\sC'$ can be written uniquely as
$M'{}\!_1 \otimes \dots \otimes M'{}\!_n$ with each $M'{}\!_r$
either $e_1(M)$ or $e_2(M)$ for some $M$.
We obtain a strict rigid structure on $\sC'$ by assigning to
$M'$ the tensor product of the pairings $P'{}\!_{M'{}\!_r}$,
where $P'{}\!_{e_2(M)}$ is defined as the pairing dual to  $P'{}\!_{e_1(M)}$.

\begin{prop}\label{p:freerigid}
For every family $(\alpha_\lambda)_{\lambda \in \Lambda}$ of elements of $k$
there exists a free rigid $k$\nd pretensor category of type
$(\alpha_\lambda)_{\lambda \in \Lambda}$.
If $\sR$ is such a category and
$(N_\lambda)_{\lambda \in \Lambda}$ is its universal family of objects,
and if $\lambda_1,\lambda_2, \dots,\lambda_n$ are distinct elements of $\Lambda$,
then
\begin{equation}\label{e:freerigid}
\sR(N_{\lambda_1}{}\!^{\otimes r_1} \otimes N_{\lambda_2}{}\!^{\otimes r_2}
    \otimes    \dots  \otimes   N_{\lambda_n}{}\!^{\otimes r_n},
N_{\lambda_1}{}\!^{\otimes s_1} \otimes N_{\lambda_2}{}\!^{\otimes s_2}
    \otimes    \dots  \otimes   N_{\lambda_n}{}\!^{\otimes s_n})
\end{equation}
is $0$ unless $r_i = s_i$ for $i = 1,2,\dots,n$, when it is
generated as a $k$\nd module by the elements
$\sigma_1 \otimes \sigma_2 \otimes \dots \otimes \sigma_n$,
with $\sigma_i$ a symmetry permuting the factors $N_{\lambda_i}$ of
$N_{\lambda_i}{}^{\otimes r_i}$.
\end{prop}

\begin{proof}
Denote by $\sP$ the category whose objects are pairs
$(\sC,(M_\lambda)_{\lambda \in \Lambda})$ consisting of a small strict rigid
$k$\nd pretensor category $\sC$ and a family $(M_\lambda)_{\lambda \in \Lambda}$
of objects in $\sC$ with $M_\lambda$ of rank $\alpha_\lambda$,
where a morphism from $(\sC,(M_\lambda)_{\lambda \in \Lambda})$ to
$(\sC',(M'{}\!_\lambda)_{\lambda \in \Lambda})$ is a strict rigid $k$\nd tensor
functor from $\sC$ to $\sC'$ which sends $M_\lambda$ to $M'{}\!_\lambda$.
Then $\sP$ has small hom sets.
Small limits in $\sP$ exist: they are given by limits of sets of objects
and sets of morphisms equipped with the evident tensor product, symmetries
and dualities.
Consider the full subcategory $\sP_0$ of $\sP$ consisting of those
$(\sC,(M_\lambda)_{\lambda \in \Lambda})$ such that $\sC$ is generated
by $(M_\lambda)_{\lambda \in \Lambda}$, in the sense that $\sC$ has
no strictly smaller strictly rigid $k$\nd pretensor subcategory
(not necessarily full) containing each $M_\lambda$.
It is essentially small, because if
$(\sC,(M_\lambda)_{\lambda \in \Lambda})$ lies in $\sP_0$ then the sets
of objects and of morphisms of $\sC$ are either countable or
of cardinality bounded by that of $\Lambda$.
Further every object $(\sC,(M_\lambda)_{\lambda \in \Lambda})$ of $\sP$ is the
target of a morphism with source
$(\sC_0,(M_{0\lambda})_{\lambda \in \Lambda})$ in $\sP_0$, with
$\sC_0$ for example the subcategory of $\sC$ generated by the $M_\lambda$.
Thus (\cite{Mac}, V.6 Theorem~1) $\sP$ has an initial object
$(\sR_0,(N_\lambda)_{\lambda \in \Lambda})$.

The map from $\Lambda \amalg \Lambda$ to $\Ob \sR_0$ with components
$\lambda \mapsto N_\lambda$ and $\lambda \mapsto (N_\lambda)^\vee$
extends uniquely to a homomorphism $h$
from the free monoid on $\Lambda \amalg \Lambda$ to the monoid
$\Ob \sR_0$ under tensor product.
By pulling back the structure of $\sR_0$ along $h$,
we obtain a strict rigid $k$\nd tensor functor $\sR' \to \sR_0$
which coincides with $h$ on monoids of objects.
It defines a morphism
from $(\sR',(N'{}\!_\lambda)_{\lambda \in \Lambda})$ to
$(\sR_0,(N_\lambda)_{\lambda \in \Lambda})$ in $\sP$,
with $\Ob \sR'$ freely generated by the $N'{}\!_\lambda$ and
$(N'{}\!_\lambda)^\vee$.
The identity of $(\sR_0,(N_\lambda)_{\lambda \in \Lambda})$ factors (uniquely)
through this morphism.
It follows that $h$ is an isomorphism from $\Ob \sR'$ to $\Ob \sR_0$,
so that $\Ob \sR_0$ is freely generated by the $N_\lambda$ and $(N_\lambda)^\vee$.
Thus any object $N$ of $\sR$ can be written uniquely as
\[
N = N_1 \otimes N_2 \otimes \dots \otimes N_{r(N)}
\]
with each $N_i$ either $N_\lambda$
or $(N_\lambda)^\vee$ for some $\lambda$.

If $\sR$ is a pseudo-abelian hull of the underlying $k$\nd
pretensor category of $\sR_0$, we show that $\sR$ together with
$(N_\lambda)_{\lambda \in \Lambda}$ is a free rigid category of type
$(\alpha_\lambda)_{\lambda \in \Lambda}$.
It will be enough for this to show that the following two properties hold
for every essentially small rigid $k$\nd pretensor category $\sC$.
\begin{enumerate}
\renewcommand{\theenumi}{(\arabic{enumi})}
\item\label{i:freerigides}
If $(M_\lambda)_{\lambda \in \Lambda}$ is a family of dualisable objects in $\sC$
with $M_\lambda$ of rank $\alpha_\lambda$, then there is a $k$\nd tensor
functor $\sR_0 \to \sC$ which sends $(N_\lambda)_{\lambda \in \Lambda}$
to $(M_\lambda)_{\lambda \in \Lambda}$.
\item\label{i:freerigidff}
If $T$ and $T'$ are $k$\nd tensor functors $\sR_0 \to \sC$ and
$(\varphi_\lambda)_{\lambda \in \Lambda}$ is a family of isomorphisms
$\varphi_\lambda:T(N_\lambda) \iso T'(N_\lambda)$, then there is a unique tensor isomorphism
$\varphi:T \iso T'$ with $\varphi_{N_\lambda} = \varphi_\lambda$
for each $\lambda$.
\end{enumerate}
Replacing $\sC$ by an equivalent category, we may suppose that $\sC$ is small
and has a structure of strict rigid $k$\nd pretensor category.
Then \ref{i:freerigides} holds with a strict rigid $\sR_0 \to \sC$.
We may thus suppose in \ref{i:freerigidff} that $T$ is strict rigid.
If $\delta_\lambda$ is the canonical isomorphism from $T'(N_\lambda)^\vee$
to $T'((N_\lambda)^\vee)$,
then the tensor product of the isomorphisms from
$T(N_i)$ to $T'(N_i)$ defined by $\varphi_\lambda$ when $N_i = N_\lambda$
and $\delta_\lambda \circ (\varphi_\lambda{}^{-1})^\vee$ when
$N_i = (N_\lambda)^\vee$ is an isomorphism
\[
T(N) = T(N_1) \otimes T(N_2) \otimes \dots \otimes T(N_{r(N)})
 \iso T'(N_1) \otimes T'(N_2) \otimes \dots \otimes T'(N_{r(N)}).
\]
Denote by $\varphi_N$ its composite with the structural isomorphism
\[
T'(N_1) \otimes T'(N_2) \otimes \dots \otimes T'(N_{r(N)}) \iso
T'(N_1 \otimes N_2 \otimes \dots \otimes N_{r(N)}) = T'(N)
\]
of $T'$.
Then $\varphi_\I = 1_\I$, so that there is a unique $k$\nd tensor functor
$\widetilde{T}:\sR_0 \to \sC$ such that the
$\varphi_N$ are the components of a tensor
isomorphism $\varphi$ from $\widetilde{T}$ to $T'$.
By construction, $\varphi_{N' \otimes N''}$ is
$\varphi_{N'} \otimes \varphi_{N''}$ composed with the structural isomorphism of $T$.
Thus $\widetilde{T}$ is a strict $k$\nd tensor functor.
Since $\varphi_{N_\lambda}$ is $\varphi_\lambda$ and
$\varphi_{(N_\lambda)^\vee}$ is $\delta_\lambda \circ (\varphi_\lambda{}^{-1})^\vee$,
the canonical
isomorphism from $\widetilde{T}(N_\lambda)^\vee$ to $\widetilde{T}((N_\lambda)^\vee)$
is the identity,
so that $\widetilde{T}$ preserves the duality pairings associated to the $N_\lambda$
and hence to all objects of $\sR_0$.
Thus $\widetilde{T}$ is a strict rigid $k$\nd tensor functor.
Since $\widetilde{T}(N_\lambda) = T(N_\lambda)$ for every $\lambda$,
we thus have $\widetilde{T} = T$,
so that $\varphi$ is a tensor isomorphism $T \iso T'$ with
$\varphi_{N_\lambda} = \varphi_\lambda$.
The only tensor automorphism $\theta$ of $T$ with
each $\theta_{N_\lambda}$ the identity is the identity, because
$\theta_{N^\vee} = ((\theta_N)^\vee)^{-1}$ and
$\theta_{N' \otimes N''} = \theta_{N'} \otimes \theta_{N''}$ for every
$N$, $N'$ and $N''$.
Thus $\varphi$ is unique.
This proves \ref{i:freerigidff}.

It is enough to prove the statement on hom $k$\nd modules with $\sR$ replaced by
$\sR_0$ in \eqref{e:freerigid}.
If $P$ and $Q$ are objects of $\sR_0$,
then to each bijection $\sigma$ from $[1,r(P)]$ to $[1,r(Q)]$ such that
$Q_i = P_{\sigma(i)}$ there is associated a symmetry $P \iso Q$ in $\sR_0$.
Denote by $\sS(P,Q)$ the set of such symmetries.
The $\sS(P,Q)$ contain all identities
and symmetries of $\sR_0$ and are stable under tensor product and
composition.
Denote by $\sR'{}\!_0(P,Q)$ the $k$\nd submodule of $\sR_0(P,Q)$
spanned by the elements
\[
q \circ (f_1 \otimes f_2 \otimes \dots \otimes f_n) \circ p
\]
with $p$ in $\sS(P,P')$ and $q$ in $\sS(Q,Q')$ for some $P'$ and $Q'$,
where each $f_i$ is an identity $1_{N_\lambda}$ or $1_{(N_\lambda)^\vee}$,
or a unit $\eta_{N_\lambda}:\I \to (N_\lambda)^\vee \otimes N_\lambda$ or counit
$\varepsilon_{N_\lambda}:N_\lambda \otimes (N_\lambda)^\vee \to \I$ for some $\lambda$.
The $\sR'{}\!_0(P,Q)$
contain all identities, symmetries, units $\eta_N$ and counits $\varepsilon_N$ of
$\sR_0$, and are
stable under tensor product.
To see that $g \circ f$ lies in $\sR'{}\!_0(P,R)$ when $f$ lies in $\sR'{}\!_0(P,Q)$
and $g$ lies in $\sR'{}\!_0(Q,R)$, we reduce first to the case where
$f$ is of the form $f_1 \otimes f_2 \otimes \dots \otimes f_n$ as above,
and then by writing $f$ as a composite to the two cases where $f$ is
$N' \otimes \varepsilon_{N_\lambda} \otimes N''$ or
$N' \otimes \eta_{N_\lambda} \otimes N''$ for some $N'$ and $N''$.
The first case is clear.
For the second, we reduce using the
naturality of the symmetries and the stability of the
$\sR'{}\!_0(P,Q)$ under tensor
product to one of the five  cases where $N' = N'' = \I$ and $g$ is the identity,
or $N' = N'' = \I$ and $g$ is $\varepsilon_{(N_\lambda)^\vee}$,
or $N' = N_\lambda$, $N'' = \I$ and $g$ is
$\varepsilon_{N_\lambda} \otimes N_\lambda$,
or $N' = \I$, $N'' = (N_\lambda)^\vee$ and $g$ is
$(N_\lambda)^\vee \otimes \varepsilon_{N_\lambda}$,
or $N' = N_\lambda$, $N'' =  (N_\lambda)^\vee$ and $g$ is
$\varepsilon_{N_\lambda} \otimes \varepsilon_{N_\lambda}$.
The respective composites $g \circ f$ for these five cases are
$\eta_{N_\lambda}$, $\alpha_\lambda 1_\I$, $1_{N_\lambda}$, $1_{(N_\lambda)^\vee}$,
and $\varepsilon_{N_\lambda}$.
The $\sR'{}\!_0(P,Q)$ thus define a (not necessarily full) strict
rigid $k$\nd tensor subcategory $\sR'{}\!_0$ of $\sR_0$ with the same objects.

The embedding $\sR'{}\!_0 \to \sR_0$ is a morphism
$(\sR'{}\!_0,(N_\lambda)_{\lambda \in \Lambda}) \to
(\sR'{}\!_0,(N_\lambda)_{\lambda \in \Lambda})$ in $\sP$,
through which the identity of $(\sR_0,(N_\lambda)_{\lambda \in \Lambda})$
necessarily factors.
Thus $\sR'{}\!_0 = \sR_0$,
and $\sR'{}\!_0(P,Q) = \sR_0(P,Q)$ for every $P$ and $Q$.
Suppose now that $P$ is the source and $Q$ the target in the
hom $k$\nd module \eqref{e:freerigid}.
Then $\sR'{}\!_0(P,Q)$ and hence \eqref{e:freerigid} is generated by $\sS(P,Q)$.
Since $\sS(P,Q)$ is empty unless $r_i = s_i$ for $i = 1,2,\dots,n$, when it
coincides with the set of
$\sigma_1 \otimes \sigma_2 \otimes \dots \otimes \sigma_n$,
with $\sigma_i$ a symmetry of $N_{\lambda_i}{}^{\otimes r_i}$,
the statement on hom $k$\nd modules follows.
\end{proof}

It is easily shown that the
hom $k$\nd module \eqref{e:freerigid} of Proposition~\ref{p:freerigid}
is in fact freely generated by
the $\sigma_1 \otimes \sigma_2 \otimes \dots \otimes \sigma_n$ when
$r_i = s_i$ for $i = 1,2,\dots,n$,
by reducing to the case where $k$ is a polynomial algebra over $\Z$ freely
generated by the $\alpha_\lambda$ and then considering $k$\nd tensor functors
to the category of $\Q$\nd vector spaces.
However this will not be required for what follows.

\subsection{Tate categories}\label{ss:Tate}

Let $\sC$ be a $k$\nd pretensor category.
An object $L$ of $\sC$ is called \emph{invertible} if there
is an object $L'$ of $\sC$ such that $L \otimes L'$ is
isomorphic to $\I$.
Every invertible object of $\sC$ has a dual: any isomorphism
$\varepsilon:L \otimes L' \to \I$ is the counit of a duality pairing between
$L$ and $L'$ with unit the composite of $\varepsilon^{-1}$ and the symmetry
$L \otimes L' \iso L' \otimes L$.
An invertible object $L$ in $\sC$ is said to be \emph{even} if the
symmetry that interchanges the two factors in $L \otimes L$ is the identity.

We denote by $\Z_k$ the strict $k$\nd pretensor category with monoid of
objects the additive group $\Z$, endomorphism $k$\nd algebras $k$
and all other hom $k$\nd spaces $0$, and symmetries the identities.
Then $L$ is an even invertible object of $\sC$ if and only if there
exists a $k$\nd tensor functor $\Z_k \to \sC$ which sends the object $1$
of $\Z_k$ to $L$.
In fact evaluation at $1 \in \Z_k$ defines an equivalence from the category
of $k$\nd tensor functors $\Z_k \to \sC$ and tensor isomorphisms between
them to the category of even invertible objects of $\sC$ and isomorphisms
between them.

An object $L$ of $\sC$ will be said to \emph{strictly commute} with all
objects of $\sC$ if the symmetry $N \otimes N' \iso N' \otimes N$ is the
identity when either $N$ or $N'$ is $L$, and to \emph{strictly associate}
with all objects of $\sC$ if the associativity
$(N \otimes N') \otimes N'' \iso N \otimes (N' \otimes N'')$ is the
identity when one of $N$, $N'$ or $N''$ is $L$.
The identity $\I$ always strictly commutes and strictly associates with
all objects of $\sC$.

\begin{defn}
A \emph{Tate $k$\nd pretensor category} is a $k$\nd pretensor category $\sC$
equipped with a family of objects $(\I(m))_{m \in \Z}$ with $\I(0) = \I$
and $\I(m) \otimes \I(n) = \I(m+n)$ for every $m, n \in \Z$, such that $\I(m)$
strictly commutes and strictly associates with all objects of $\sC$ for every
$m \in \Z$.
A \emph{Tate $k$\nd tensor functor} from a Tate $k$\nd pretensor category $\sC$
to a Tate $k$\nd pretensor category $\sC'$ is a $k$\nd tensor functor
$T:\sC \to \sC'$ with $T(\I(m)) = \I(m)$ for every $m \in \Z$,
such that the structural isomorphism
$T(M) \otimes T(N) \iso T(M \otimes N)$ is the identity when either $M$ or $N$
is $\I(m)$ for some $m \in \Z$.
\end{defn}

There is a canonical structure of Tate $k$\nd pretensor category on $\Z_k$,
with $\I(m)$ the object $m$ of $\Z_k$.
Let $\sC$ be a Tate $k$\nd pretensor category.
Then $\I(m)$ is an even invertible object  in $\sC$ for every $m$.
In fact there is a unique Tate $k$\nd tensor
functor from $\Z_k$ to $\sC$.
For any object $M$ of $\sC$ and integer $m$, we write
\[
M(m) = M \otimes \I(m) = \I(m) \otimes M,
\]
and similarly for morphisms of $\sC$.
We then have
\[
M(m) \otimes N(n) = (M \otimes N)(m + n)
\]
for every pair of objects $M$ and $N$ of $\sC$ and of integers $m$ and $n$.
If $\sigma_{M,N}$ is the symmetry $M \otimes N \iso N \otimes M$ in $\sC$,
then $\sigma_{M(i),N(j)} = \sigma_{M,N}(i+j)$, and similarly for
the associativities and for the structural isomorphisms of a Tate $k$\nd tensor functor.

\begin{lem}\label{l:cattwist}
Let $\sB$ be a $k$\nd pretensor category
and $L$ be an even invertible object of $\sB$.
Then there exists a Tate $k$\nd pretensor category $\widetilde{\sB}$
and a strict $k$\nd tensor functor $E:\sB \rightarrow \widetilde{\sB}$,
such that
\begin{enumerate}
\renewcommand{\theenumi}{(\alph{enumi})}
\item
$E$ is an equivalence, and every object of $\widetilde{\sB}$ can be
written uniquely in the form $E(M)(m)$ with $M \in \sB$ and $m \in \Z$
\item
$E(L)$ is isomorphic to $\I(1)$ in  $\widetilde{\sB}$.
\end{enumerate}
\end{lem}

\begin{proof}
There exists a $k$\nd tensor functor $I:\Z_k \to \sB$ with $I(1) = L$.
Denote by $K_1:\sB \to \sB \otimes_k \Z_k$  and $K_2:\Z_k \to \sB \otimes_k \Z_k$
the strict $k$\nd tensor functors $(-,0)$ and $(\I,-)$.
Then there is a $k$\nd tensor functor $P:\sB \otimes_k \Z_k \to \sB$
with $PK_1 = \Id_{\sB}$ and $PK_2 = I$.
Factor $P$ as
\[
\sB \otimes_k \Z_k \xrightarrow{P'} \widetilde{\sB} \xrightarrow{P''} \sB
\]
with $P'$ a strict $k$\nd tensor functor which is bijective on objects and $P''$
fully faithful.
If we write $E = P'K_1$ and $\widetilde{I} = P'K_2$, then
$P''E = \Id_{\sB}$ and $P''\widetilde{I} = I$.
Since $P''$ is fully faithful, it follows that $E$ is an
equivalence and that $\widetilde{I}(1)$ is isomorphic to $E(L) = (EI)(1)$.
The objects $K_2(m) = (\I,m)$ of $\sB \otimes_k \Z_k$
strictly commute and strictly associate with all other objects.
Since $P'$ is strict and bijective on objects the same
holds for the $\widetilde{I}(m) = P'(K_2(m))$ in $\widetilde{\sB}$.
Thus $\widetilde{\sB}$ is a Tate $k$\nd pretensor category with
$\I(m) = \widetilde{I}(m)$.
\end{proof}

\begin{lem}\label{l:compatibletwist}
Let $\sB$ and $\sC$ be Tate $k$\nd pretensor categories,
$T:\sB \to \sC$ be a $k$\nd tensor functor,
$\sB_0$ be a subset of $\Ob \sB$,
and $t:\I(1) \iso T(\I(1))$ be an isomorphism.
Suppose that $\sB_0$ contains $\I$, and that every object of
$\sB$ can be written uniquely in the
form $M(r)$, with $M \in \sB_0$ and $r \in \Z$.
Then there is a unique pair $(T',\varphi)$ consisting of a
$k$\nd tensor functor $T':\sB \to \sC$ and
a tensor isomorphism $\varphi:T' \iso T$,
such that
\begin{enumerate}
\renewcommand{\theenumi}{(\alph{enumi})}
\item\label{i:compatibletwist}
$T'$ is a Tate $k$\nd tensor functor
\item\label{i:Midentity}
$(T'(M),\varphi_M) = (T(M),1_{T(M)})$ when $M \in \sB_0$
\item\label{i:Tate}
$\varphi_{\I(1)} = t$.
\end{enumerate}
\end{lem}

\begin{proof}
Suppose that \ref{i:compatibletwist}, \ref{i:Midentity} and
\ref{i:Tate} hold with $(T',\varphi)$ replaced by
$(T_1,\varphi_1)$ and by $(T_2,\varphi_2)$.
Then $\theta = (\varphi_2)^{-1} \circ \varphi_1$ is a tensor
isomorphism $T_1 \iso T_2$ such that $\theta_M$ is the
identity when either $M \in \sB_0$ or $M = \I(1)$.
Since $T_1$ and $T_2$ are Tate $k$\nd tensor functors,
this implies that $\theta$ is the identity and $T_1 = T_2$,
so that $(T_1,\varphi_1) = (T_2,\varphi_2)$.
Thus the pair $(T',\varphi)$ is unique if it exists.

Denote the unique Tate $k$\nd tensor functors from $\Z_k$ to $\sB$ and $\sC$
by $I:\Z_k \to \sB$ and $K:\Z_k \to \sC$, and the
structural isomorphisms of $T$ by
\[
\psi_{M,N}:T(M) \otimes T(N) \iso T(M \otimes N).
\]
There is a unique tensor isomorphism
$\tau:K \iso TI$ such that $\tau_1 = t$.
By the hypotheses on $\sB_0$, there is a (unique) family
$(\varphi_M)_{M \in \Ob \sB}$ of isomorphisms
\[
\varphi_{M}:T'(M) \iso T(M)
\]
in $\sC$ such that for every $M_0 \in \sB_0$ and $m \in \Z$ we have
\(
T'(M_0(m)) = T(M_0)(m)
\)
and
\begin{equation}\label{e:phiM}
\varphi_{M_0(m)} = \psi_{M_0,\I(m)} \circ (T(M_0) \otimes \tau_m).
\end{equation}
Then $\varphi_\I = 1_\I$, so that the assignment $M \mapsto T'(M)$ extends
uniquely to a $k$\nd tensor functor $T':\sB \to \sC$
such that the $\varphi_M$ are the components of a tensor
isomorphism $\varphi:T' \iso T$.
Clearly \ref{i:Midentity} and \ref{i:Tate} hold for the pair $(T',\varphi)$.
To show that \ref{i:compatibletwist} holds for $(T',\varphi)$, it remains to check
that the structural isomorphism
\[
{\psi'}_{M,N}:T'(M) \otimes T'(N) \iso T'(M \otimes N)
\]
of $T'$ is the identity when $N = \I(m)$ with $m \in \Z$.
We have $\varphi_{M_0} = 1_{T(M_0)}$ for every $M_0 \in \sB_0$,
and  $\varphi_{\I(m)} = \tau_m$ for every $m \in \Z$.
By compatibility of $\varphi$ with the tensor products, \eqref{e:phiM}
thus also holds if its left hand side is composed on the right with
${\psi'}_{M_0,\I(m)}$.
Hence
\[
{\psi'}_{M_0,\I(m)} = 1
\]
for every $M_0 \in \sB_0$ and $m \in \Z$.
Further $\varphi I = \tau$, so that $T'I = K$ is strict.
Hence
\[
{\psi'}_{\I(m'),\I(m)} = 1
\]
for every $m,m' \in \Z$.
By the compatibility of $\psi'$ with the associativities we have
\[
{\psi'}_{M(m'),\I(m)} \circ ({\psi'}_{M,\I(m')} \otimes T'(\I(m)))
= {\psi'}_{M,\I(m+m')} \circ (T'(M) \otimes {\psi'}_{\I(m'),\I(m)})
\]
for every $M \in \sB$ and $m,m' \in \Z$,
because the relevant associativities are identities.
Taking $M = M_0 \in \sB_0$ then shows that
${\psi'}_{M_0(m'),\I(m)} = 1$ for every $m,m' \in \Z$.
Hence \ref{i:compatibletwist} holds.
\end{proof}

Call a Tate $k$\nd pretensor category \emph{separated} when
$M(i) = M$ implies $i = 0$ for every object $M$ and integer $i$.
Let $\sB$ be a Tate $k$\nd pretensor category.
Then using Lemmas~\ref{l:cattwist} and
\ref{l:compatibletwist} we can construct as follows a separated Tate $k$\nd pretensor
category $\sB'$ and a Tate $k$\nd tensor functor $T:\sB' \to \sB$ such that
the underlying $k$\nd pretensor category of $\sB'$ is strict and
the underlying  $k$\nd tensor functor of $T$ is an equivalence.
Let $E_1:\sB \to \sB_1$ be a $k$\nd tensor equivalence to
a strict $k$\nd pretensor category $\sB_1$.
By Lemma~\ref{l:cattwist} there is a Tate $k$\nd pretensor category
$\sB'$ and a strict $k$\nd tensor functor $E':\sB_1 \to \sB'$ such that
$E'$ is an equivalence, every object of $\sB'$ can be written uniquely in the
form $E'(M)(m)$ for some $M$ in $\sB_1$ and $m \in \Z$, and $(E'E_1)(\I(1))$
is isomorphic to $\I(1)$.
The associativity
\[
\alpha:(L' \otimes M') \otimes N' \iso L' \otimes (M' \otimes N')
\]
in $\sB'$ is the
identity when each of $L'$, $M'$ and $N'$ lies in the image of $E'$ by strictness of
$\sB_1$ and $E'$, and also when one of $L'$, $M'$ or $N'$ is of the form $\I(m)$
because $\sB'$ is a Tate $k$\nd pretensor category.
Appropriate pentagonal diagrams for the associativities in $\sB'$
then show that $\alpha$ is the identity for every $L'$, $M'$ and $N'$ in $\sB'$.
Thus $\sB'$ is strict.
Let $T':\sB' \to \sB$ be a $k$\nd tensor functor quasi-inverse to $E'E_1$.
Then $T'(\I(1))$ is isomorphic to $\I(1)$.
Thus by Lemma~\ref{l:compatibletwist} there is a Tate $k$\nd tensor functor
$T:\sB' \to \sB$ which is tensor isomorphic to $T'$, and hence
an equivalence.

Let $\sC$ be a Tate $k$\nd pretensor category, and $(\zeta_{M,N})$ be a family
parametrised by pairs of objects of $\sC$, with $\zeta_{M,N}$ an isomorphism
with target $M \otimes N$ such that $\zeta_{M,N}$ is the identity when either $M$
or $N$ is $\I$ and $\zeta_{M(i),N(j)} = \zeta_{M,N}(i+j)$ for every
$M$, $N$, $i$ and $j$.
Then by modifying the tensor product of $\sC$ according to $(\zeta_{M,N})$ as
in \ref{ss:pretensor},
we obtain a Tate $k$\nd tensor functor $\sC \to \sC'$ which is
the identity on underlying $k$\nd linear categories
and has structural isomorphisms the $\zeta_{M,N}$.
Let $T:\sB \to \sC$ be a Tate $k$\nd tensor functor which is injective on objects.
Take for $\zeta_{M,N}$ the inverse
\[
T(M_1 \otimes N_1) \iso T(M_1) \otimes T(N_1)
\]
of the structural isomorphism of $T$ when $M = T(M_1)$ and $N = T(N_1)$,
and the identity of $M \otimes N$ when either $M$ or $N$ is not in the
image of $T$.
Then the composite $T':\sB \to \sC'$ of $T$ with $\sC \to \sC'$
is a strict Tate $k$\nd tensor functor,
and $\sC'$ is separated if $\sC$ is.

Let $\sB$ be a separated Tate $k$\nd pretensor category.
If $\sC$ is a $k$\nd pretensor
category and $T:\sB \to \sC$ is a fully faithful $k$\nd tensor functor,
then there is a separated Tate $k$\nd pretensor category $\sB'$ containing
$\sB$ as a full $k$\nd pretensor subcategory
and a $k$\nd tensor equivalence $E:\sC \to \sB'$, such that
the embedding $\sB \to \sB'$ is a Tate $k$\nd tensor functor
which is tensor isomorphic $ET$.
Indeed taking a $k$\nd tensor equivalence $F:\sC \to \sC_1$
with the set of objects of $\sC_1$ sufficiently large,
and replacing $\sC$ by $\sC_1$ and $T$ by an appropriate $k$\nd tensor functor
tensor isomorphic to $FT$, we may suppose that $T$ is injective on objects.
Replacing $T$ by its composite with an appropriate $k$\nd tensor functor as in
Lemma~\ref{l:cattwist}, we may further suppose that $\sC$ is a
separated Tate $k$\nd pretensor category with $\I(1)$ isomorphic to $T(\I(1))$.
Since $\sB$ is separated, we may apply Lemma~\ref{l:compatibletwist}
with a suitable $\sB_0$.
Thus after replacing $T$ may also suppose that $T$
is a Tate $k$\nd tensor functor.
Composing with $\sC \to \sC'$ as above and replacing $\sC$ and $T$ by
$\sC'$ and $T'$ we may suppose in addition that
$T$ is strict.
Relabelling appropriately the objects of $\sC$ now gives a $\sB'$
and an isomorphism $E:\sC \to \sB'$ of $k$\nd pretensor categories
with the required properties.
Using this construction, we obtain for example a pseudo-abelian hull or
ind-completion $\sB'$ for $\sB$
with $\sB'$ a Tate $k$\nd pretensor category and the embedding $\sB \to \sB'$
a Tate $k$\nd tensor functor.

\section{Chow theories and Poincar\'e duality theories}\label{s:ChowPoin}

\subsection{Chow theories}

Recall that a $\Z$\nd graded $k$\nd module is a family of $k$\nd modules
parametrised by $\Z$, or equivalently a $k$\nd module $M$ together
a family $(M_i)_{i \in \Z}$ of $k$\nd submodules of $M$ such that
$M$ is the coproduct of the $M_i$.
An element of $M$ is be said to be \emph{homogeneous} if it lies $M_i$
for some $i$.
A homomorphism $M \to M'$ of degree $d$ of $\Z$\nd graded $k$\nd modules
a homomorphism $M \to M'$ of $k$\nd modules which sends $M_i$
into $M'{}_{i+d}$.
When $d = 0$ we speak simply of a homomorphism of
$\Z$\nd graded $k$\nd modules.
A (commutative) $\Z$\nd graded $k$\nd algebra is a $\Z$\nd graded
$k$\nd module $R$ together with a structure of (commutative)
algebra on its underlying $k$\nd module such that
$x.y \in R_{i+j}$ when $x \in R_i$ and $y \in R_j$.
A homomorphism of $\Z$\nd graded $k$\nd algebras is
a homomorphism of their underlying $\Z$\nd graded $k$\nd modules
which is at the same time a homomorphism of $k$\nd algebras.

A \emph{dimension function} on a cartesian monoidal category
$\sV$
is an assignment to each object $X$ of $\sV$ of an integer $\dim X$,
such that $\dim X' = \dim X$ if $X'$ is isomorphic to $X$, and
$\dim (X \times Y) = \dim X + \dim Y$
for every $X$ and $Y$.

\begin{defn}\label{d:Chow}
Let $\sV$ be a cartesian monoidal category
with a dimension function.
A \emph{$k$\nd linear Chow theory with source $\sV$} is an assignment $C$
to each object $X$ in $\sV$ of a commutative $\Z$\nd graded $k$\nd algebra $C(X)$
and to each morphism $p:X \to Y$ in $\sV$ of a homomorphism
$p^*:C(Y) \to C(X)$ of graded $k$\nd algebras and a homomorphism
$p_*:C(X) \to C(Y)$ of degree $\dim Y - \dim X$ of graded $k$\nd modules,
such that the following conditions hold.
\begin{enumerate}
\renewcommand{\theenumi}{(\alph{enumi})}
\item\label{i:Chowfun}
$(1_X)^* = 1_{C(X)}$ and $(1_X)_* = 1_{C(X)}$ for every $X$ in $\sV$,
and $(q\circ p)^* = p^* \circ q^*$ and $(q \circ p)_* = q_* \circ p_*$
for every $p:X\to Y$ and $q:Y \to Z$ in $\sV$.
\item\label{i:Chowproj}
$p_*(x.p^*(y)) = p_*(x).y$ for every $p:X \to Y$ in $\sV$ and $x \in C(X)$ and $y \in C(Y)$.
\item\label{i:Chowiso}
$i_* = (i^{-1})^*$ for every isomorphism $i:X \iso X'$ in $\sV$.
\item\label{i:Chowtens}
$(r \times s)_*(x \otimes y) = r_*(x) \otimes s_*(y)$
for every $r:X \rightarrow X'$ and $s:Y \rightarrow Y'$ in $\sV$
and $x \in C(X)$ and $y \in C(Y)$,
where $z_1 \otimes z_2$ denotes $(\pr_1)^*(z_1).(\pr_2)^*(z_2)$
in $C(Z_1 \times Z_2)$
for $Z_1$ and $Z_2$ in $\sV$ and $z_1 \in C(Z_1)$ and $z_2 \in C(Z_2)$.
\end{enumerate}
\end{defn}

In the presence of \ref{i:Chowfun} and \ref{i:Chowproj},
condition  \ref{i:Chowiso} of Definition~\ref{d:Chow} is equivalent
to the following one:
\begin{enumerate}
\item[\ref{i:Chowiso}$'$]
$i_*(1) = 1$ for every isomorphism $i:X \iso X'$ in $\sV$.
\end{enumerate}
This can be seen by taking $x = 1$ in \ref{i:Chowproj}.
In the presence of \ref{i:Chowfun}, \ref{i:Chowproj} and \ref{i:Chowiso},
condition \ref{i:Chowtens} of Definition~\ref{d:Chow} is equivalent
to the following one:
\begin{enumerate}
\item[\ref{i:Chowtens}$'$]
$(r \times Y)_* \circ (\pr_1)^* = (\pr_1)^* \circ r_*$ for every
$r:X \rightarrow X'$ and $Y$ in $\sV$, where the projections are
from $X \times Y$ and $X' \times Y$.
\end{enumerate}
Indeed \ref{i:Chowtens}$'$ is the particular case
of \ref{i:Chowtens} where  $s = 1_Y$ and $y = 1$.
Conversely if \ref{i:Chowtens} holds with  $s = 1_Y$ and $y = 1$
it holds with  $s = 1_Y$ and arbitrary $y$ by \ref{i:Chowproj},
so that applying the symmetries interchanging the two factors and using
\ref{i:Chowiso} shows that it holds also with $r = 1_X$ and arbitrary $s$,
and hence in general.

A $k$\nd linear Chow theory with source $\sV$ will often be called simply
a Chow theory with source $\sV$.
If $C$ is such a Chow theory,
we write $C^i(X)$ for $C(X)_i$.
The homomorphisms $p^*:C(Y) \to C(X)$ and $p_*:C(X) \to C(Y)$ associated to
\mbox{$p:X \to Y$} in $\sV$ will be called the pullback and push forward along $p$.

\begin{defn}\label{d:Chowmor}
Let $C$ and $C'$ be $k$\nd linear Chow theories with source $\sV$.
A \emph{morphism from $C$ to $C'$} is an assignment $\varphi$ to
every object $X$ of $\sV$
of a homomorphism $\varphi_X:C(X) \rightarrow C'(X)$ of $\Z$\nd graded $k$\nd algebras
such that for every $p:X \rightarrow Y$ in $\sV$
\begin{enumerate}
\renewcommand{\theenumi}{(\alph{enumi})}
\item\label{i:Chowmorpull}
$\varphi_X \circ p^* = p^* \circ \varphi_Y$
\item\label{i:Chowmorpush}
$\varphi_Y \circ p_* = p_* \circ \varphi_X$.
\end{enumerate}
\end{defn}

We thus have a category of Chow theories with source $\sV$, with composition
of morphisms defined component-wise.

Let $C$ be a Chow theory $C$ with source $\sV$.
An \emph{ideal} of $C$ is an assignment $J$
to every object $X$ in $\sV$ of a graded ideal $J(X)$ of $C(X)$ such that $p^*$ sends
$J(Y)$ into $J(X)$ and $p_*$ sends $J(X)$ into $J(Y)$ for every $p:X \to Y$ in $\sV$.
Similarly we define a \emph{Chow subtheory} of $C$ by considering graded $k$\nd subalgebras
instead of graded ideals.
In the case of ideals, it is enough require stability under pullback and push forward
along projections in $\sV$, because
\begin{equation}\label{e:pullproj}
p^* = (\pr_1)_*((\Gamma_p)_*(1) . (\pr_2)^*(-)),
\end{equation}
where $\Gamma_p$ denotes the graph of $p$, and
\begin{equation}\label{e:pushproj}
p_* = (\pr_2)_*((\pr_1)^*(-) . (\Gamma_p)_*(1)).
\end{equation}
If $J$ is an ideal in $C$, then by factoring out the $J(X)$ we obtain the quotient Chow
theory $C/J$ of $C$ with source $\sV$.
The ideals of $C$ behave in many ways like the ideals of a commutative $k$\nd algebra.
For example call an ideal $J$ of $C$ prime if $J \ne C$ and if
for every $x_1 \in C(X_1)$ and $x_2 \in C(X_2)$ with $x_1 \otimes x_2 \in J(X_1 \times X_2)$
either $x_1 \in J(X_1)$ or $x_2 \in J(X_2)$.
Then the intersection of all prime ideals of $C$ is the nilradical of $C$,
which assigns to $X$ the set $x \in C(X)$ with
$x^{\otimes n} = 0$ for some $n$.
When $C^0(\I)$ is a field, $C$ has a unique maximal ideal:
the elements $x$ of $C^i(X)$ which lie in this ideal are those
which are ``numerically equivalent to 0'', i.e.\ those for which $a_*(x.x') = 0$
for every $x'$ in $C^{\dim X - i}(X)$, where $a$ is the unique morphism $X \to \I$.
The quotient of such a $C$ by its unique maximal ideal will be written $\overline{C}$.

Let $E:\sV' \to \sV$ be a functor which preserves products and
dimensions and $C$ be a Chow theory with source $\sV$ .
Then the \emph{pullback of $C$ along $E$}
is the Chow theory $CE$ with source $\sV'$ defined by
$(CE)(X') = C(E(X'))$ for $X'$ in $\sV'$, and $p'{}^* = E(p')^*$ and
$p'{}\!_* = E(p')_*$ for $p':X' \to Y'$ in $\sV'$.
The pullback along $E$ of a morphism $\varphi:C \to C'$ of Chow theories
with source $\sV$ is the morphism $\varphi E:CE \to C'E$ with
$(\varphi E)_{X'} = \varphi_{E(X')}$ for $X'$ in $\sV'$.
If $E$ is an equivalence, it induces by pullback an equivalence on categories
of Chow theories.
The pullback $JE$ of the ideal $J$ of $C$ along $E$
is defined by $(JE)(X') = J(E(X'))$.
Suppose that $E$ is essentially surjective.
Then the assignment $J \mapsto JE$ defines a bijection from ideals of $C$
to ideals of $CE$.
The inverse of $J \mapsto JE$ sends the ideal $J'$ of $CE$ to the unique ideal
$J$ of $C$ with $J(X) = i^*(J'(X'))$ for every isomorphism $i:X \iso E(X')$.
That such a $J$ exists follows from \eqref{e:pullproj} and \eqref{e:pushproj}.

Although it will not be necessary for what follows,
it sometimes useful to consider instead of a dimension \emph{function}
on $\sV$ more generally a dimension \emph{functor} on $\sV$.
Explicitly, a dimension functor is a symmetric monoidal functor $(\sV^*)^{\mathrm{op}} \to \Z_k$,
where $\sV^*$ has objects the objects of $\sV$ and morphisms
the isomorphisms of $\sV$, and $\Z_k$ is as in Section~\ref{ss:Tate}.
Suppose for example that $\sV$ is obtained from the category
$\sV_0$ of abelian varieties over a field by formally inverting the isogenies,
and that the choice of products in $\sV$ is the same as that in $\sV_0$.
If $k$ contains $\Q$,
the usual dimension function on $\sV_0$ extends to a strict symmetric
monoidal functor $(\sV^*)^{\mathrm{op}} \to \Z_k$ which sends $A$ to
$\dim A$ in $\Z = \Ob \Z_k$
and $f:A \iso A'$ to the endomorphism $d(f)$ of $\dim A = \dim A'$,
where $d(f)^{-1}$ is the degree of $f$.
Then provided that condition \ref{i:Chowiso} of Definition~\ref{d:Chow} of a Chow theory
is replaced by the condition $i_* = d(i)^{-1} (i^{-1})^*$,
the Chow theory $CH(-) \otimes k$ with source $\sV_0$ can be extended to a Chow theory
with source $\sV$.

\subsection{Poincar\'e algebras}

Let $\sC$ be a Tate $k$\nd pretensor category.
If $d$ is an integer, then a \emph{Poincar\'e algebra of dimension $d$}
in $\sC$ is a pair $(R,\nu)$ consisting of a commutative algebra $R$ in
$\sC$ and a morphism $\nu:R(d) \to \I$ in $\sC$
such that the composite
\[
R \otimes R(d) = (R \otimes R)(d) \xrightarrow{\mu(d)} R(d)
                                      \xrightarrow{\nu}  \I,
\]
where $\mu:R \otimes R \to R$ is the multiplication,
is the counit of a duality pairing between $R$ and $R(d)$.
Often we omit the $\nu$ from the notation.
For a Poincar\'e algebra $(R,\nu)$ of dimension $d$, the object $R$ has thus
a canonical dual $R^\vee = R(d)$ and a canonical duality pairing between
$R$ and $R(d)$.
The transpose of the identity $\I \to R$ of $R$ for this pairing is $\nu$.
More generally we have a canonical duality pairing between $R(j)$
and $R(d-j)$ for every $j$, with the same unit and counit as for that
between $R$ and $R(d)$.
Since the multiplication is commutative, the counit for this pairing
is symmetric, in the sense that composing the symmetry interchanging
the two factors in $R(d-j) \otimes R(j)$ with the counit leaves the counit unchanged.
Hence the dual pairing, with $R(j)$ the dual $R(d-j)^\vee$ of $R(d-j)$,
coincides with the pairing between $R(d-j)$ and $R(j)$.

If $T:\sC \to \sC'$ is a Tate $k$\nd tensor functor and $(R,\nu)$ is
a Poincar\'e algebra of dimension $d$ in $\sC$, then
$(T(R),T(\nu))$ is a Poincar\'e algebra of dimension $d$ in $\sC'$.

Let $(R,\nu)$ and $(R',\nu')$ be Poincar\'e algebras in
$\sC$ of dimensions respectively $d$ and $d'$.
Then for a morphism $f:R \rightarrow R'$ in $\sC$ we define using the
canonical duality pairings the transpose
\[
f^\vee:R'(d') \rightarrow R(d)
\]
of $f$.
The morphism
\[
\nu \otimes \nu':(R \otimes R')(d+d') = R(d) \otimes R'(d') \to \I
\]
defines a structure of Poincar\'e algebra $(R \otimes R',\nu \otimes \nu')$
on $R \otimes R'$ of dimension $d+d'$.
The duality pairing between $R \otimes R'$ and $(R \otimes R')(d + d')$
defined by this structure of Poincar\'e algebra is the tensor product
of those between $R$ and $R(d)$ and $R'$ and $R'(d')$.

Let $(R,\nu)$ and $(R',\nu')$ be Poincar\'e algebras in $\sC$
of respective dimensions $d$ and $d'$,
and $f:R \to R'$ be a morphism of algebras in $\sC$.
If  $R'(d')$ is regarded as an $R$\nd module by restricting its
$R'$\nd module structure along $f$,
then $f^\vee:R'(d') \rightarrow R(d)$ is a morphism of $R$-modules,
i.e.\ the square
\begin{equation}\label{e:projsquare}
\begin{CD}
R \otimes R'(d')   @>{\beta(d')}>>   R'(d')   \\
@VV{R \otimes f^\vee}V       @VV{f^\vee}V \\
R \otimes R(d)   @>{\mu(d)}>>         R(d)
\end{CD}
\end{equation}
commutes, where $\mu$ and $\mu'$ are the multiplications and
\[
\beta = \mu' \circ (f \otimes R')
\]
is the action $R \otimes R' \to R'$ of $R$ on $R'$.
To see that \eqref{e:projsquare} commutes, we may suppose by taking a Tate $k$\nd tensor functor
$T:\sC' \to \sC$ with $\sC'$ strict and $T$ an equivalence, and replacing
$R$ and $R'$ by Poincar\'e algebras in the image of $T$ and then $\sC$ by $\sC'$,
that $\sC$ is strict.
It suffices then to show that if $\varepsilon = \nu \circ \mu(d)$ and
$\varepsilon'= \nu' \circ \mu'(d')$
are the counits, then the isomorphism
$\omega_{R,\varepsilon;R \otimes R'(d'),\I}$ of \eqref{e:adjiso} sends both legs of
\eqref{e:projsquare} to
\begin{equation}\label{e:RRRI}
\varepsilon' \circ ((f \circ \mu) \otimes R'(d')):
R \otimes R \otimes R'(d') \to \I.
\end{equation}
Since $\omega_{R,\varepsilon;R \otimes R'(d'),\I}$ is defined by tensoring on the left with
$R$ and then composing with $\varepsilon$, it sends the top right leg of \eqref{e:projsquare}
to the composite of $R \otimes \beta(d')$ with \mbox{$\varepsilon' \circ (f \otimes R'(d'))$},
by \eqref{e:transpose}, and hence to the composite of $f \otimes f \otimes R'(d')$ with
\[
\varepsilon' \circ (R' \otimes \mu'(d')) = \varepsilon' \circ (\mu' \otimes R'(d')),
\]
by bifunctoriality of the tensor product and the associativity of $\mu'$.
This composite coincides with \eqref{e:RRRI} by the compatibility $\mu' \circ (f \otimes f) = f \circ \mu$
of $f$ with $\mu$ and $\mu'$.
Similarly $\omega_{R,\varepsilon;R \otimes R'(d'),\I}$ sends the bottom left leg of
\eqref{e:projsquare} to \eqref{e:RRRI}.

\subsection{Poincar\'e duality theories}

Let $\sV$ be a cartesian monoidal category.
We denote by $\sV^*$ the symmetric monoidal category with objects
those of $\sV$, morphisms the isomorphisms of $\sV$,
and identities, compositions, tensor product, associativities
and symmetries those of $\sV$.
A product-preserving functor $E:\sV' \to \sV$ then induces a symmetric monoidal
functor $E^*:\sV'{}^* \to \sV^*$.
Any symmetric monoidal functor
$h$ from $\sV^{\mathrm{op}}$ to a $k$\nd pretensor category $\sC$
lifts canonically to a symmetric
monoidal functor from $\sV^{\mathrm{op}}$ to the category of commutative
algebras in $\sC$, with the multiplication of $h(X)$ given by composing
the structural isomorphism $h(X) \otimes h(X) \iso h(X \times X)$ of $h$
with $h(\Delta_X)$, where $\Delta_X:X \to X \times X$ is the diagonal,
and with the identity $\I \to h(X)$ the image of $X \to \I$.
If $\sV$ is equipped with a dimension function and $\sC$ has a structure of Tate
$k$\nd pretensor category,
there is associated to $h$ a symmetric monoidal functor
\[
h(-)(\dim -):(\sV^*)^{\mathrm{op}} \rightarrow \sC,
\]
which sends $X$ to $h(X)(\dim X)$ and $u:Y \iso X$
to $h(u)(\dim X) = h(u)(\dim Y)$.

\begin{defn}
Let $\sV$ be a cartesian monoidal category with a dimension function.
Then a \emph{$k$\nd linear Poincar\'e duality theory with source $\sV$} is
a triple $(\sM,h,\nu)$, where
$\sM$ is a Tate $k$\nd pretensor category,
$h$ is a symmetric monoidal functor from
$\sV^{\mathrm{op}}$ to $\sM$,
and $\nu$ is a monoidal natural transformation
$h(-)(\dim -) \rightarrow \I$ of symmetric monoidal functors
from $(\sV^*)^{\mathrm{op}}$ to $\sM$, such that
for every object $X$ of $\sV$ the pair $(h(X),\nu_X)$
is a Poincar\'e algebra of dimension $\dim X$ in $\sM$.
\end{defn}

Let $(\sM,h,\nu)$ be a Poincar\'e duality theory
with source $\sV$.
If $T:\sM \rightarrow \sM'$ is a Tate $k$\nd tensor functor,
then $(\sM',Th,T\nu)$ is a Poincar\'e duality theory
with source $\sV$, which we call the \emph{push forward of $(\sM,h,\nu)$
along $T$}.
If $E:\sV' \rightarrow \sV$ is a functor which preserves
products and dimensions, then $(\sM,hE^{\mathrm{op}},\nu (E^*)^{\mathrm{op}})$ is
a Poincar\'e duality theory with source $\sV'$,
which we call the \emph{pullback of
$(\sM,h,\nu)$ along $E$}.

\begin{defn}
Let $(\sM,h,\nu)$ and $(\sM',h',\nu')$ be $k$\nd linear
Poincar\'e duality theories with source $\sV$.
Then a \emph{morphism from $(\sM,h,\nu)$
to $(\sM',h',\nu')$} is  a Tate $k$\nd tensor functor
$T$ from $\sM$ to $\sM'$ such that $h' = Th$ and $\nu' = T\nu$.
\end{defn}

We thus have category of $k$\nd linear
Poincar\'e duality theories with source $\sV$,
with composition given by composition of $k$\nd tensor functors.
A symmetric monoidal functor $h$ from $\sV^{\mathrm{op}}$ to a Tate
$k$\nd pretensor category $\sC$ will be called \emph{taut} if every object
of $\sC$ can written uniquely as $h(X)(i)$ with $X \in \sV$ and $i \in \Z$,
and a Poincar\'e duality theory $(\sM,h,\eta)$ will be called taut if $h$ is taut.

\begin{prop}\label{p:tautrefl}
Let $(\sM,h,\nu)$ be a Poincar\'e duality theory with source $\sV$.
\begin{enumerate}
\item\label{i:tautreflU}
There exists a morphism $U:(\sM_0,h_0,\nu_0) \to (\sM,h,\nu)$
of Poincar\'e duality theories with $h_0$ taut and strict
and $U$ fully faithful.
\item\label{i:tautreflT}
If $U_1:(\sM_1,h_1,\nu_1) \to (\sM,h,\nu)$ and
$U_2:(\sM_2,h_2,\nu_2) \to (\sM,h,\nu)$ are morphisms of Poincar\'e duality
theories with $h_1$ taut and $U_2$ fully faithful,
then there exists a unique morphism
$T:(\sM_1,h_1,\nu_1) \to (\sM_2,h_2,\nu_2)$ of Poincar\'e duality
theories such that $U_1 = U_2T$.
\end{enumerate}
\end{prop}

\begin{proof}
(i)
Define $(\sM_0,h_0,\nu_0)$ as follows.
The objects of $\sM_0$ are pairs $(X,r)$ with $X$ an object of $\sV$
and $r$ an integer.
Set
\[
\sM_0((X,r),(Y,s)) = \sM(h(X)(r),h(Y)(s)).
\]
The identities and composition in $\sM_0$ are the same as those in $\sM$.
The underlying functor of $h$ then factors as $Uh_0$,
where $h_0:\sV^\text{op} \to \sM_0$ sends $X$  to $(X,0)$ and
coincides with $h$ on morphisms,
and $U:\sM_0 \to \sM$ is the $k$\nd linear functor which sends
$(X,r)$ to $h(X)(r)$ and is the identity on hom spaces.
Thus $U$ is fully faithful.
The tensor product of $(X,r)$ and $(Y,s)$ in $\sM_0$ is defined as
$(X \times Y,r+s)$.
If
\[
\psi_{X,Y}:h(X) \otimes h(Y) \iso h(X \times Y),
\]
is the structural isomorphism of $h$, write
\[
\rho_{(X,r),(Y,s)} = \psi_{X,Y}(r+s):U((X,r)) \otimes U((Y,s))
\iso U((X,r) \otimes (Y,s)).
\]
Since $U$ is fully faithful, the tensor product on morphisms
of $\sM_0$ is uniquely defined by requiring that the isomorphism
$\rho_{(X,r),(Y,s)}$ be natural in $(X,r)$ and $(Y,s)$.
Similarly, the associativities and symmetries of $\sM_0$ are uniquely defined
by requiring that the $\rho_{(X,r),(Y,s)}$ should be the structural
isomorphisms of a $k$\nd tensor functor $U$.
That the associativities and symmetries in $\sM_0$ so defined are
natural and satisfy the necessary compatibilities follows from
the corresponding properties in $\sM$ and the faithfulness
of $U$.
By construction, the $\rho_{(X,r),(Y,s)}$ then define a structure
of $k$\nd tensor functor on $U$.
Also $h_0:\sV^\text{op} \rightarrow \sM_0$ is a strict symmetric monoidal
functor such that $Uh_0 = h$, because $U$ is faithful and
$\rho_{h_0(X),h_0(Y)} = \psi_{X,Y}$.

The object $\I(r)$ of $\sM_0$ is $(\I,r)$.
That the relevant
associativities and symmetries are identities is easily checked
by applying $U$ and using the fact that $\rho_{(X,r),(Y,s)}$ coincides with
$\rho_{(X,0),(Y,0)}(r+s)$.
It is then clear that $U:\sM_0 \rightarrow \sM$ is a Tate $k$\nd tensor functor.
Finally $\nu_0$ is uniquely defined by requiring that $U\nu_0 = \nu$.
That $((X,0),(\nu_0)_X)$ is a Poincar\'e algebra of dimension $\dim X$ for every
$X$ then follows from the full faithfulness of $U$.
Thus $(\sM_0,h_0,\nu_0)$ is a Poincar\'e duality theory with source
$\sV$.
That $h_0$ and $U$ have the required properties is clear.

(ii) $T$ is determined uniquely on the object $h_1(X)(r)$ of $\sM_1$ by
\[
T(h_1(X)(r)) = h_2(X)(r),
\]
because it is required that $h_2 = Th_1$,
and uniquely on the morphism $f$ of $\sM_1$ by
\[
U_2(T(f)) = U_1(f)
\]
and the full faithfulness of $U_2$.
Also the structural isomorphism
\[
\varphi_{M,N}:T(M) \otimes T(N) \iso T(M \otimes N)
\]
of $T$ is uniquely determined by the requirement that if $\theta_i$
is the structural isomorphism of $U_i$ then
\[
(\theta_1)_{M,N} = U_2(\varphi_{M,N}) \circ (\theta_2)_{T(M),T(N)}.
\]
That $T$ so determined preserves identities and compositions follows
from the fact that $U_1$ and $U_2$ do and the faithfulness of $U_2$.
The naturality of $\varphi$, its compatibility with the associativities
and symmetries of $\sM_1$ and $\sM_2$,
and the condition that $\varphi_{M,\I(r)} = 1$ for every
$M$ and $r$, follow similarly.
Thus $T$ is a Tate $k$\nd tensor functor from $\sM_1$ to $\sM_2$.
By construction $U_1 = U_2T$.
Since $h_2$ and $Th_1$ coincide on objects and $U_2h_2 = h = U_2Th_1$, we have
$h_2 = Th_1$ by faithfulness of $U_2$.
Similarly $\nu_2 = T\nu_1$.
Thus $T$ is a morphism from $(\sM_1,h_1,\nu_1)$ to $(\sM_2,h_2,\nu_2)$.
\end{proof}

Let $(\sM,h,\nu)$ be Poincar\'e duality theory with source $\sV$.
Then for each $X$ in $\sV$ we have a Poincar\'e algebra $(h(X),\nu_X)$ of
dimension $\dim X$ in $\sM$.
For each $X$ and $i$, the object $h(X)(i)$ of $\sM$
thus has a canonical dual, given by
\[
h(X)(i)^\vee = h(X)(\dim X-i),
\]
together with a canonical counit $\varepsilon_X$, given by the composite
\[
h(X)(i) \otimes h(X)(i)^\vee = h(X) \otimes h(X)(\dim X)
\xrightarrow{\mu_X(\dim X)} h(X)(\dim X) \xrightarrow{\nu_X} \I,
\]
with $\mu_X :h(X) \otimes h(X) \to h(X)$ the multiplication of $h(X)$.
We have \mbox{$\I(i)^\vee = \I(-i)$} and $\varepsilon_\I = 1_\I$.
Given $Y$ and $j$, the canonical pairing
between $h(X \times Y)(i + j)$ and $h(X \times Y)(i + j)^\vee$
coincides, modulo the appropriate structural isomorphism
of $h$, with the tensor product
of those between $h(X)(i)$ and $h(X)(i)^\vee$ and between
$h(Y)(j)$ and $h(Y)(j)^\vee$.
Any morphism $(\sM,h,\nu) \to (\sM',h',\nu')$ of Poincar\'e duality theories
sends $(h(X)(i),\nu_X)$ to $(h'(X)(i),\nu'{}\!_X)$,
and hence preserves canonical duality pairings and transposes.

We associate as follows to each $k$\nd linear Poincar\'e duality theory
$(\sM,h,\nu)$ with source $\sV$ a $k$\nd linear
Chow theory
\[
\sM(\I,h(-)(\cdot))
\] with source $\sV$,
and to each morphism $T$ of Poincar\'e duality theories from
$(\sM,h,\nu)$  to $(\sM',h',\nu')$ a morphism of Chow theories
\[
T_{\I,h(-)(\cdot)}:\sM(\I,h(-)(\cdot)) \to \sM'(\I,h'(-)(\cdot)).
\]
The Chow theory $\sM(\I,h(-)(\cdot))$ assigns
to the object $X$ of $\sV$ the graded $k$\nd algebra
$\sM(\I,h(X)(\cdot))$
whose component of degree $i$ is $\sM(\I,h(X)(i))$,
with the identity $\I \to h(X)(0)$ the identity
of the algebra $h(X)$, and the product of $a:\I \to h(X)(i)$ and
\mbox{$b:\I \to h(X)(j)$} their tensor product in $\sM$ composed with the
$(i+j)$\nd fold twist of the multiplication of $h(X)$.
It assigns to the morphism $p:X \to Y$ of $\sV$ the homomorphism
of graded $k$\nd algebras
\[
p^*:\sM(\I,h(Y)(\cdot)) \to \sM(\I,h(X)(\cdot))
\]
that sends $b:\I \to h(Y)(j)$ to $h(p)(j) \circ b$, and the
homomorphism of graded $k$\nd modules of degree $\dim Y - \dim X$
\[
p_*:\sM(\I,h(X)(\cdot)) \to \sM(\I,h(Y)(\cdot))
\]
that sends $a:\I \to h(X)(i)$ to $h(p)^\vee(i - \dim X) \circ a$.
The morphism of Chow theories $T_{\I,h(-)(\cdot)}$  assigns to the
object $X$ of $\sV$ the homomorphism  $T_{\I,h(X)(\cdot)}$ of graded $k$\nd algebras
with component $T_{\I,h(X)(i)}$ of degree $i$.
That the $T_{\I,h(X)(\cdot)}$ are indeed homomorphisms of graded
$k$\nd algebras which are compatible as in
Definition~\ref{d:Chowmor} with pullback is clear, and their compatibility
with push forward follows from the fact that $T$ preserves
canonical transposes.
To see that $\sM(\I,h(-)(\cdot))$ satisfies the conditions of
Definition~\ref{d:Chow},
choose an $(\sM_0,h_0,\nu_0)$ and $U$
as in  Proposition~\ref{p:tautrefl}~\ref{i:tautreflU}.
Since $U$ is fully faithful, each $U_{\I,h_0(X)(\cdot)}$ is an isomorphism,
so that after replacing
$(\sM,h,\nu)$ by $(\sM_0,h_0,\nu_0)$ we may suppose that $h$ is strict.
That condition~\ref{i:Chowfun} of Definition~\ref{d:Chow}
is satisfied is clear.
That \ref{i:Chowproj} is satisfied follows by composing the tensor product of
$\I \to h(Y)(j)$ and $\I \to h(X)(i)$ with the twist by $i - \dim X$ of
\eqref{e:projsquare}, where $R = h(Y)$, $R' = h(X)$ and $f= h(p)$.
Condition \ref{i:Chowiso} is satisfied because for any
$p:X \iso Y$ in $\sV$ we have
\[
h(p)^\vee = h(p^{-1})(\dim X)
\]
by \eqref{e:transpose}, the fact that $h(p)$ is a morphism of algebras,
and the naturality of $\nu$.
For $a$ in $\sM(\I,h(X)(i))$ and $b$ in $\sM(\I,h(Y)(j))$,
it is easily checked that $a \otimes b$ in the sense
of Definition~\ref{d:Chow}~\ref{i:Chowtens} coincides with the
tensor product of the morphisms
$a:\I \to h(X)(i)$ and $b:\I \to h(Y)(j)$ in $\sM$.
Thus condition \ref{i:Chowtens} is satisfied because the canonical
transpose of $h(r \times s) = h(r) \otimes h(s)$ is
$h(r)^\vee \otimes h(s)^\vee$.

By assigning to each $(\sM,h,\nu)$ its associated Chow theory
$\sM(\I,h(-)(\cdot))$
and to each $T:(\sM,h,\nu) \to (\sM',h',\nu')$ its associated morphism
$T_{\I,h(-)(\cdot)}$ of Chow theories we obtain a functor
from $k$\nd linear Poincar\'e duality theories with source $\sV$ to
$k$\nd linear Chow theories with source $\sV$.
When $(\sM,h,\nu)$ is taut, $T_{\I,h(-)(\cdot)}$ is an isomorphism if and only if
$T$ is fully faithful.
For any functor $E:\sV' \rightarrow \sV$
which preserves products and dimensions,
passage to associated Chow theories and associated morphisms of Chow
theories commutes with pullback along $E$.

While the category of Poincar\'e duality theories defined
above will suffice for what follows, it is for some purposes inadequate.
For example an equivalence $\sV' \to \sV$ does not in general induce by
pullback an equivalence of categories of Poincar\'e duality theories.
Instead of the category of Poincar\'e duality theories with source $\sV$,
we may consider a $2$\nd category with the same objects.
A $1$\nd morphism from $(\sM,h,\nu)$ to $(\sM',h',\nu')$ is a pair
$(T,\varphi)$ with $T:\sM \to \sM'$ a Tate $k$\nd tensor functor and
\mbox{$\varphi:h' \iso Th$} a monoidal isomorphism such that
$\nu'{}\!_X = T(\nu_X) \circ \varphi_X(\dim X)$ for every $X$, and a
$2$\nd morphism from $(T_1,\varphi_1)$ to $(T_2,\varphi_2)$ is a tensor
isomorphism $\theta:T_1 \iso T_2$ such that $\theta_{\I(i)} = 1_{\I(i)}$
for every $i$ and
$\varphi_2 = \theta h \circ \varphi_1$.
Pullback along an equivalence $\sV' \to \sV$ then induces an equivalence
of $2$\nd categories.
Every $1$\nd morphism induces a morphism of Chow theories,
and $2$\nd isomorphic $1$\nd morphisms induce the same morphism of Chow theories.
Morphisms in the usual sense from $(\sM,h,\nu)$ to $(\sM',h',\nu')$ may be identified with
$1$\nd morphisms of the form $(T,\id_{h'})$.
The important case for what follows is that where $(\sM,h,\nu)$ is taut.
In this case it is enough to consider only morphisms in the usual sense from
$(\sM,h,\nu)$ to $(\sM',h',\nu')$, because they form a discrete skeleton of the category of
$1$\nd morphisms and $2$\nd morphisms between them.

\subsection{Universal Poincar\'e duality theories}\label{ss:uniPoin}

Let $(\sM,h,\nu)$ be Poincar\'e duality theory with source $\sV$.
Then for each $X$, $Y$, $i$ and $j$ we have the canonical isomorphism
\[
\omega_{h(X)(i),\varepsilon_X;h(Y)(j)}:
\sM(\I,h(X) \otimes h(Y)(\dim X+j-i)) \iso \sM(h(X)(i),h(Y)(j))
\]
of \eqref{e:adj1iso}, where $\varepsilon_X$ is the counit of the canonical duality
pairing between $h(X)(i)$ and $h(X)(\dim X -i)$.
Let $C$ be a Chow theory with source $\sV$, and $\gamma$ be a morphism from
$C$ to the Chow theory $\sM(\I,h(-)(\cdot))$ associated to $(\sM,h,\nu)$.
Then we have a homomorphism
\[
\gamma_{X,i}:C^i(X) \to \sM(\I,h(X)(i))
\]
for every $X$ and $i$.
Composing $\gamma_{X \times Y,\dim X +j-i}$ with the isomorphism
\[
\sM(\I,h(X \times Y)(\dim X+j-i)) \iso \sM(\I,h(X) \otimes h(Y)(\dim X+j-i))
\]
defined by the inverse of the structural isomorphism for $h$, and then
with $\omega_{h(X)(i),\varepsilon_X;h(Y)(j)}$, gives
a homomorphism
\begin{equation}\label{e:PDhomChow}
\gamma_{X,Y,i,j}:C^{\dim X+j-i}(X \times Y) \to  \sM(h(X)(i),h(Y)(j))
\end{equation}
for every $X$, $Y$, $i$ and $j$.
We have $\gamma_{X,i} = \gamma_{\I,X,0,i}$.
When $\gamma$ is an isomorphism, each $\gamma_{X,Y,i,j}$ is an isomorphism.
If $C'$ is a Chow theory and $(\sM',h',\nu')$ a Poincar\'e duality theory
with source $\sV$, and $\gamma':C' \to \sM'(\I,h'(-)(\cdot))$  and $\varphi:C \to C'$
are morphisms of Chow theories and $T:(\sM,h,\nu) \to (\sM',h',\nu')$ a morphism of
Poincar\'e duality theories such that
\begin{equation}\label{e:thetacircphi}
\gamma' \circ \varphi = T_{\I,h(-)(\cdot)} \circ \gamma,
\end{equation}
then for every $X$, $Y$, $i$ and $j$ we have
\begin{equation}\label{e:gammaT}
\gamma'{}\!_{X,Y,i,j} \circ \varphi_{X \times Y,\dim X+j-i} =
T_{h(X)(i),h(Y)(j)} \circ \gamma_{X,Y,i,j}.
\end{equation}
This can be seen by combining three squares, with the left
commutative by \eqref{e:thetacircphi}, the middle by the fact that $h' = Th$,
and the right by the fact that, modulo the relevant structural
isomorphism, $T$ sends $\varepsilon_X$ to the counit
\mbox{$h'(X) \otimes h'(X)(\dim X) \to \I$}.
Let $\sV'$ be a cartesian monoidal category with a dimension function
and $E:\sV' \to \sV$ be functor which preserves products and dimensions.
Then for every $X'$, $Y'$, $i$ and $j$ the homomorphism
$(\gamma E)_{X',Y',i,j}$ is the composite with $\gamma_{E(X'),E(Y'),i,j}$
of the isomorphism from  $C^{\dim X'+j-i}(E(X' \times Y'))$ to
$C^{\dim X'+j-i}(E(X') \times E(Y'))$
induced by the structural isomorphism of $E$.

\begin{prop}\label{p:PDhom}
Let $C$ be a Chow theory and $(\sM,h,\nu)$ a Poincar\'e duality theory
with source $\sV$, and $\gamma:C \to \sM(\I,h(-)(\cdot))$ be a
morphism of Chow theories.
\begin{enumerate}
\item\label{i:PDhomtrans}
$\gamma_{X,Y,i,j}(\alpha)^\vee = \gamma_{Y,X,\dim Y - j,\dim X -i}(\sigma^*(\alpha))$
for every $\alpha$,
where $\sigma$ is the symmetry \mbox{$Y \times X \iso X \times Y$}.
\item\label{i:PDhomtens}
$\gamma_{X,Y,i,j}(\alpha) \otimes \gamma_{X',Y',i',j'}(\alpha') =
\xi^{-1} \circ \gamma_{X \times X',Y \times Y',i + i',j + j'}(\rho^*(\alpha \otimes \alpha'))
\circ \zeta$ for every $\alpha$ and $\alpha'$,
where $\zeta$ is the $(i + i')$\nd fold twist of the structural isomorphism
$h(X) \otimes h(X') \iso h(X \times X')$ and $\xi$ is the $(j + j')$\nd fold twist
of the structural isomorphism
$h(Y) \otimes h(Y') \iso h(Y \times Y')$,
and $\rho$ is the symmetry
$(X \times X') \times (Y \times Y') \iso (X \times Y) \times (X' \times Y')$.
\item\label{i:PDhomcomp}
$\gamma_{Y,Z,j,l}(\beta) \circ \gamma_{X,Y,i,j}(\alpha) =
\gamma_{X,Z,i,l}((\pr_{13})_*((\pr_{12})^*(\alpha).(\pr_{23})^*(\beta)))$
for every $\alpha$ and $\beta$, where $\pr_{12}$, $\pr_{23}$ and $\pr_{13}$ are the
projections from $X \times (Y \times Z)$.
\item\label{i:PDhomgraph}
$h(p) = \gamma_{X,Y,0,0}((p,1_Y)_*(1))$ for every $p:Y \to X$.
\end{enumerate}
\end{prop}

\begin{proof}
By Proposition~\ref{p:tautrefl}~\ref{i:tautreflU}
there is a $U:(\sM_0,h_0,\nu_0) \to (\sM,h,\nu)$ with $h_0$ strict and
$U$ fully faithful.
Then $U_{\I,h_0(-)(\cdot)}$ is an isomorphism,
so that $\gamma = U_{\I,h_0(-)(\cdot)} \circ \gamma_0$ for some
$\gamma_0:C \to \sM_0(\I,h_0(-)(\cdot))$.
By \eqref{e:gammaT} with $\varphi$ the identity, it is enough to prove
the required results with $(\sM,h,\nu)$ and $\gamma$ replaced by
$(\sM_0,h_0,\nu_0)$ and $\gamma_0$.
Thus we may assume that $h$ is strict.
Then we have
\begin{equation}\label{e:omegagamma}
\gamma_{X,Y,i,j} = \omega_{h(X)(i),\varepsilon_X;h(Y)(j)} \circ \gamma_{X \times Y,\dim X+j-i}
\end{equation}
for every $X$, $Y$, $i$ and $j$.

The symmetry $\widetilde{\sigma}$ that interchanges the two factors of the tensor product of
$h(X)(\dim X -i)$ and $h(Y)(j)$ is given by $h(\sigma)(\dim X + j -i)$,
because $h$ is strict.
Since $\gamma$ is a morphism of Chow theories, it follows that
\[
\widetilde{\sigma} \circ \gamma_{X \times Y,\dim X + j-i}(\alpha) =
\gamma_{Y \times X,\dim X + j-i}(\sigma^*(\alpha))
\]
Applying $\omega_{h(Y)(\dim Y -j),\varepsilon_Y;h(X)(\dim X -i)}$ and using
\eqref{e:omegagamma} and \eqref{e:adjtrans} now gives \ref{i:PDhomtrans}.

By strictness of $h$, the tensor product in the Chow theory
$\sM(\I,h(-)(\cdot))$ is given by taking the tensor product in $\sM$ of
the relevant morphisms with source $\I$ in $\sM$.
Further if we write $d$ for $\dim X + j - i$ and $d'$ for $\dim X' + j' - i'$,
the symmetry $\widetilde{\rho}$ that interchanges the middle two factors
of the tensor product of $h(X)(i)^\vee \otimes h(Y)(j)$ with
$h(X')(i')^\vee \otimes h(Y')(j')$ is $h(\rho)(d+d')$.
Since $\gamma$ is a morphism of Chow theories, we thus have
\begin{align*}
\widetilde{\rho} \circ
(\gamma_{X \times Y,d}(\alpha) \otimes \gamma_{X' \times Y',d'}(\alpha'))
&= \widetilde{\rho} \circ
\gamma_{(X \times Y) \times (X' \times Y'),d+d'}(\alpha \otimes \alpha')  \\
&= \gamma_{(X \times X') \times (Y \times Y'),d+d'}(\rho^*(\alpha \otimes \alpha')).
\end{align*}
Applying $\omega_{h(X \times X')(i + i'),\varepsilon_{X \times X'};h(Y \times Y')(j + j')}$
and using \eqref{e:omegagamma}, \eqref{e:adjtens},
and the fact that $\zeta$ and $\xi$ are identities,
now gives \ref{i:PDhomtens}.

By definition, the counit $\varepsilon_Y$ is the composite of
$h(\Delta_Y)(\dim Y)$ with the transpose $\nu_Y$ of the morphism $\I \to h(Y)$
defined by $Y \to \I$.
If we write $d$ for $\dim X + j - i$ and $e$ for $\dim Y + l - j$,
it follows that
\[
h(X)(i)^\vee \otimes (\varepsilon_Y \otimes h(Z)(l)) =
h(\pr_{13})^\vee(-\dim Z + l - i) \circ h(X \times (\Delta_Y \times Z))(d + e) = \theta,
\]
say.
With $\widetilde{\lambda}$ the appropriate associativity in $\sM$,
we then have
\begin{align*}
\theta \circ \widetilde{\lambda} \circ
(\gamma_{X \times Y,d}(\alpha) \otimes \gamma_{Y \times Z,e}(\beta))
&=
\theta \circ
\gamma_{X \times ((Y \times Y) \times Z),d+e}((\pr_{12})^*(\alpha).(\pr_{34})^*(\beta)) \\
&=
\gamma_{X \times Z,\dim X + l - i}((\pr_{13})_*((\pr_{12})^*(\alpha).(\pr_{23})^*(\beta))),
\end{align*}
by the definition of $\alpha \otimes \beta$ as in Definition~\ref{d:Chow}~\ref{i:Chowtens}
and the fact that $\widetilde{\lambda} = h(\lambda)(d+e)$ for the appropriate associativity
$\lambda$ in $\sV$.
Applying $\omega_{h(X)(i),\varepsilon_X,h(Z)(l)}$ and using \eqref{e:omegagamma}
and \eqref{e:adjtenscomp} now gives \ref{i:PDhomcomp}.

If $\iota:\I \to h(Y)$ is the unit, then since
$(p,1_Y) = (p \times Y) \circ \Delta_Y$, we have
\[
(h(p) \otimes h(Y)^\vee)^\vee \circ \varepsilon_Y{}\!^\vee =
h((p,1_Y))^\vee(- \dim Y) \circ \iota =
\gamma_{X \times Y,\dim X}((p,1_Y)_*(1)).
\]
Applying $\omega_{h(X),\varepsilon_X;h(Y)}$ and using \eqref{e:omegagamma}
and \eqref{e:adjgraph} now gives \ref{i:PDhomgraph}.
\end{proof}

Let $(\sM,h,\nu)$, $C$ and $\gamma$ be as in Proposition~\ref{p:PDhom}.
Then it is clear from the definition of $\gamma_{X,Y,i,j}$ that
for every $\alpha$ and $d$ we have
\begin{equation}\label{e:PDhomtwist}
\gamma_{X,Y,i,j}(\alpha)(d) = \gamma_{X,Y,i + d,j + d}(\alpha).
\end{equation}
In fact \eqref{e:PDhomtwist} follows from Proposition~\ref{p:PDhom},
because $\gamma_{\I,\I,0,0}(1)$ is the identity of $\I = h(\I)$ by \ref{i:PDhomgraph}
and $\gamma_{\I,\I,d,d}(1)$ is an idempotent endomorphism of $\I(d)$
by \ref{i:PDhomcomp} which is invertible by \ref{i:PDhomtens} and hence the identity.
For any $Z$ and $p:Y \to X$ and  $\beta \in C(Y \times Z)$, we have
\[
(\pr_{12})^*((p,1_Y)_*(1)).(\pr_{23})^*(\beta) = ((p,1_Y) \times Z)_*(\beta).
\]
in $C((X \times Y) \times Z)$,
by \ref{d:Chow}~\ref{i:Chowproj} with $(p,1_Y) \times Z$ for $p$  and $x = 1$ and
$y = (\pr_{23})^*(\beta)$,
and \ref{d:Chow}~\ref{i:Chowtens}.
Proposition~\ref{p:PDhom}~\ref{i:PDhomcomp} with $\alpha = (p,1_Y)_*(1)$ and
Proposition~\ref{p:PDhom}~\ref{i:PDhomgraph} thus give
\begin{equation}\label{e:PDhomcomppush}
\gamma_{Y,Z,j,l}(\beta) \circ h(p)(j) = \gamma_{X,Z,j,l}((p \times Z)_*(\beta)).
\end{equation}
Similarly for any $X$ and $q:Z \to Y$ and $\alpha \in C(X \times Y)$ we have
\[
(\pr_{12})^*(\alpha).(\pr_{23})^*((q,1_Z)_*(1)) =
(X \times (q,1_Z))_*((X \times q)^*(\alpha))
\]
in $C(X \times (Y \times Z))$.
Proposition~\ref{p:PDhom}~\ref{i:PDhomcomp} with $\beta = (q,1_Z)_*(1)$ thus gives
\begin{equation}\label{e:PDhomcomppull}
h(q)(j) \circ \gamma_{X,Y,i,j}(\alpha) = \gamma_{X,Z,i,j}((X \times q)^*(\alpha)).
\end{equation}
For fixed $i$ and $j$, we may regard $C^{\dim X + j - i}(X \times Y)$ as
a bifunctor of $X$ and $Y$, covariant in $X$ and contravariant in $Y$.
Given for example $a:X \to X'$ in $\sV$, the homomorphism
$C^{\dim X +j-i}(X \times Y) \to C^{\dim X'+j-i}(X' \times Y)$ associated to $a$
is $(a \times Y)_*$.
Then \eqref{e:PDhomcomppush} and \eqref{e:PDhomcomppull} show that
the homomorphism $\gamma_{X,Y,i,j}$ of \eqref{e:PDhomChow} is natural
in $X$ and $Y$.

\begin{prop}\label{p:tautff}
Let $(\sM,h,\nu)$ be a taut Poincar\'e duality theory with source~$\sV$.
Then for any Poincar\'e duality theory $(\sM',h',\nu')$
with source $\sV$ and morphism
$\varphi:\sM(\I,h(-)(\cdot)) \to \sM'(\I,h'(-)(\cdot))$
of Chow theories, there is a unique morphism
$T:(\sM,h,\nu) \to (\sM',h',\nu')$ of Poincar\'e duality theories
such that $T_{\I,h(-)(\cdot)} = \varphi$.
\end{prop}

\begin{proof}
By Proposition~\ref{p:tautrefl}~\ref{i:tautreflU} and \ref{i:tautreflT},
we may suppose that $h$ and $h'$ are strict.
Any $T:(\sM,h,\nu) \to (\sM',h',\nu')$ must send $h(X)(i)$
in $\sM$ to $h'(X)(i)$ in $\sM'$.
Further $T$ must be a strict $k$\nd tensor functor,
because it is a Tate $k$\nd tensor functor
with $Th = h'$ and $h$ is taut.
If we write $C$ for $\sM(\I,h(-)(\cdot))$ and $C'$ for
$\sM'(\I,h'(-)(\cdot))$, then we have identity isomorphisms of
Chow theories
$\gamma:C \iso \sM(\I,h(-)(\cdot))$  and
$\gamma':C' \iso \sM'(\I,h'(-)(\cdot))$.
Then $T_{\I,h(-)(\cdot)} = \varphi$ if and only if
\eqref{e:thetacircphi} holds,
and it has been seen that \eqref{e:thetacircphi} holds if and only if
\eqref{e:gammaT} holds.
Since the $\gamma_{X,Y,i,j}$ are isomorphisms, a $T$
with $T_{\I,h(-)(\cdot)} = \varphi$ is thus unique if it exists.
To show that such a $T$ exists, define $T$ on objects as sending
$h(X)(i)$ to $h'(X)(i)$ and on morphisms by requiring that
\eqref{e:gammaT} should hold.
That $T$ so defined is a strict $k$\nd tensor functor with $Th = h'$
follows from the strictness of $h$ and $h'$ together
with Proposition~\ref{p:PDhom}~\ref{i:PDhomtens}, \ref{i:PDhomcomp}
and \ref{i:PDhomgraph}.
That it preserves canonical transposes, and hence sends the
transpose $\nu_X:h(X)(\dim X) \to \I$ in $\sM$ of the identity
$\I \to h(X)$ to $\nu'{}\!_X$ in $\sM'$, then follows from
Proposition~\ref{p:PDhom}~\ref{i:PDhomtrans}.
\end{proof}

Consider the functor $\sC\sH$ from Poincar\'e duality theories with source
$\sV$ to Chow theories with source $\sV$ that sends
$(\sM,h,\nu)$ to  $\sM(\I,h(-)(\cdot))$.
By Proposition~\ref{p:tautff}, the restriction $\sC\sH_0$ of $\sC\sH$ to the
full subcategory of taut Poincar\'e duality theories is fully faithful.
By Proposition~\ref{p:tautrefl}, there exists for every $(\sM,h,\nu)$ a
morphism $U$ which is universal among morphisms with target  $(\sM,h,\nu)$
and taut source, and $\sC\sH$ sends $U$ to an isomorphism.
Thus $\sC\sH$ factors up to isomorphism as a right adjoint to the embedding
of taut Poincar\'e duality theories followed by $\sC\sH_0$.
To prove Theorem~\ref{t:Poinfunctor} below, it is thus enough to show that
$\sC\sH$ is essentially surjective.
Recall that a left adjoint is fully faithful if and only if the unit of the
adjunction is an isomorphism.

\begin{thm}\label{t:Poinfunctor}
The functor from the category of Poincar\'e duality theories with source $\sV$
to the category of Chow theories with source $\sV$ that sends $(\sM,h,\nu)$ to
$\sM(\I,h(-)(\cdot))$ and $T:(\sM,h,\nu) \to (\sM',h',\nu')$
to $T_{\I,h(-)(\cdot)}$ has a fully faithful left adjoint,
with essential image consisting of the taut Poincar\'e duality theories.
\end{thm}

\begin{proof}
Let $C$ be a Chow theory with source $\sV$.
By the above remarks, it is enough to construct an $(\sM,h,\nu)$
with $\sM(\I,h(-)(\cdot))$ isomorphic to $C$.
There exists a symmetric monoidal equivalence $E:\sV \to \sV'$ with $\sV'$ strict.
Then $\sV'$ is cartesian monoidal and has a dimension function
such that $E$ preserves dimensions, and $C$ is isomorphic to the pullback
along $E$ of a Chow theory $C'$ on $\sV'$.
Replacing $\sV$ and $C$ by $\sV'$ and $C'$, we may suppose that $\sV$ is strict.

Define $(\sM,h,\nu)$ as follows.
An object of $\sM$ is a symbol $h(X)(i)$ with $X$ an object of $\sV$ and
$i$ an integer.
The tensor product is defined on objects of $\sM$ by
\[
h(X)(i) \otimes h(Y)(j) = h(X \times Y)(i + j),
\]
with $h(\I)(0)$ the identity of $\sM$.
The functor $h$ is defined on objects by
\[
h(X) = h(X)(0).
\]
A morphism from $h(X)(i)$ to $h(Y)(j)$ in $\sM$ is a symbol
$\gamma_{X,Y,i,j}(\alpha)$ with $\alpha$ an element of $C^{\dim X - i + j}(X \times Y)$,
and the $k$\nd module structure on $\sM(h(X)(i),h(Y)(j))$ is that
for which the map $\alpha \mapsto \gamma_{X,Y,i,j}(\alpha)$ is an isomorphism of
$k$\nd modules.
There is then a canonical isomorphism from this hom-space to that obtained
by replacing $i$ by $i + d$ and $j$ by $j + d$, which we write as $a \mapsto a(d)$
so that \eqref{e:PDhomtwist} holds.
The tensor product of morphisms
of $\sM$, composition of $\sM$, and action of $h$ on morphisms of $\sV$,
are defined by requiring that the equalities of
Proposition~\ref{p:PDhom}~\ref{i:PDhomtens} (with $\zeta$ and $\xi$ identities),
\ref{i:PDhomcomp}, and \ref{i:PDhomgraph}, should respectively hold.
The identity of $h(X)(i)$ is defined as $h(1_X)(i)$, and the symmetry that
interchanges the two factors in $h(X)(i) \otimes h(Y)(j)$ as $h(\sigma)(i+j)$
with $\sigma$ the symmetry that interchanges the two factors in $Y \times X$.
With these definitions, both \eqref{e:PDhomcomppush} and \eqref{e:PDhomcomppull} hold,
because they follow as above from the equalities of
Proposition~\ref{p:PDhom}~\ref{i:PDhomcomp}
and \ref{i:PDhomgraph}.

That the composition is associative follows from the fact that given
morphisms $a = \gamma_{X,Y,i,j}(\alpha)$, $b = \gamma_{Y,Z,j,l}(\alpha)$ and
$z = \gamma_{Z,W,l,m}(\zeta)$
in $\sM$, the composites $(z \circ b) \circ a$ and
$z \circ (b \circ a)$ both coincide with $\gamma_{X,W,i,m}(\xi)$,
with $\xi$ the image under $(\pr_{14})_*$ of
\begin{equation}\label{e:pr14image}
(\pr_{12})^*(\alpha).(\pr_{23})^*(\beta).(\pr_{34})^*(\zeta),
\end{equation}
where the projections are from $X \times Y \times Z \times W$.
In the case of $z \circ (b \circ a)$ for example, factoring $\pr_{14}$ as
\[
X \times Y \times Z \times W \xrightarrow{\pr_{134}}
X \times Z \times W \xrightarrow{\pr_{13}}
X \times W,
\]
and noting that by \ref{d:Chow}~\ref{i:Chowproj} the image of \eqref{e:pr14image}
under $(\pr_{134})_*$ is
\begin{equation}\label{e:pr124image}
(\pr_{134})_*((\pr_{12})^*(\alpha).(\pr_{23})^*(\beta)).
(\pr_{23})^*(\zeta)
\end{equation}
shows that $\xi$ is the image of \eqref{e:pr124image}
under $(\pr_{13})_*$, so that $\gamma_{X,W,i,m}(\xi) = z \circ (b \circ a)$
by \ref{d:Chow}~\ref{i:Chowtens} applied to $\pr_{13} \times W = \pr_{134}$.
That $h(1_X)(i)$ is indeed the identity of $(X,i)$ follows from
\eqref{e:PDhomcomppush} and \eqref{e:PDhomcomppull}.
Thus $\sM$ is a category.
We have
\begin{equation}\label{e:bracketcomp}
h(q)(j) \circ h(p)(j) = h(p \circ q)(j)
\end{equation}
for $p:Y \to X$ and $q:Z \to Y$,
by \eqref{e:PDhomcomppush} with $\beta = (q,1_Z)_*(1)$.
The assignments $X \mapsto h(X) = h(X)(0)$ and $p \mapsto h(p)$ thus define a functor
$\sV^\mathrm{op} \to \sM$.
The bilinearity of the tensor product of $\sM$ is immediate,
and the bifunctorialty follows from \ref{d:Chow}\ref{i:Chowtens}.
The naturality of the symmetries follows from \eqref{e:PDhomcomppush} and
\eqref{e:PDhomcomppull}, together with condition
\ref{d:Chow}\ref{i:Chowiso}.
We thus obtain on $\sM$ a structure of strict $k$\nd pretensor
category: the symmetries satisfy the required compatibilities by \eqref{e:bracketcomp}.
It is easily checked that $h$ is a strict
symmetric monoidal functor, and that $\sM$ has a structure of Tate
$k$\nd pretensor category with $\I(i) = h(\I)(i)$.
Then $h(X)(i)$ is the $i$\nd fold twist of $h(X)$
and the $d$\nd fold twist sends $a:h(X)(i) \to h(Y)(j)$
to the $a(d)$ already defined.

We have a duality pairing between $h(X)$ and $h(X)(\dim X)$ whose
unit $\eta_X$ is $\gamma_{\I,X \times X,0,\dim X}((\Delta_X)_*(1))$
and whose counit $\varepsilon_X$ is $\gamma_{X \times X,\I,\dim X,0}((\Delta_X)_*(1))$.
Indeed $(\varepsilon_X \otimes h(X)) \circ (h(X) \otimes \eta_X)$ is
\[
\gamma_{X,X,0,0}((\pr_{15})_*(\delta_{12}.\delta_{23}.\delta_{34}.\delta_{45}))
= \gamma_{X,X,0,0}((\Delta_X)_*(1)),
\]
where $\delta_{rs} = (\pr_{rs})^*((\Delta_X)_*(1))$ with the projections
from $X \times X \times X \times X \times X$.
This gives one triangular identity, and the other is similar.
Write $\nu_X$ for the morphism  $\gamma_{X,\I,\dim X,0}(1)$
from $h(X)(\dim X)$ to $\I$,
Then since $h(\Delta_X)$ is the multiplication of $h(X)$ and since
$\nu_X \circ h(\Delta_X)(\dim X) = \varepsilon_X$ by \eqref{e:PDhomcomppush},
it follows that $(h(X),\nu_X)$ is a Poincar\'e algebra of dimension
$\dim X$ in $\sM$.
Clearly $\nu_\I = 1_\I$ and $\nu_{X \times Y} = \nu_X \otimes \nu_Y$
for every $X$ and $Y$,
while $\nu_X = \nu_Y \circ h(p)(\dim X)$ for an isomorphism $p:Y \iso X$
by \eqref{e:PDhomcomppush} and \ref{d:Chow}~\ref{i:Chowiso}.
This shows that $(\sM,h,\nu)$ is a Poincar\'e duality theory with source $\sV$.

We show finally that the $k$\nd isomorphisms
$\alpha \mapsto \gamma_{\I,X,0,i}(\alpha)$ define an isomorphism of Chow theories
from $C$ to $\sM(\I,h(-)(\cdot))$.
That they define an isomorphism of graded algebras
$C(X) \iso \sM(\I,h(X)(\cdot))$ for each $X$ is clear from the equality
of Proposition~\ref{p:PDhom}~\ref{i:PDhomtens} and
\eqref{e:PDhomcomppull} with $q = \Delta_X$.
The compatibility with pullbacks follows from \eqref{e:PDhomcomppull} with $X = \I$.
The transpose of $\gamma_{X,\I,i,0}(\beta)$ from $h(X)(i)$ to $\I$
is $\gamma_{\I,X,0,\dim X - i}(\beta)$ from $\I$ to $h(X)(\dim X - i)$, because
$\gamma_{X,\I,i,0}(\beta)$ coincides with
\[
\gamma_{X,\I,i,0}((\pr_1)_*(\delta_{12}.(\pr_3)^*(\beta).\delta_{23})) =
\varepsilon_X \circ (h(X)(i) \otimes \gamma_{\I,X,0,\dim X - i}(\beta)),
\]
where $\delta_{rs} = (\pr_{rs})^*((\Delta_X)_*(1))$ with the projections
from $X \times X \times X$.
Taking the transpose of \eqref{e:PDhomcomppush} with $Z = \I$ thus gives the
compatibility with push forward.
\end{proof}

By Proposition~\ref{p:tautrefl}, the left adjoint of Theorem~\ref{t:Poinfunctor}
may be chosen so that the for every $(\sM,h,\nu)$ in its image, $h$ is strict.
We fix such a left adjoint and its unit,
and call the image of the Chow theory $C$
the \emph{Poincar\'e duality theory associated to $C$}, and similarly
for a morphism of Chow theories.
The morphism of Poincar\'e duality theories associated to any morphism of Chow theories
is then a strict $k$\nd tensor functor.
The Poincar\'e duality theory $(\sM,h,\nu)$ associated to $C$ comes equipped with a
universal morphism $\gamma:C \iso \sM(\I,h(-)(\cdot))$, which is an isomorphism
such that for any $(\sM',h',\nu')$ and $\gamma':C \to \sM'(\I,h'(-)(\cdot))$
there is a unique morphism $T$ from $(\sM,h,\nu)$ to $(\sM',h',\nu')$
with \mbox{$\gamma' = T_{\I,h(-)(\cdot)} \circ \gamma$}.
Further $T:\sM \to \sM'$ is faithful (resp.\ full) if
and only if every component $\gamma'{}\!_X$ is injective (resp.\ surjective).

Given $(\sM,h,\nu)$, any $\gamma:C \to \sM(\I,h(-)(\cdot))$ induces a map
$\sJ \mapsto J$ from the set of tensor ideals of $\sM$ to the set of ideals of $C$,
with $J(X)$ the inverse image under $\gamma$ of $\sJ(\I,h(X)(\cdot))$.
If $(\sM,h,\nu)$ is the Poincar\'e duality theory associated to $C$ and
$\gamma$ is the universal morphism, then $\sJ \mapsto J$ is bijective:
its inverse sends $J$ to the kernel of the morphism of Poincar\'e duality theories
associated to the projection $C \to C/J$.
Thus we may identify the Poincar\'e duality theory associated to $C/J$ with the
push forward of $(\sM,h,\nu)$ along the projection $\sM \to \sM/\sJ$.
More generally, for any $(\sM,h,\nu)$ and isomorphism $\gamma:C \iso \sM(\I,h(-)(\cdot))$,
the map $\sJ \mapsto J$ is surjective,
because tensor ideals can be extended from full $k$\nd tensor subcategories.

\section{Kimura categories}\label{s:Kimura}

\emph{In this section $k$ is a field of characteristic $0$}.

\subsection{Positive and negative objects}\label{ss:posneg}

Let $M$ be an object in a $k$\nd tensor category.
For each integer $d \ge 0$ we have a $k$\nd homomorphism from $k[\mathfrak{S}_d]$ to
$\End(M^{\otimes d})$ which assigns to an element of the symmetric group $\mathfrak{S}_d$
the associated symmetry of $M^{\otimes d}$.
Then the respective images $s_n$ and $a_n$ in $\End_{\sC}(M^{\otimes n})$ of the
symmetrising and antisymmetrising idempotents of $k[\mathfrak{S}_n]$ are
idempotent endomorphisms of $M^{\otimes n}$.
The image of $s_n$ will be denoted by $S^n M$ and the image of $a_n$ by $\bigwedge^n M$.
Thus $S^n M$ and $\bigwedge^n M$ are direct summands of $M^{\otimes n}$, and they
may be regarded as functors of $M$.
We have $\bigwedge^m M = 0$ if and only if $a_m = 0$,
and in that case  $\bigwedge^n M = 0$
for every $n \ge m$, because the ideal generated by the antisymmetriser of
$k[\mathfrak{S}_m] \subset k[\mathfrak{S}_n]$ contains the antisymmetriser of
$k[\mathfrak{S}_n]$.
Similarly if $S^m M = 0$ then $S^n M = 0$ for every $n \ge m$.
An object $M$ in a $k$\nd tensor category will be called \emph{positive}
(resp.\ \emph{negative}) if it is dualisable and if $\bigwedge^m M = 0$
(resp.\ $S^m M = 0$) for some $m$.
Suppose that $M$ is dualisable of rank $r$.
Then the contraction with respect to the last factor $M$ of a symmetry $\sigma$ permuting
the factors of $M^{\otimes (n+1)}$ is $r\sigma_0$ when $\sigma = \sigma_0 \otimes M$
leaves the last factor $M$ fixed, and otherwise is the symmetry of $M^{\otimes n}$
given by omitting $n+1$ from the cycle containing it in the permutation
that defines $\sigma$.
Thus the contraction with respect
to the last factor $M$ of $a_{n+1}$ is $(r-n)/(n+1)a_n$.
Contracting $n+1$ times shows that if $a_{n+1} = 0$ then $\binom{r}{n+1} = 0$.
Hence if $M$ is positive then $r = \sum_i r_i e_i$
with the $r_i$ integers $\ge 0$ and the  $e_i$ idempotents in $\End(\I)$.
Further if $M$ is positive and $r$ is an integer,
then $r$ is the least integer $\ge 0$ for which $\bigwedge^{r+1} M$ is $0$.
Similarly if $M$ is negative then $r = \sum_i r_i e_i$
with the $r_i$ integers $\le 0$ and the  $e_i$ idempotents,
and if $M$ is negative and $r$ is an integer,
then $r$ is the greatest integer $\le 0$ for which $S^{-r+1} M$ is $0$.
In particular if $M$ is both positive and negative then $M = 0$.

By a $k$\nd group we mean an affine group scheme over $k$, and
by a $k$\nd subgoup of a $k$\nd group a closed group subscheme.
If $G$ is a $k$\nd group, the category of $G$\nd modules will be written $\REP_k(G)$.
It has an structure of $k$\nd tensor category given by the tensor product
of $G$\nd modules over $k$.
An algebra in $\REP_k(G)$ is the same as a $G$\nd algebra,
i.e.\ an algebra over $k$ with a structure of $G$\nd module
such that the points of $G$ act by algebra isomorphisms.
The full $k$\nd tensor subcategory of $\REP_k(G)$ consisting of the
finite dimensional $G$\nd modules will be written $\Rep_k(G)$.
Then $\REP_k(G)$ is an ind-completion of $\Rep_k(G)$.
A finite dimensional $G$\nd module will also be called a representation of $G$.
For the following classical result, see for example \cite{FulHar}, Theorem~6.3.

\begin{lem}\label{l:symgroup}
Denote by $V$ the standard representation of $GL_n$ over $k$.
Then the homomorphism from
$k[\mathfrak{S}_r]$ to $\End_{GL_n}(V^{\otimes r})$
that assigns to each element in $\mathfrak{S}_r$ the
associated symmetry of $V^{\otimes r}$
is surjective, with kernel $0$ for $r \le n$ and the ideal
generated by the antisymmetriser of
$k[\mathfrak{S}_{n+1}] \subset k[\mathfrak{S}_r]$
for $r>n$.
\end{lem}

The following consequence of Proposition~\ref{p:freerigid} and Lemma~\ref{l:symgroup}
is well known.
Since the treatment in \cite{O} is rather sketchy, and there seems to be no other suitable
reference, we include a proof here.

\begin{thm}\label{t:GL}
Let $(n_\lambda)_{\lambda \in \Lambda}$
be a family of non-negative integers.
Denote by $G$ the $k$\nd group $\prod_{\lambda \in \Lambda}GL_{n_\lambda}$
and by $V_\lambda$ the standard representation of the factor $GL_{n_\lambda}$ of $G$.
Then for any $k$\nd tensor category $\sC$ and family
$(M_\lambda)_{\lambda \in \Lambda}$ of objects in $\sC$ with $M_\lambda$ positive
of rank $n_\lambda$, there exists, uniquely up to tensor isomorphism,
a $k$\nd tensor functor from $\Rep_k(G)$
to $\sC$ which sends $V_\lambda$ to $M_\lambda$.
\end{thm}

\begin{proof}
Let $\sR$ be a free rigid $k$\nd tensor category of type
$(n_\lambda)_{\lambda \in \Lambda}$ and $(N_\lambda)_{\lambda \in \Lambda}$
be its universal family of objects.
Then $\sR$ is generated as a rigid $k$\nd tensor category by the $N_\lambda$.
Denote by $\sJ$ the tensor ideal of $\sR$ generated by
the identities of the $\bigwedge^{n_\lambda +1} N_\lambda$.
We have $k$\nd tensor functors $E:\sR/\sJ \to \Rep_k(G)$
and $\sR/\sJ \to \sC$, unique up to tensor isomorphism,
which send $N_\lambda$ respectively to $V_\lambda$ and $M_\lambda$.
It will suffice to show that $E$ is an equivalence.
By Proposition~\ref{p:freerigid}, the hom $k$\nd spaces of $\sR$ are
finite-dimensional, so that by lifting of idempotents $\sR/\sJ$ is
pseudo-abelian.
Since $\Rep_k(G)$ is generated as a rigid $k$\nd tensor category by the
$V_\lambda$, it is thus enough to show that $E$ is fully faithful.

Let $L$ and $M$ be objects of $\sR$.
To show that
\[
E_{L,M}:(\sR/\sJ)(L,M) \to \Hom_G(E(L),E(M))
\]
is bijective, we may suppose that $L$ is the
source and $M$ is the target of the hom $k$\nd space \eqref{e:freerigid}.
By Proposition~\ref{p:freerigid}, the canonical homomorphism from the tensor
product over $k$ of the
$(\sR/\sJ)(N_{\lambda_i}{}^{\otimes r_i},N_{\lambda_i}{}^{\otimes s_i})$ to
$(\sR/\sJ)(L,M)$ is surjective.
Its composite with $E_{L,M}$ factors through the canonical homomorphism
from the tensor product over $k$ of the
$\Hom_G(V_{\lambda_i}{}^{\otimes r_i},V_{\lambda_i}{}^{\otimes s_i})$
to $\Hom_G(E(L),E(M))$, which is an isomorphism because $G$ is the product
of the $GL_{n_\lambda}$ and $V_\lambda$ is a representation of $GL_{n_\lambda}$.
Thus we reduce to the case where $L = N_\lambda{}^{\otimes r}$ and
$M = N_\lambda{}^{\otimes s}$
for some $\lambda$, $r$ and $s$.
If $r \ne s$, then the source of $E_{L,M}$ is $0$ by Proposition~\ref{p:freerigid},
while its target is $0$ because a $k$\nd point $z$ of the centre $\bG_m$ of
$GL_{n_\lambda}$ acts a $z^r$ on $E(L)$ and as $z^s$ on $E(M)$.
If $u$ is given by the action of $\mathfrak{S}_r$ on
$N_\lambda{}^{\otimes r}$ and $v$ modulo structural isomorphisms of
$E$ by $E_{N_\lambda{}^{\otimes r},N_\lambda{}^{\otimes r}}$, then
\[
k[\mathfrak{S}_r] \xrightarrow{u}
\End_{\sF/\sJ}(N_\lambda{}^{\otimes r}) \xrightarrow{v}
\End_G(V_\lambda{}^{\otimes r})
\]
is given by the action of
$\mathfrak{S}_r$ on $V_\lambda{}^{\otimes r}$.
By Proposition~\ref{p:freerigid}, $u$ is surjective,
and by definition of $\sJ$ the kernel of $u$ contains the antisymmetriser
$\alpha$ of $k[\mathfrak{S}_{n_\lambda+1}] \subset k[\mathfrak{S}_r]$ when
$r > n_\lambda$.
By Lemma~\ref{l:symgroup}, $v \circ u$ is surjective, with kernel $0$ when
$r \le n_\lambda$ and the ideal generated by $\alpha$ when $r > n_\lambda$.
It follows that $v$ an isomorphism.
\end{proof}

Let $\sC$ be a $k$\nd tensor category.
Clearly the dual or any direct summand of a positive (resp.\ negative)
object in $\sC$  is positive (resp.\ negative).
Using Theorem~\ref{t:GL}, we can see as follows
that the direct sum
of two positive (resp.\ negative) objects in $\sC$ is positive (resp.\ negative),
and that the tensor product of two positive or two negative objects in $\sC$ is
positive and the tensor product of positive and a negative object in $\sC$
is negative.
We may suppose that the objects have integer rank.
Then in the positive case we reduce by Theorem~\ref{t:GL} to the where
$\sC$ is a category $\Rep_k(G)$, in which every object is positive.
To see for example that $N_0 \otimes N_1$ is negative when $N_0$ is
positive of rank $r_0 \ge 0$ and $N_1$ is negative of rank $r_1 \le 0$ in $\sC$,
we may suppose that $\sC = \sR/\sJ$, where $\sR$ is a
free rigid $k$\nd tensor category of type $(r_i)_{i = 0, 1}$ with universal family
$(N_i)_{i = 0, 1}$, and $\sJ$ is generated by
the identities of $\bigwedge^{r_0+1}N_0$ and $S^{-r_1+1}N_1$.
By Proposition~\ref{p:freerigid}, $\sC$ has then a $\Z/2$\nd grading, with $N_0$ of
degree $0$ and $N_1$ of degree $1$.
If $\sC^\dagger$ is obtained from $\sC$ by modifying its symmetry according
to this $\Z/2$\nd grading, then $N_0$ and $N_1$ and hence $N_0 \otimes N_1$
are positive in $\sC^\dagger$.
Thus since $N_0 \otimes N_1$ is of degree $1$, it is negative in $\sC$.

Let $M$ be a positive and $N$ be a negative object in a $k$\nd tensor category.
Then any morphism $M \to N$ or $N \to M$ is tensor nilpotent.
Indeed since $N \otimes M^\vee$ and $M \otimes N^\vee$ are negative,
it is enough by \eqref{e:adjtens} to show that $f:\I \to N$ is tensor nilpotent
when $N$ is negative.
By naturality of the symmetries, the composite of $f^{\otimes n}$ with any symmetry
permuting the factors of $N^{\otimes n}$ is $f^{\otimes n}$.
Thus $f^{\otimes n}$ factors through the embedding $S^n N \to N^{\otimes n}$.
Taking $n$ large now gives the result.

For a more precise version of the following Corollary, see \cite{O}, Lemma~3.4.

\begin{cor}\label{c:posrep}
Let $\sB$ be an essentially small $k$\nd tensor category
in which every object is positive.
Then there exists a faithful and conservative
$k$\nd tensor functor from $\sB$
to the category of finitely generated projective modules over some
commutative $k$\nd algebra.
\end{cor}

\begin{proof}
Every object of $\sB$ is a direct summand of one of rank an integer $\ge 0$.
Thus for an appropriate product $G$ of general linear groups,
there is by Theorem~\ref{t:GL} an essentially surjective
$k$\nd functor from $\Rep_k(G)$ to a full $k$\nd pretensor subcategory $\sB_0$
of $\sB$ with pseudo-abelian hull $\sB$.
Lemma~\ref{l:frmod}~\ref{i:frmodes} with $\sC = \Rep_k(G)$ and
$\widehat{\sC} = \REP_k(G)$ shows that $\Rep_k(G) \to \sB_0$ factors
through a $k$\nd tensor  equivalence $\sF_R \to \sB_0$ for some
commutative algebra $R$ in $\REP_k(G)$.
Now $\sF_R$ is $k$\nd tensor equivalent to the category of free $R$\nd modules
on objects of $\Rep_k(G)$.
If $R_0$ is the underlying $k$\nd algebra of $R$, we thus have a faithful
and conservative $k$\nd tensor functor from $\sF_R$, and hence from $\sB_0$ and $\sB$,
to the category of finitely generated projective $R_0$\nd modules.
\end{proof}

{}From Corollary~\ref{c:posrep} we obtain the Cayley--Hamilton theorem:
if $M$ is a positive object in a $k$\nd tensor category and $\bigwedge^{m+1}{M} = 0$,
then
\[
\sum_{i=0}^m (-1)^i \tr(\bigwedge^i f) f^{m-i} = 0
\]
for every endomorphism $f$ of $M$.
This can also be obtained by contraction, starting from the fact that $M \otimes f^{\otimes m}$
composed with the antisymmetriser of $M^{\otimes (m+1)}$ is $0$.

\subsection{Kimura objects}\label{ss:Kimura}

We recall here some results on Kimura objects and Kimura categories.
Most of these are due to Andr\'e and Kahn (\cite{AndKah},~\S~9),
based on the work of Kimura \cite{Kim}.

Let $\sC$ be a $k$\nd tensor category.
An object in $\sC$ will be called a \emph{Kimura object} if it is
the direct sum of a positive and a negative object.
Kimura objects are dualisable, direct sums and tensor products of Kimura
objects are Kimura objects, and any $k$\nd tensor functor between $k$\nd tensor
categories sends Kimura objects to Kimura objects.
If $\sC$ is rigid, then any endomorphism in the kernel of
$\sC \to \sC_{\mathrm{red}}$ is nilpotent,
so that $\sC_{\mathrm{red}}$ is a $k$\nd tensor category
and an object of $\sC$ is a Kimura object if and only if its image in
$\sC_{\mathrm{red}}$ is.
Any direct summand of a Kimura object in $\sC$ is a Kimura object:
it is enough to consider the case where $\sC$ is reduced, when there are
no non-zero morphisms between positive and negative objects.
The full subcategory of $\sC$ consisting of the Kimura objects is thus a
$k$\nd tensor subcategory of $\sC$.
Similarly the summands in any decomposition of a Kimura object into a direct sum of
a positive and a negative object are unique up to isomorphism.
Considering their ranks shows that
any endomorphism of a Kimura object which is either a section or a retraction is
an isomorphism.

A $k$\nd tensor category $\sC$ will be called a \emph{Kimura category} if
$\sC$ is essentially small, $\End_{\sC}(\I) = k$, and every object of $\sC$
is a Kimura object.
A Kimura $k$\nd tensor category will be called \emph{positive}
if each of its objects is positive,
and \emph{split} if $\sC(M,N)$
and $\sC(N,M)$ are $0$ when $M$ is positive and $N$ is negative.
A split Kimura $k$\nd tensor category has a canonical $\Z/2$\nd grading by positive
and negative objects, and any $k$\nd tensor functor between split Kimura
$k$\nd tensor categories preserves the $\Z/2$\nd gradings.
A reduced Kimura $k$\nd tensor category is split.

Let $\sC$ be a Kimura $k$\nd tensor category.
Since $\sC$ is rigid with $\End(\I)$ a field,
it has a unique maximal tensor ideal $\Rad(\sC)$,
consisting of those morphisms $f:M \to N$ of $\sC$ for which
$\tr(f \circ g) = 0$ for every $g:N \to M$.
It has been shown in \cite{AndKah},~9.1.14 that the ideals $\Rad(\sC)(N,N)$
are nilpotent.
However we need here only the following weaker result (cf.\ \cite{Kim},~7.5):
any endomorphism of $\sC$ which lies in $\Rad(\sC)$ is nilpotent.
To see this we may after replacing $\sC$ by $\sC_{\mathrm{red}}$
suppose that $\sC$ is reduced, and hence split.
Modifying the symmetry of $\sC$ according to its canonical $\Z/2$\nd grading,
we may suppose further that $\sC$ is positive, in which case it suffices to
use the Cayley--Hamilton theorem.
It follows in particular that for any tensor ideal $\sJ \ne \sC$ of $\sC$
the quotient $\sC/\sJ$ is a Kimura $k$\nd tensor category, and the projection
$\sC \to \sC/\sJ$ reflects isomorphisms, sections and retractions.

Let $\sC$ be a Kimura $k$\nd tensor category and $\sC_0$ be a full $k$\nd linear
subcategory of $\sC$ which is pseudo-abelian.
Then $\sC_0$ is semisimple abelian if and only if the restriction to $\sC_0$
of the projection $\sC \to \overline{\sC} = \sC/\Rad(\sC)$ is faithful.
Indeed if $\sC_0$ is semisimple abelian then any morphism in $\sC_0$ factors
as a retraction followed by a section, and a retraction or section is $0$ when
its image in $\overline{\sC}$ is $0$.
For the converse it is enough to show that $\overline{\sC}$ is semisimple abelian.
Replacing $\sC$ by  $\overline{\sC}$, we may suppose that $\Rad(\sC) = 0$.
Modifying the symmetry of $\sC$, we may further suppose that $\sC$ is positive.
Then by Corollary~\ref{c:posrep} there is a $k$\nd tensor functor $T$
from $\sC$ to the category of finite-dimensional vector spaces over
some extension of $k$.
The dimension over $k$ of $\sC(\I,M)$ is bounded by the rank of $M$,
because any non-zero $\I \to M$ is a section.
Thus $\sC$ has finite-dimensional endomorphism $k$\nd algebras.
It is enough to show that they are semisimple.
In fact if $f$ lies in the radical of $\End_\sC(N)$,
then $f \circ g$ is nilpotent and hence $\tr(f \circ g) = \tr(T(f \circ g)) = 0$
for every $g$ in $\End_\sC(N)$, so that $f = 0$.

For any Kimura $k$\nd tensor category $\sC$ there is
a split Kimura $k$\nd tensor category $\sC'$
and a $k$\nd tensor functor $\sC' \to \sC$ with the following universal property:
any $k$\nd tensor functor $\sB \to \sC$ with
$\sB$ a split Kimura $k$\nd tensor category
factors uniquely up to tensor isomorphism through $\sC' \to \sC$.
To construct such a $\sC'$ note that an isomorphism between a positive
and a negative object of a  Kimura
$k$\nd tensor category can exist only when both objects are $0$.
Thus for every $\sC$ we can define a (not necessarily full)
$k$\nd pretensor subcategory $\sC^{\pm}$
whose objects are those objects of $\sC$ which are either positive or negative,
where $\sC^{\pm}(M,N)$ coincides with $\sC(M,N)$ when
$M$ and $N$ are both positive or both negative, and is $0$ otherwise.
Now take for $\sC'$ a pseudo-abelian hull of $\sC^{\pm}$.
Then $\sC'$ is a split Kimura $k$\nd tensor category, and the embedding
$\sC^{\pm} \to \sC$ extends uniquely up to tensor isomorphism
to a $k$\nd tensor functor $\sC' \to \sC$.
Let $\sB$ be a split Kimura $k$\nd tensor category.
Then $\sB^{\pm}$ is a full $k$\nd pretensor subcategory of $\sB$,
and $\sB$ is a pseudo-abelian hull of $\sB^{\pm}$.
Further any $k$\nd tensor functor $\sB^{\pm} \to \sC$ factors (uniquely)
through $\sC^{\pm}$.
Thus any $k$\nd tensor functor $\sB \to \sC$ factors uniquely
up to tensor isomorphism through $\sC' \to \sC$.
It is clear from the construction that $\sC' \to \sC$ induces an equivalence
on full subcategories of either positive or negative objects.
Thus $\sC' \to \sC$ composed with the projection $\sC \to \overline{\sC}$
is full and essentially surjective, and hence
induces a $k$\nd tensor equivalence
$\overline{\sC'} \to \overline{\sC}$.

\subsection{Hopf algebras}\label{ss:Hopf}

Let $\sC$ be a $k$\nd pretensor category.
Then coproducts exist in the category of commutative algebras in $\sC$,
and are given by the tensor product of algebras.
A \emph{commutative and cocommutative Hopf algebra} in $\sC$ is a
commutative cogroup object in the category of commutative algebras in $\sC$.
Explicitly, such a Hopf algebra $R$ has in addition to its commutative
algebra structure a cocommutative coalgebra structure, defined
by a counit $R \to \I$ and a comultiplication $R \to R \otimes R$,
and an antipode $R \to R$, satisfying appropriate compatibilities.
For each such $R$ and integer $n$, we have an endomorphism $n_R$ of $R$,
given by taking the $n$th multiple, with $(-1)_R$ the antipode.

Let $N$ be a negative object of a $k$\nd tensor category $\sC$.
Since $S^n N = 0$ for $n$ large, there is defined an object
\[
\Sym N = \bigoplus_r S^r N
\]
in $\sC$.
It has a canonical structure of commutative and cocommutative Hopf algebra,
which respects the $\Z$\nd grading by the $S^r N$, defined as follows.
The unit and counit of $\Sym N$ are respectively the embedding of and the projection
onto $S^0 N = \I$.
If $v_r:S^r N \to N^{\otimes r}$ is the embedding and
$w_r:N^{\otimes r} \to S^r N$ is the projection,
then the multiplication is defined by
\[
w_{r+s} \circ (v_r \otimes v_s):S^r N \otimes S^s N \to S^{r+s} N
\]
and the comultiplication by
\[
\frac{(r+s)!}{r!s!}(w_r \otimes w_s) \circ v_{r+s}:S^{r+s} N \to S^r N \otimes S^s N.
\]
The embedding of $N = S^1N$ into $\Sym N$ is universal among morphisms from $N$
to a commutative algebra in $\sC$,
and the projection of $\Sym N$ onto $N$ is universal among morphisms from a
cocommutative coalgebra in $\sC$ to $N$.
For every $n$, the morphism $n_{\Sym N}$ acts on $S^1 N = N$ as $n$
and hence on $S^r N$ as $n^r$.

A commutative and cocommutative Hopf
algebra in a $k$\nd tensor category $\sC$ will be called \emph{symmetric}
if it is isomorphic to $\Sym N$ for some negative object $N$ of $\sC$.
A symmetric Hopf algebra has a canonical $\Z$\nd grading, with $R_r$ the component
of $R$ on which the $n_R$ act as $n^r$.

We have a functor $\Sym$ from the category of negative objects in a $k$\nd tensor
category $\sC$ to the category of commutative and cocommutative Hopf algebras in $\sC$.
It can be seen as follows that it is fully faithful.
Any morphism $h:\Sym N \to \Sym N'$ of
Hopf algebras respects the canonical gradings, because
$h \circ  2_{\Sym N} = 2_{\Sym N'} \circ h$ for example.
If $h_1$ is the component of $h$ of degree $1$, the universal property of
$\Sym N$ as a free commutative algebra shows that $h = \Sym f$
if and only if $f = h_1$.

We recall here the following facts
about morphisms in a pseudo-abelian $k$\nd linear category,
which will be needed in the proof of Theorem~\ref{t:Hopf}
\begin{enumerate}
\renewcommand{\theenumi}{(\arabic{enumi})}
\item\label{i:polydec}
If $t:M \to M$ satisfies
$\prod_{i=1}^d p_i(t) = 0$ for mutually coprime polynomials $p_i(T)$ over $k$,
then $M$ has a unique direct sum decomposition $\bigoplus_{i=1}^d M_i$
which is respected by $t$ and is such that the component $t_i:M_i \to M_i$
of $t$ satisfies $p_i(t_i) = 0$ for each $i$.
\item\label{i:polyzero}
If $t:M \to M$ and $t':M' \to M'$ satisfy
$p(t) = 0$ and $p'(t') = 0$ for mutually coprime polynomials $p(T)$ and $p'(T)$
over $k$,
then $f \circ t = t' \circ f$ implies $f = 0$ for any $f:M \to M'$.
\end{enumerate}
For \ref{i:polyzero}, note that $f \circ t = t' \circ f$ implies $f \circ q(t) = q(t') \circ f$
for any polynomial $q(T)$ over~$k$.
Choosing $q(T)$ such that its image in $k[T]/p(T) \times k[T]/p'(T)$
is $(1,0)$ shows that $f = 0$.
If the image of $r_i(T)$ in
$k[T]/(\prod_{i=1}^d p_i(T)) = \prod_{i=1}^d k[T]/p_i(T)$ has $i$th component $1$
and $j$th component $0$ for $j \ne i$,
then the $r_i(t)$ are idempotents in $\End(M)$ giving a decomposition $\bigoplus_{i=1}^d M_i$
as in \ref{i:polydec}.
If $M = \bigoplus_{i=1}^d M'{}\!_i$ is another such decomposition, the composite
$M_i \to M \to M'{}\!_j$ is $0$ for $i \ne j$ by \ref{i:polyzero}.

\begin{thm}\label{t:Hopf}
A commutative and cocommutative Hopf algebra $R$ in a Kimura $k$\nd tensor category $\sC$
is symmetric if and only if for every integer $n \ne 0$ the morphism $n_R:R \to R$ in $\sC$
is an isomorphism.
\end{thm}

\begin{proof}
That $n_R$ is an isomorphism for any $n \ne 0$ when $R$ is symmetric is clear.
Conversely suppose that for $n \ne 0$ the morphism $n_R$
in $\sC$ is an isomorphism.
It is to be shown that $R$
isomorphic to $\Sym N$ for some negative $N$ in $\sC$.

Consider first the case where $\sC$ is positive.
Then by Corollary~\ref{c:posrep} there is a $k$\nd tensor functor
$U$ from $\sC$ to finite-dimensional $k'$\nd  vector spaces for some
extension $k'$ of $k$, which we may assume is algebraically closed.
In that case $\Spec(U(R))$ is a discrete finite commutative group scheme over $k'$
on which each $n \ne 0$ acts as an isomorphism.
Thus $\Spec(U(R))$ is trivial, so that $U(R)$ and hence $R$ has rank $1$.
Since $R$ is positive and has a direct summand $\I$, we have $R = \I$.

Now consider the general case.
The Kimura $k$\nd tensor category $\overline{\sC}$ is semisimple abelian,
and hence with an exact tensor product,
and  has a $\Z/2$\nd grading by positive and negative objects.
Write $M^+$ and $M^-$ for the positive and negative summands of $M$ in $\overline{\sC}$,
and $\overline{\sC}{}^+$ for the strictly full subcategory of $\overline{\sC}$
consisting of the positive objects.
If $Q$ is a commutative algebra in $\overline{\sC}$, then the cokernel $Q_+$
of the morphism
\[
Q^- \otimes Q^- \to Q^+
\]
in $\overline{\sC}$ obtained by restriction
from the multiplication of $Q$ is a quotient algebra of $Q$ which lies in
$\overline{\sC}{}^+$.
The composite $Q \to Q^+ \to Q_+$ of the projections
is then universal among morphisms of algebras from $Q$ in
$\overline{\sC}$ to commutative algebras in $\overline{\sC}{}^+$.

The image $\overline{R}$ of $R$ in $\overline{\sC}$ is a commutative and
cocommutative Hopf algebra in $\sC$ with
$n_{\overline{R}}$ an isomorphism for $n \ne 0$.
By the universal property of the algebra $\overline{R}_+$,
there is a unique structure of
commutative and cocommutative Hopf algebra on $\overline{R}_+$
such that $\overline{R} \to \overline{R}_+$ is a morphism of Hopf algebras,
and $n_{\overline{R}_+}$ is then an isomorphism for $n \ne 0$.
Since $\overline{\sC}{}^+$ is positive, we have $\overline{R}_+ = \I$.
The kernel $J$ of the counit $\overline{R} \to \I$ is an ideal in the algebra
$\overline{R}$.
We have $\overline{R} = \I \oplus J$ and hence
$\overline{R}{}^+ = \I \oplus J^+$.
It follows that $\overline{R}{}^- \otimes \overline{R}{}^- \to \overline{R}{}^+$
factors through an epimorphism
\begin{equation}\label{e:RRJepi}
\overline{R}{}^- \otimes \overline{R}{}^- \to \overline{J}{}^+ \to 0
\end{equation}
in $\overline{\sC}$.
If $\overline{R}{}^-$ has rank $-d$, we have $S^r \overline{R}{}^- = 0$
for $r > d$, so that by commutativity of  the multiplication of $\overline{R}$
the restriction to $(\overline{R}{}^-)^{\otimes r}$ of the
$r$\nd fold multiplication $\mu_r:\overline{R}{}^{\otimes r} \to \overline{R}$
is $0$ for $r > d$.
Since $\overline{R}{}^- \otimes \overline{R}{}^- \to \overline{R}{}^+$ factors through
\eqref{e:RRJepi}, it follows that the restriction
of $\mu_r$ to $J^{\otimes r} = (J^+ \oplus \overline{R}{}^-)^{\otimes r}$ is $0$ for $r > d$.
Thus the images $J^r$ of the $J^{\otimes r} \to \overline{R}$ form a filtration of
$\overline{R}$ by ideals for which $J^{d+1} = 0$ and $J^+ \subset J^2$.

The restriction $J \to \overline{R} \otimes \overline{R}$ of the comultiplication
to $J$ has components at $\I \otimes \I$, $\I \otimes J$ and $J \otimes \I$
respectively $0$, $1_J$ and $1_J$.
Composing with the multiplication thus shows that
the endomorphism $2_{\overline{R}}$ of $\overline{R}$ acts as
$2$ on $J/J^2$.
Hence $2_{\overline{R}}$ acts as $2^r$ on the quotient $J^r/J^{r+1}$ of
$(J/J^2)^{\otimes r}$ for $r \ge 0$.
It follows that for $s \ge 0$ the endomorphism
$\prod_{i=0}^s(2_{\overline{R}} - 2^{i+r})$ acts as $0$ on $J^r/J^{r+s+1}$.
Taking $r = 0$ and $s = d$ and using \ref{i:polydec} above
thus shows that there is a unique decomposition
$\overline{R} = \bigoplus_{r = 0}^d \overline{R}_i$
which is respected by $2_{\overline{R}}$ and for which $2_{\overline{R}}$
acts as $2^r$ on $\overline{R}_r$.
The composites $J^r \to \overline{R} \to \bigoplus_{i<r} \overline{R}_i$ and
$\bigoplus_{i \ge r} \overline{R}_i \to \overline{R} \to \overline{R}/J^r$ are $0$
for each $r \ge 1$ by \ref{i:polyzero} above,
so that $J^r = \bigoplus_{i \ge r} \overline{R}_i$.
Thus the quotient $\overline{R}_1 = J/J^2$ of $J/J^+$ is negative,
and the morphism of algebras
$\Sym \overline{R}_1 \to \overline{R}$ defined by the embedding
$\overline{R}_1 \to \overline{R}$ is
an epimorphism in $\overline{\sC}$ and hence by semisimplicity of
$\overline{\sC}$ a retraction.

Since the endomorphism $\prod_{i=0}^d(2_{\overline{R}} - 2^i)$ of
$\overline{R}$ is $0$, there is an $e$ for which
the endomorphism $\prod_{i=0}^d(2_R - 2^i)^e$ of $R$ is $0$.
By \ref{i:polydec} above, $R$ has thus a unique
decomposition $R = \bigoplus_{r = 0}^d R_i$
respected by $2_R$ such that $(2_R - 2^r)^e$ is $0$ on $R_r$.
Then $R_r$ lies above $\overline{R}_r$ for each $R$.
Thus $R_1$ is negative because its image in $\overline{\sC}$ is negative,
and the morphism of algebras
$\alpha:\Sym R_1 \to R$
defined by the embedding
$R_1 \to R$ is a retraction in $\sC$ because its image in $\overline{\sC}$
is a retraction.
By \ref{i:polyzero} above, the unit and counit of $R$ factor through $R_0$.
Also the endomorphism $2_{R \otimes R} = 2_R \otimes 2_R$ of $R \otimes R$
sends $R_r \otimes R_s$ to itself, and $(2_{R \otimes R} - 2^{r+s})^{2e -1}$
acts on it as $0$.
Since the multiplication and comultiplication of $R$ are morphisms of
Hopf algebras, their only non-zero components are thus by \ref{i:polyzero}
those of the form $R_r \otimes R_s \to R_{r+s}$ and $R_{r+s} \to R_r \otimes R_s$.
Thus $\alpha$ induces a retraction $\I \to R_0$, so that $R_0 = \I$,
and the restriction $\delta:R_1 \to R \otimes R$ of the comultiplication of
$R$ to $R_1$ factors through a morphism, necessarily the diagonal, from $R_1$ to
$(R_1 \otimes \I) \oplus (\I \otimes R_1)$.
The two morphisms of algebras $\Sym R_1 \to R \otimes R$ given by composing
$\alpha$ with the comultiplication of $R$ and the comultiplication of $\Sym R_1$
with $\alpha \otimes \alpha$ thus coincide,
because both have restriction $\delta$ to $R_1$.
Hence $\alpha$ is a morphism of Hopf algebras.
Similarly the morphism of coalgebras $\beta:R \to \Sym R_1$ defined by the projection
$R \to R_1$ is a morphism of Hopf algebras.
Thus $\beta \circ \alpha$ is the identity, and $\alpha$ is an isomorphism.
\end{proof}

\subsection{The splitting and unique lifting theorems}\label{ss:splitlift}

Let $G$ be a $k$\nd group.
A commutative $G$\nd algebra $R$ will be called \emph{simple} if $R \ne 0$
and $R$ has no quotient $G$\nd algebra other than $R$ and $0$.
By a $G$\nd scheme we mean a scheme over $k$ equipped with an action of $G$.
If $R$ is a commutative $G$\nd algebra, then $\Spec(R)$ has a structure of
$G$\nd scheme with the action of the point $g$ of $G$ on $\Spec(R)$ defined
by the action of $g^{-1}$ on $R$.
Suppose that $G$ is of finite type.
Then a $G$\nd scheme $X$ is said to be \emph{homogeneous} if for some
extension $k'$ of $k$ and $k'$\nd subgroup $H$ of $G_{k'}$, the $G_{k'}$\nd scheme
$X_{k'}$ is isomorphic to $G_{k'}/H$.
It is equivalent to require that for every $k$\nd scheme $S$ and pair
of points $x_1$ and $x_2$ in $S$, there should exist a surjective \'etale
morphism $S' \to S$ and a point $g$ of $G$ in $S'$ such that $gx_1 = x_2$.

\begin{lem}\label{l:homogeneous}
 \textnormal{(Magid \cite{Mag},~Theorem~4.5)}
 Let $G$ be a $k$-group of finite type and
 let $R$ be a commutative $G$-algebra with $R^G = k$.
 Then the $G$\nd algebra $R$ is simple if and only if
 the $G$\nd scheme $\Spec (R)$ is homogeneous.
\end{lem}

Let $G$ be a $k$\nd group.
Then limits in the category of $G$\nd algebras exist, and coincide
with limits of the underlying $G$\nd modules.
However limits of $G$\nd modules do not in general coincide with
limits of the underlying $k$\nd vector spaces:
the canonical homomorphism from a limit of $G$\nd modules to the
limit of their underlying $k$\nd vector spaces is injective
but in general not surjective.
If $G$ is reductive, then limits of $G$\nd modules can be calculated
isotypic component by isotypic component,
and if $R$ is a finitely generated commutative $G$\nd algebra with $R^G = k$,
the isotypic components of $R$ are finite dimensional.
This fact renders plausible the following rather surprising result,
for whose proof we refer to \cite{O}, Lemma~5.2.

\begin{lem}\label{l:complete}
 Let $G$ be a reductive $k$\nd group, $R$ be a finitely generated
 commutative $G$-algebra with $R^G = k$, and $J \ne R$ be
 a $G$-ideal of $R$.
 Then $R$ is the limit in the category of $G$\nd algebras
 of its $G$\nd quotients $R/J^n$.
\end{lem}

Let $G$ be a $k$\nd group and $H$ be a normal $k$\nd subgroup of $G$.
Then a $(G/H)$\nd scheme may also be regarded as a $G$\nd scheme,
and a $G$\nd scheme on which $H$ acts trivially as a $G/H$\nd scheme.
Suppose that $G$ is of finite type, and let and $X$ be a homogeneous
$G$\nd scheme.
Then the quotient $X/H$ exists, and is a homogeneous $(G/H)$\nd scheme.
Further if $H$ is reductive and $X = \Spec(R)$ is affine, then $X/H = \Spec(R^H)$.

\begin{lem}\label{l:fibreprodhomog}
Let $G$ be a $k$\nd group of finite type, $H$ be a normal $k$-subgroup of $G$,
$X$ be a homogeneous $G$-scheme and $Y$ be a homogeneous $G/H$-scheme.
Then for any $G$\nd morphism $Y \to X/H$, the $G$-scheme $X \times_{X/H} Y$ is homogeneous.
\end{lem}

\begin{proof}
Let $S$ be a $k$-scheme and $(x_1,y_1)$ and $(x_2,y_2)$
be $S$-points of $X \times_{X/H} Y$.
Since $Y$ is a homogeneous $G$\nd scheme,
there is a surjective \'etale morphism $S' \rightarrow S$
and a $g \in G(S')$ such that $gy_1 = y_2$.
Then $gx_1$ and $x_2$ lie above the same $S'$-point
of $X/H$.
Thus there is a surjective \'etale morphism
$S'' \rightarrow S'$ and an $h \in H(S'')$ such that
$hgx_1 = x_2$.
Since $H$ acts trivially on $Y$, we have $hg(x_1,y_1) = (x_2,y_2)$.
\end{proof}

Let $G$ be a $k$\nd group of finite type
and $X$ be a reduced $G$\nd scheme of finite type.
Then any $G$\nd morphism to $X$ from a homogeneous $G$\nd scheme factors through a
homogeneous $G$\nd subscheme of $X$.
Thus $X$ has a dense open homogeneous $G$\nd subscheme if and only if
there exists a dominant $G$\nd morphism to $X$ from some homogeneous $G$\nd scheme.
Such a $G$\nd subscheme is unique when it exists, and every
dominant $G$\nd morphism from a homogeneous $G$\nd scheme factors
(uniquely) through it.

\begin{lem}\label{l:retraction}
 Let $G$ be a proreductive $k$-group,
 $R$ be a commutative $G$\nd algebra with $R^G = k$,
 $D$ be a $G$-subalgebra of $R$, and $p:R \rightarrow \overline{R}$
 be the projection onto a simple quotient $G$-algebra of $R$.
 Suppose that the restriction of $p$ to $D$ is injective.
 Then $R$ has a $G$-subalgebra $D'$ containing $D$
 such that the restriction of $p$ to $D'$ is an isomorphism.
\end{lem}

\begin{proof}
Denote by $J$ the kernel of $p$.
We consider successively the cases where
(1) $G$ is of finite type, $J^2 = 0$, and $D$ is a
finitely generated $k$-algebra
(2) $G$ is of finite type and $D$ is a finitely generated
$k$\nd algebra, and then (3) the general case.

\quad (1)
By Lemma~\ref{l:homogeneous}, the $G$\nd scheme $\Spec(\overline{R})$ is homogeneous.
Since $D \to \overline{R}$ is injective and $D$ is finitely generated,
the $G$\nd morphism  $\Spec(\overline{R}) \to \Spec(D)$ thus factors through an
open homogeneous $G$\nd subscheme of $\Spec(D)$.
Hence $\overline{R}$ is smooth over $D$.
Write $E_1$ for the set of $D$-algebra homomorphisms
$\overline{R} \rightarrow R$ right inverse to $p$,
and $E_0$ for the $k$-space of derivations of $\overline{R}$ over
$D$ with values in $J$.
Let $V$ be a finite-dimensional
$G$-submodule of $\overline{R}$ which contains $k$ and generates $\overline{R}$.
We may regard $E_1$ as a subset and $E_0$ as a $k$\nd subspace of the $G$\nd module
$V^\vee \otimes_k R$ of $k$\nd linear maps from $V$ to $R$.
By smoothness of $\overline{R}$ over $D$, the set $E_1$ is non-empty.
Let $e$ be an element of $E_1$.
If extension of scalars to a commutative
$k$\nd algebra $k'$ is denoted by a subscript $k'$,
then $(E_0)_{k'}$ is the $k'$\nd module of derivations of $\overline{R}_{k'}$
over $D_{k'}$ with values in $J_{k'}$, and
\[
e + (E_0)_{k'} \subset (V^\vee \otimes_k R)_{k'}
\]
is the set of $D_{k'}$-algebra homomorphisms $\overline{R}_{k'} \rightarrow R_{k'}$
right inverse to $p_{k'}$, and hence is stable under $G(k')$.
Thus $E = ke + E_0$ is a $G$\nd submodule of $V^\vee \otimes_k R$,
and evaluation at $1 \in V$ defines a surjective $G$\nd homomorphism from
$E$ to $k \subset R$ with fibre $E_1$ above $1 \in k$.
Since $G$ is reductive, the set $E^G \cap E_1$
of homomorphisms of $G$-algebras $\overline{R} \rightarrow R$
over $D$ right inverse to $p$ is non-empty.
Take for $D'$ the image of such a homomorphism.

\quad (2)
By Lemma~\ref{l:homogeneous}, $\overline{R}$ is a finitely generated
$k$\nd algebra.
Replacing $R$ by its $G$\nd subalgebra generated by $D$ and the lifting
to $R$ of a finite set of generators of $\overline{R}$, we may suppose
that $R$ is a finitely generated $k$\nd algebra.
Then by Lemma~\ref{l:complete}, $R$ is the limit in the category of
$G$\nd algebras of its quotients $R/J^n$.
If $D_n$ is the image of $D$ in $R/J^n$, it is thus enough to show
that any $G$\nd subalgebra $D'{}\!_n \supset D_n$ of $R/J^n$ with
$D'{}\!_n \to \overline{R}$ an isomorphism can be lifted to
a $G$\nd subalgebra $D'{}\!_{n+1} \supset D_{n+1}$ of $R/J^{n+1}$ with
$D'{}\!_{n+1} \to \overline{R}$ an isomorphism.
To do this, apply (1) with the inverse
image of $D'{}\!_n$ in $R/J^{n+1}$ for $R$ and $D_{n+1}$ for $D$.

\quad (3)
By Zorn's Lemma the set of those $G$\nd subalgebras
$B \supset D$ of $R$ for which the restriction of $p$ to $B$ is injective
has a maximal element $D'$.
Write $D'$ as the filtered union of its finitely generated
$G$-subalgebras $D_\mu$, and denote by $\sH$ the set of those normal
$k$\nd subgroups $H$ of $G$ for which $G/H$ is of finite type.
For each $\mu$ there is an $H_\mu \in \sH$ such that $G$ acts on
$D_\mu$ through $G/H_\mu$.
Now if $Q$ is a simple $G$\nd algebra and $H$ is a normal $k$\nd subgroup of $G$
then $Q^{H}$ is a simple $(G/H)$\nd algebra:
we have $Q = Q^H \oplus Q_1$ with
$Q_1{}\!^H = 0$, so that $I = (QI)^H$
for any $(G/H)$\nd ideal $I$ of $Q^H$.
Hence by Lemma~\ref{l:homogeneous} $\Spec(\overline{R}{}^H)$
is a homogeneous $(G/H)$\nd scheme for $H \in \sH$.
Since $p$ induces an injective $G$\nd homomorphism from
$D_\mu$ to $\overline{R}{}^{H_\mu}$,
each $\Spec(D_\mu)$ has thus an open dense homogeneous
$(G/H_\mu)$\nd subscheme $X_\mu$.
The $X_\mu$ then form an inverse system of $G$\nd schemes.

Let $B \supset D_\mu$ be a $G$\nd subalgebra of $C$ with
$p(B) \supset \overline{R}{}^{H_\mu}$.
Then \mbox{$p(B^{H_\mu}) \supset \overline{R}{}^{H_\mu}$},
so that by (2) with $G/H_\mu$, $B^{H_\mu}$ and $D_\mu$
for $G$, $R$ and $D$, there is
a simple $(G/H_\mu)$\nd subalgebra of $B^{H_\mu}$
containing $D_\mu$.
Thus we have a unique factorisation
\[
\Spec(B) \xrightarrow{f_{\mu,B}} X_\mu \to \Spec(D_\mu)
\]
of the $G$\nd morphism $\Spec(B) \rightarrow \Spec(D_\mu)$.

Since $\overline{R}$ is the union of the $\overline{R}{}^H$ with $H \in \sH$,
it will be enough to show that $p(D') \supset \overline{R}{}^H$ for
every $H \in \sH$.
Fix an $H \in \sH$.
Then $(X_\mu/H)$ is a filtered inverse system of homogeneous
$(G/H)$\nd schemes, and hence for some $\mu_0$ is constant
for $\mu \ge \mu_0$.
Also $H_0 = H \cap H_{\mu_0}$ lies in $\sH$,
so that by (2) with $G/H_0$, $R^{H_0}$, and $D_{\mu_0}$  for $G$,
$R$, and $D$ there is a $G$-subalgebra
$D_0 \supset D_{\mu_0}$ of $R^{H_0}$ such that $p$ induces
an isomorphism $D_0 \iso \overline{R}{}^{H_0}$.
Then $D_0$ and $D_0{}\!^H$ are simple $G$\nd algebras,
and $p(D_0{}\!^H) \supset \overline{R}{}^H$.
Write
\[
i:D' \otimes_k D_0{}\!^H \to R
\]
for the homomorphism of $G$-algebras defined by the embeddings,
and $D''$ for the image of $i$.
Then $D'' \supset D'$ and $p(D'') \supset \overline{R}{}^H$.
We show that
\begin{equation}\label{e:kereq}
\Ker i = \Ker (p \circ i).
\end{equation}
This will imply that the restriction of $p$ to $D''$ is injective,
and thus $D' = D''$ by maximality of $D'$,
so that $p(D') \supset \overline{R}{}^H$ as required.

For every $\mu \ge \mu_0$ we have a commutative diagram
of $G$\nd morphisms
\[
\begin{CD}
\Spec(R) @>>>      \Spec(D_0)  @>>>      \Spec(D_0{}\!^H)  \\
@V{f_{\mu,R}}VV      @V{f_{\mu_0,D_0}}VV    @VVV   \\
X_\mu     @>>>        X_{\mu_0}      @>>>   X_{\mu_0}/H
\end{CD}
\]
where the top arrows are defined by the embeddings,
the bottom arrows are the transition morphism and the projection,
and the right vertical arrow is defined using
the isomorphism $\Spec(D_0)/H \iso \Spec(D_0{}\!^H)$.
Thus we have a factorisation
\[
\Spec(R) \to X_\mu \times_{X_{\mu_0}/H} \Spec(D_0{}\!^H)
\to X_\mu \times \Spec(D_0{}\!^H)
\to \Spec(D_\mu \otimes_k D_0{}\!^H)
\]
of the morphism defined by the restriction $i_\mu$ of $i$ to
$D_\mu \otimes_k D_0{}\!^H$.
Now $X_\mu \rightarrow X_{\mu_0}/H$ factors through an
isomorphism $X_\mu/H \rightarrow X_{\mu_0}/H$,
so that the $G$\nd scheme
\[
Z = X_\mu \times_{X_{\mu_0}/H} \Spec(D_0{}\!^{H}) =
X_\mu \times_{X_\mu/H} \Spec(D_0{}\!^{H})
\]
is homogeneous as a $(G/(H \cap H_\mu))$\nd scheme, by Lemma~\ref{l:fibreprodhomog}.
Hence the composite $\Gamma(Z,\sO_Z) \to R \to \overline{R}$ is injective,
because $\Spec(\overline{R}) \to Z$ factors through
a faithfully flat morphism $\Spec(\overline{R}{}^{H \cap H_\mu}) \to Z$.
Thus for $\mu \ge \mu_0$  we have $\Ker i_\mu = \Ker (p \circ i_\mu)$,
because both coincide with the kernel of $D_\mu \otimes_k D_0{}\!^{H} \to \Gamma(Z,\sO_Z)$.
Since $D'$ is the union of the $D_\mu$ for  $\mu \ge \mu_0$,
this gives \eqref{e:kereq}.
\end{proof}

Let $\sD'$ and $\sD''$ be $k$\nd pretensor categories with $\End(\I) = k$.
We may regard $\sD'$ and $\sD''$ as full $k$\nd pretensor subcategories
of $\sD' \otimes_k \sD''$.
If the pseudo-abelian hulls
of $\sD'$ and $\sD''$ are (positive) Kimura $k$\nd tensor categories,
then the pseudo-abelian hull of $\sD' \otimes_k \sD''$
is a (positive) Kimura $k$\nd tensor category,
because every object of $\sD' \otimes_k \sD''$ is a tensor product
of an object $\sD'$ with an object in $\sD''$.
Let $\sD$ be a rigid $k$\nd pretensor category with $\End(\I) = k$ and $\Rad(\sD) = 0$.
Then $\Rad(\sD \otimes_k \sD) = 0$.
In particular the $k$\nd tensor functor $\sD \otimes_k \sD \to \sD$ with restriction
the identity to each factor $\sD$ is faithful.
If $U':\sD' \to \sD$ and $U'':\sD'' \to \sD$ are faithful $k$\nd tensor functors,
it thus follows by factoring through $\sD \otimes_k \sD$ that
$\sD' \otimes_k \sD'' \to \sD$ with restriction $U'$ to $\sD'$ and $U''$ to $\sD''$
is faithful.

Let $G$ be a $k$\nd group.
Then $\REP_k(G)$ is an ind-completion $\Rep_k(G)$,
and an algebra in $\REP_k(G)$ is the same as a $G$\nd algebra.
Given a $G$\nd algebra $R$, we thus have the $k$\nd pretensor category $\sF_R$
and $k$\nd tensor functor $F_R:\Rep_k(G) \to \sF_R$ defined in  Section~\ref{ss:rigid},
and given a morphism $f:R \to R'$ of commutative $G$\nd algebras we have the $k$\nd tensor
functor $\sF_f:\sF_R \to \sF_{R'}$ with $F_{R'} = \sF_f F_R$.
If $\sF_f:\sF_R \to \sF_{R'}$ is faithful (resp.\ full), then $f:R \to R'$ is
injective (resp.\ surjective).
Indeed for every $M$ in $\Rep_k(G)$ the isomorphisms $\theta_{R;M,\I}$
and $\theta_{R';M,\I}$ of \eqref{e:frmodhom} show that
the homomorphism from $\Hom_G(M,R)$ to $\Hom_G(M,R')$
induced by $f$ is injective (resp.\ surjective).
If $\sF_R$ is non-zero and has no tensor ideals other than
itself and $0$, then $R$ is a simple $G$\nd algebra.
Indeed any non-zero morphism $f:R \to R'$ of commutative $G$\nd algebras
is injective, because $\sF_f$ is faithful.

\begin{lem}\label{l:splitting}
Let $\sC$ and $\sD$ be $k$\nd pretensor categories whose pseudo-abelian
hulls are Kimura $k$\nd tensor categories.
Then any lifting $\sD \to \sC$ along the projection $P:\sC \to \overline{\sC}$
of a faithful $k$\nd tensor functor $\sD \to \overline{\sC}$
factors up to tensor isomorphism through some right inverse to $P$.
\end{lem}

\begin{proof}
Let $E:\sD \to \sC$ be a $k$\nd tensor functor with $PE$ faithful.
We note that $E$ factors up to tensor isomorphism through some right inverse to $P$
if and only if $E$ is tensor isomorphic to $TPE$ for some right inverse $T$
of $P$.
Further these two equivalent conditions are also equivalent to those
obtained by replacing ``right inverse'' by ``right quasi-inverse'',
where $T$ is said to be right quasi-inverse to $P$ when $PT$ is tensor isomorphic
to the identity.
Indeed modifying such a $T$ by liftings to $\sC$ of the components of a tensor
isomorphism $PT \iso \Id_{\overline{\sC}}$ gives a strict right inverse.

Let $\sC_1$ and $\sD_1$ be pseudo-abelian hulls of $\sC$ and $\sD$.
Then $\overline{\sC}$ is a full $k$\nd pretensor subcategory of $\overline{\sC}_1$,
and $P$ is obtained by restriction from the projection $P_1:\sC_1 \to \overline{\sC_1}$.
Further $E:\sD \to \sC$ extends to an  $E_1:\sD_1 \to \sC_1$, any right inverse
$T_1$ to $P_1$ defines by restriction a right inverse $T$ to $P$, and
$E_1$ tensor isomorphic to $T_1P_1E_1$ implies $E$ tensor isomorphic to $TPE$.
Replacing $\sC$ and  $\sD$ by $\sC_1$ and $\sD_1$, we may thus
suppose that $\sC$ and $\sD$ are Kimura $k$\nd tensor categories.

Since $\overline{\sC}$ is split and $PE$ is faithful, $\sD$ is split.
It has been seen in Section~\ref{ss:Kimura} that there is a split Kimura
$k$\nd tensor category $\sC'$
and a $k$\nd tensor functor $\sC' \to \sC$ through which every $k$\nd tensor
functor from a split Kimura $k$\nd tensor category to $\sC$ factors up
to tensor isomorphism, and whose composite $\sC' \to \overline{\sC}$ with $P$
is full and essentially surjective.
Then $\sC' \to \sC$ induces an equivalence $\overline{\sC'} \to \overline{\sC}$.
Factoring $E$ up to tensor isomorphism through $\sC' \to \sC$
and replacing $\sC$ by $\sC'$, we may thus suppose that $\sC$ is split.
If we denote by a dagger the $k$\nd tensor category obtained from a split
Kimura $k$\nd tensor category by modifying its symmetry according to the
$\Z/2$\nd grading by positive and negative objects, then a $k$\nd tensor
functor $\sA \to \sC$ is at the same time a $k$\nd tensor functor
$\sA^\dagger \to \sC^\dagger$, and the projection $\sC \to \overline{\sC}$ is at the
same time the projection $\sC^\dagger \to \overline{\sC^\dagger}$.
Replacing $\sC$ and $\sD$ by $\sC^\dagger$ and $\sD^\dagger$,
we may thus suppose that
$\sC$ and $\sD$ are positive Kimura $k$\nd tensor categories.

For $\sB$ the $k$\nd tensor category of finite-dimensional representations of
an appropriate product of general linear groups over $k$,
there is by Theorem~\ref{t:GL} an essentially surjective
$k$\nd tensor functor $\sB \to \sC$.
We have $\End_\sB(\I) = k$, and $\Rad(\sB) = 0$ because $\sB$ is semisimple abelian.
The composite with $P$ of the $k$\nd tensor functor
\[
\sB \otimes_k \sD \to \sC
\]
defined by $\sB \to \sC$ and $E$ is thus faithful,
because its restrictions to $\sB$ and $\sD$ are faithful.
If $\widetilde{\sD}$ is a pseudo-abelian hull of $\sB \otimes_k \sD$, it follows
that $E$ factors through an essentially surjective $k$\nd tensor functor
$\widetilde{E}:\widetilde{\sD} \to \sC$ with $P\widetilde{E}$ faithful.
Since each of $\sB$ and $\sD$ is a positive Kimura $k$\nd tensor category,
so also is $\widetilde{\sD}$.
Replacing $E$ by $\widetilde{E}$, we may thus suppose further that $E$
is essentially surjective.

For an appropriate product $G$ of general linear groups over $k$
we have an essentially surjective $k$\nd tensor functor $K:\Rep_k(G) \to \sD$.
We obtain from it as follows commutative $G$\nd algebras $D$, $R$ and $\overline{R}$
and a diagram of $k$\nd tensor functors
\begin{equation}\label{e:sFdiag}
\begin{CD}
\sD  @>{E}>>  \sC  @>{P}>>   \overline{\sC}   \\
@AA{I_1}A          @AA{I_2}A             @AA{I_3}A           \\
\sF_D  @>{\sF_e}>>  \sF_R @>{\sF_p}>>   \sF_{\overline{R}}
\end{CD}
\end{equation}
where the $I_i$ are $k$\nd tensor equivalences
and the squares commute up to tensor isomorphism.
By Lemma~\ref{l:frmod}~\ref{i:frmodes}, we have $K = I_1F_D$ and $EK = I_2F_R$
with $I_1$ and $I_2$ fully faithful.
Then $I_1$ and $I_2$ are $k$\nd tensor equivalences because $K$ and  $EK$ are
essentially surjective.
Thus there is a $k$\nd tensor functor $\widetilde{E}:\sF_D \to \sF_R$ with
$EI_1$ tensor isomorphic to $I_2\widetilde{E}$.
Composing with $I_2$ shows that $F_R$ is tensor isomorphic to $\widetilde{E}F_D$.
Since $F_D$ is bijective on objects, we may assume after replacing $\widetilde{E}$
by a tensor isomorphic functor that $F_R = \widetilde{E}F_D$.
Thus by Lemma~\ref{l:frmod}~\ref{i:frmodff} $\widetilde{E} = \sF_e$
for a homomorphism of $G$\nd algebras $e:D \to R$.
This gives the left square of \eqref{e:sFdiag},
and the right square is obtained similarly.
{}From \eqref{e:sFdiag} it follows that  $\Rad(\sF_{\overline{R}}) = 0$,
that $\End_{\sF_R}(\I) = k$,
and that $\sF_p$ is full and $\sF_{p \circ e} = \sF_p\sF_e$ is faithful.
Thus $\overline{R}$ is a simple $G$\nd algebra, $R^G = \Hom_G(k,R) = k$,
and $p$ is surjective and $p \circ e$ is injective.
Lemma~\ref{l:retraction} thus shows that $p$ has a right inverse through which
$e$ factors, so that $\sF_p$ has a right inverse through which $\sF_e$ factors.
Thus by \eqref{e:sFdiag}, $P$ has a right
quasi-inverse through which $E$ factors
up to tensor isomorphism.
\end{proof}

The following fundamental result is due to Andr\'e and Kahn
\cite{AndKah}, 16.1.1~(a) (see also \cite{O},~Theorem~1.1).
As with Theorem \ref{t:uniquelift} below, it is obtained here as a
consequence of Lemma~\ref{l:splitting}.

\begin{thm}\label{t:splitting}
\textnormal{(Andr\'e and Kahn)}
Let $\sC$ be a $k$\nd pretensor category whose pseudo-abelian hull is a
Kimura $k$\nd tensor category.
Then the projection $\sC \to \overline{\sC}$ has a right inverse.
\end{thm}
\begin{proof}
Apply Lemma~\ref{l:splitting} with $\sD$ consisting of the single object $\I$
for which $\End(\I) = k$.
\end{proof}

\begin{thm}\label{t:uniquelift}
Let $\sC$ and $\sD$ be $k$\nd pretensor categories whose pseudo-abelian hulls are
Kimura $k$\nd tensor categories.
Then between any two liftings $\sD \to \sC$ along the projection $\sC \to \overline{\sC}$
of a faithful $k$\nd tensor functor $K:\sD \to \overline{\sC}$ there exists
a tensor isomorphism lying above the identity of $K$.
\end{thm}

\begin{proof}
Let $K_1$ and $K_2$ be liftings $\sD \to \sC$ of $K$ along the projection
$P:\sC \to \overline{\sC}$.
Then the $k$\nd tensor functor $\sD \otimes_k \sD \to \sC$
with restrictions $K_1$ and $K_2$ to the factors $\sD$
has a faithful composite with $P$.
By Lemma~\ref{l:splitting}, it thus factors up to tensor isomorphism
through a right inverse $T$ to $P$.
There thus exists for $i = 1,2$ a tensor isomorphism $K_i \iso TK'{}\!_i$
for some $K'{}\!_i$, and hence $\varphi_i:K_i \iso TK$.
Composing with $(TP)(\varphi_i{}\!^{-1})$ shows that $\varphi_i$ may be taken to lie
above the identity of $K$.
The result follows.
\end{proof}

It follows from \cite{AndKah},~Proposition~13.7.1 that the conclusion of
Theorem~\ref{t:uniquelift} holds for any $k$\nd pretensor category $\sD$
(not necessarily rigid) whose pseudo-abelian hull is semisimple abelian.
On the other hand when the pseudo-abelian hull of $\sC$ is a positive Kimura
category, it can be shown that the conclusion of Theorem~\ref{t:uniquelift}
holds for an arbitrary $k$\nd pretensor category $\sD$.
It seems very likely that in fact Theorem~\ref{t:uniquelift} is valid
without any hypothesis on $\sD$.
If so, this would have important consequences for algebraic cycles,
as will be explained in \ref{ss:conclrem}.

\section{Motives}\label{s:mot}

\emph{In this section $k$ is a field of characteristic $0$, and $S$ is a non-empty,
connected, separated, regular excellent noetherian scheme of finite Krull dimension}.

\subsection{Chow motives}\label{ss:Chowmot}

Denote by $\sS$ the category of those separated, regular, excellent noetherian
schemes which are non-empty with every connected component of the same finite
Krull dimension.
Given a scheme $Z$ in $\sS$, we write $CH(Z)$ for the Chow group of $Z$,
graded by codimension of cycles, and $CH(Z)_k$ for $CH(Z) \otimes k$.
Since $Z$ is separated, regular and noetherian, the Grothendieck group $K_0(Z)$
of locally free $\sO_Z$\nd modules of finite type coincides with the Grothendieck
group of coherent $\sO_Z$\nd modules.
Thus $K_0(Z)$ has both a $\gamma$\nd filtration defined by its structure
of $\lambda$\nd ring and a filtration by codimension of support of coherent sheaves.
There is a canonical homomorphism $CH(Z) \to K_0(Z)$ which sends the class
in $CH(Z)$ of the reduced and irreducible closed subscheme $W$ of $Z$
to the class in $K_0(Z)$ of the coherent $\sO_Z$\nd module $\sO_W$.
Since $k$ is a field of characteristic $0$ and $Z$ is separated, regular, noetherian
and of finite Krull dimension,
it follows for example from \cite{Sou}, Th\'eor\`eme~4 that
the $\gamma$\nd filtration and the filtration by codimension of support
induce the same filtration on $K_0(Z) \otimes k$, and that  $CH(Z) \to K_0(Z)$ induces
an isomorphism from $CH(Z)_k$ to the graded $k$\nd vector space
associated to $K_0(Z) \otimes k$.
The ring structure on $K_0(Z)$ then defines a structure of graded commutative
$k$\nd algebra on the graded $k$\nd vector space $CH(Z)_k$,
which coincides when $Z$ is smooth and quasi-projective over a field with
that induced by the Chow ring structure of $CH(Z)$.
Similarly the pullback $K_0(Z) \to K_0(Z')$ along any $f:Z' \to Z$ in $\sS$
defines a homomorphism $f^*:CH(Z)_k \to CH(Z')_k$ of graded $k$\nd algebras,
and if $f$ is proper and $Z$ and $Z'$ have Krull dimensions $d$ and $d'$,
the push forward  $K_0(Z') \to K_0(Z)$ along $f$
defines a homomorphism $f_*:CH(Z')_k \to CH(Z)_k$ of degree $d - d'$ of graded
$k$\nd vector spaces.
We have $(f \circ f')^* = f'{}^* \circ f^*$ and if $f$ and $f'$ are proper
$(f \circ f')_* = f_* \circ f'{}_*$.
If $f:Z' \to Z$ is proper we have the projection
formula $f_*(z'.f^*z) = f_*z'.z$ for $z$ in $CH(Z)_k$ and $z'$ in $CH(Z')_k$.
If $W$ is a closed subscheme of $Z$ with complement $U$ and $W$ lies in $\sS$,
we have an exact sequence
\begin{equation}\label{e:exacthomotopy}
CH(W)_k \to CH(Z)_k \to CH(U)_k \to 0,
\end{equation}
with the first arrow push forward along the embedding $W \to Z$ and the
second pullback along the embedding $U \to Z$.

Let $f:Z' \to Z$ and $p:Y \to Z$ be morphisms in $\sS$ such that the
pullback $h:Y' \to Y$ of $f$ along $p$ lies in $\sS$,
and suppose that $f$ is proper and that $Y$ and $Z'$ are Tor-independent over $Z$.
Then if $p':Y' \to Z'$ is the projection, we have
\begin{equation}\label{e:basechange}
p^* \circ f_* = h_* \circ p'{}^*.
\end{equation}
Indeed \eqref{e:basechange} holds with $CH(-)_k$ replaced by $K_0(-)$,
as can be seen by a \v{C}ech calculation after taking a finite affine
open cover of $Z'$ (see for example \cite{Ill},~8.3.2 and \cite{SGA6},~II~2.2.2.1).

Let $j:W \to Z$ in $\sS$ be the embedding of a closed subscheme.
Then $j$ is a regular immersion.
If the conormal sheaf $\sN$ of $W$ in $Z$ a free $\sO_W$\nd module
of rank $> 0$, then
\begin{equation}\label{e:selfint}
j^* \circ j_* = 0.
\end{equation}
Indeed \eqref{e:selfint} holds with $CH(-)_k$ replaced by $K_0(-)$,
because (\cite{SGA6}, VII~2.4 and 2.5)
$\underline{\Tor}^{\sO_Z}_i(\sF,\sO_W)$ is isomorphic to
$\sF \otimes_{\sO_W} \bigwedge^i\sN$
for every $i$ and locally free $\sO_W$\nd module of finite type $\sF$.

Let $f:Z' \to Z$ in $\sS$ be proper and surjective.
Then $f^*$ is injective and $f_*$ is surjective.
Indeed let $W$ be a reduced and irreducible closed subscheme of $Z$ with
generic point $w$.
Then the fibre of $f$ above $w$ is non-empty.
If $w'$ is the image in $Z'$ of a closed point of this fibre,
then $f(w') = w$, and the residue field of $\sO_{Z',w'}$
is finite over the residue field of $\sO_{Z,w}$.
Thus $f_*$ sends the class in $CH(Z')_k$ of the reduced and irreducible closed subscheme
of $Z'$ with generic point $w'$ to a non-zero
multiple of the class in $CH(Z)_k$ of $W$.
This shows that $f_*$ is surjective, and if $f_*(z') = 1$, the injectivity
of $f^*$ follows from the projection formula $f_*(z'.f^*z) = z$.

Suppose now that $f$ is proper and bijective, and hence a homeomorphism.
Then $f^*$ and $f_*$ are bijective.
Indeed let $W$ be a reduced and irreducible closed subscheme of $Z$ with
generic point $w$.
Then $w = f(w')$ for a unique point $w'$ of $Z'$.
Any regular sequence in $\sO_{Z,w}$ which generates
the maximal ideal of  $\sO_{Z,w}$ has image in $\sO_{Z',w'}$ a regular sequence
which generates in $\sO_{Z',w'}$ an ideal with radical the maximal ideal.
If $w'$ is the generic point of the reduced and irreducible closed subscheme $W'$
of $Z'$, it thus follows by considering Koszul complexes that the fibre at $w'$ of
$\underline{\Tor}^{\sO_Z}_i(\sO_W,\sO_{Z'})$ is non-zero for $i = 0$ and zero
for $i > 0$.
Thus $f^*$ sends the class of $W$ in $CH(Z)_k$ to a non-zero multiple of the
class of $W'$ in $CH(Z')_k$.
This shows that $f^*$ is surjective.
That $f_*$ is injective now follows from the projection formula
$f_*f^*(z) = f_*(1).z$, because $f_*(1)$ is a unit by the fact that
$f$ is a homeomorphism.

Denote by $\sV_S$ the cartesian monoidal category of those proper and
smooth schemes over $S$ which are non-empty and of
constant relative dimension.
The category $\sV_S$ has a dimension function with $\dim X$ the relative
dimension of $X$ over $S$.
By assumption, $S$ lies in $\sS$, and hence every scheme in $\sV_S$ lies in $\sS$.
We have a Chow theory $CH(-)_k$ on $\sV_S$ which sends $X$ to the graded
$k$\nd algebra $CH(X)_k$, with $p^*$ and $p_*$ for $p:X \to Y$ in $\sV_S$ as
defined above.
Indeed it is clear that conditions \ref{i:Chowfun}, \ref{i:Chowproj}, and
\ref{i:Chowiso} in the equivalent form \ref{i:Chowiso}$'$, of
Definition~\ref{d:Chow} are satisfied.
Given $X \to X'$ and $Y$ in $\sV_S$, the projection $X' \times_S Y \to X'$ is flat,
and hence $X' \times_S Y$ and $X$ are Tor-independent over $X'$,
so that by \eqref{e:basechange} condition \ref{i:Chowtens} in the equivalent form
\ref{i:Chowtens}$'$ is also satisfied.

Write $(\sM^0{}\!_{S,k},h^0{}\!_{S,k},\nu^0{}\!_{S,k})$ for the
Poincar\'e duality theory
associated to $CH(-)_k$ on $\sV_S$, and $\gamma$ for the universal morphism.
Pushing forward along the embedding of $\sM^0{}\!_{S,k}$
into a pseudo-abelian hull $\sM_{S,k}$, we obtain a Poincar\'e duality theory
$(\sM_{S,k},h_{S,k},\nu_{S,k})$.
Then $\sM_{S,k}$ is the rigid $k$\nd tensor category of $k$\nd linear Chow motives
over $S$, and
\[
\sM^0{}\!_{S,k}(\I,h^0{}\!_{S,k}(-)(\cdot)) =
\sM_{S,k}(\I,h_{S,k}(-)(\cdot)).
\]
Usually we write $h_{S,k}$ and $\nu_{S,k}$ as $h_S$ and $\nu_S$ or simply as $h$ and $\nu$.
As in \eqref{e:PDhomChow}, $\gamma$ defines for every $X$ and $Y$ in $\sV_S$ and $m$ and $n$
an isomorphism
\[
\gamma_{X,Y,m,n}:CH^{\dim X + n - m}(X \times Y)_k \iso \sM_{S,k}(h(X)(m),h(Y)(n))
\]
such that \ref{i:PDhomtrans}, \ref{i:PDhomtens}, \ref{i:PDhomcomp} and \ref{i:PDhomgraph}
of Proposition~\ref{p:PDhom} hold, as well as \eqref{e:PDhomtwist}, \eqref{e:PDhomcomppush}
and \eqref{e:PDhomcomppull}.
It is possible to take as $\sM_{S,k}$ the category of Chow motives as usually defined,
where an object is a triple $(X,e,m)$ with $X$ a proper smooth $S$\nd scheme,
$e$ in $CH(X \times_S X)_k$ idempotent under composition of correspondences,
and $m$ an integer, and to identify $\sM^0{}\!_{S,k}$ with the full subcategory
consisting of the $(X,e,m)$ with $X$ in $\sV_S$ and $e$ the identity.
If $k'$ is an extension of $k$, there is a $k$\nd tensor functor from
$\sM_{S,k}$ to $\sM_{S',k}$, unique up to tensor isomorphism,
which defines a morphism from $(\sM_{S,k},h_{S,k},\nu_{S,k})$ to
$(\sM_{S,k'},h_{S,k'},\nu_{S,k'})$ compatible with the isomorphisms $\gamma$.

Let $p:X \to Y$ be a morphism in $\sV_S$.
By \eqref{e:PDhomcomppush}, $h(p):h(Y) \to h(X)$ is a section if and only
if there is a $\beta$ in $CH(X \times_S Y)_k$ with $(p \times Y)_*(\beta)$ the class of the
diagonal of $Y$ in $CH(Y \times_S Y)_k$, and by \eqref{e:PDhomcomppull}
$h(p)$ is a retraction if and only if there is an $\alpha$ in $CH(X \times_S Y)_k$
with $(X \times p)^*(\alpha)$ the class of the diagonal of $X$ in $CH(X \times_S X)_k$.
It follows that $h(p)$ is a section when $p$ is surjective,
and an isomorphism when $p$ is a universal homeomorphism.
In particular taking for $p$ the structural morphism of $X$ shows that the
identity $\I \to h(X)$ is a section in $\sM_{S,k}$.

Let $f:S' \to S$ be a morphism with $S'$ in $\sS$ connected.
Then the product-preserving functor $- \times_S S':\sV_S \to \sV_{S'}$ defined by $f$
preserves the dimensions, and the homomorphisms
\begin{equation}\label{e:Chowpull}
CH(X)_k \to CH(X \times_S S')_k
\end{equation}
defined by the projections $X \times_S S' \to X$ give a morphism $\varphi$ of Chow theories from
$CH(-)_k$ on $\sV_S$ to the pullback
along \mbox{$\sV_S \to \sV_{S'}$} of $CH(-)_k$ on $\sV_{S'}$.
Indeed condition \ref{i:Chowmorpull} of Definition~\ref{d:Chowmor} is clearly satisfied.
Also $X$ and $Y \times_S S'$
are Tor-independent over $Y$ for any $X \to Y$ in $\sV_S$, because $X$ and $Y$
are both flat over $S$.
It thus follows from \eqref{e:basechange} that condition \ref{i:Chowmorpush}
of Definition~\ref{d:Chowmor} is satisfied.
By the universal property of $(\sM^0{}\!_{S,k},h^0{}\!_{S,k},\nu^0{}\!_{S,k})$,
there exists a unique morphism from $(\sM^0{}\!_{S,k},h^0{}\!_{S,k},\nu^0{}\!_{S,k})$ to the
pullback of  $(\sM^0{}\!_{S',k},h^0{}\!_{S',k},\nu^0{}\!_{S',k})$ along $\sV_S \to \sV_{S'}$
whose associated morphism of Chow theories coincides, modulo isomorphisms
$\gamma$, with $\varphi$.
We obtain from it, using the universal property of the pseudo-abelian hull
$\sM_{S,k}$ of $\sM^0{}\!_{S',k}$, a $k$\nd tensor functor $f^*:\sM_{S,k} \to \sM_{S',k}$
such that
\[
f^*(h_S(X)(n)) = h_{S'}(X \times_S S')(n)
\]
for every $X$ and $n$, and such that by \eqref{e:gammaT} the action of $f^*$ on morphisms
from $h_S(X)(m)$ to $h_{S'}(Y)(n)$ is given,
modulo the isomorphisms $\gamma$ and the structural isomorphisms of $\sV_S \to \sV_{S'}$,
by \eqref{e:Chowpull} with $X$ replaced by $X \times_S Y$.
If $f$ is proper and surjective then \eqref{e:Chowpull} is injective for every $X$,
so that $f^*$ is faithful.

Suppose that $f:S' \to S$ is proper and smooth.
Then any object in $\sV_{S'}$ may be regarded, after composing its structural morphism
with $f$, as an object of $\sV_S$.
For any integers $n$ and $n'$ and $X$ in $\sV_S$ and $X'$ in $\sV_{S'}$,
we obtain an isomorphism
\begin{equation}\label{e:schadj}
\sM_{S,k}(h_S(X)(n),h_S(X')(n')) \iso \sM_{S',k}(h_{S'}(X \times_S S')(n),h_{S'}(X')(n'))
\end{equation}
by applying $f^*$ and then composing with $h_{S'}(a_{X'})(n')$, where $a_{X'}:X' \to X' \times_S S'$
the canonical morphism over $S'$.
Indeed \eqref{e:PDhomcomppull} shows that \eqref{e:schadj} coincides, modulo
the isomorphisms $\gamma$,
with $CH^{\dim X + n' - n}(-)_k$ applied to the isomorphism
\[
(X \times_S S') \times_{S'} X' \to (X \times_S S') \times_{S'} (X' \times_S S') \to
X \times_S X'.
\]
It follows from \eqref{e:schadj} that $f^*$ has a right adjoint $f_*$, with
\[
f_*(h_{S'}(X')(n')) = h_S(X')(n'),
\]
where the counit is
\[
h(a_{X'})(n'):f^*f_*(h_{S'}(X')(n')) \to h_{S'}(X')(n').
\]
If $\alpha$ is the push forward of $\alpha'$ along $X' \times_{S'} Y' \to X' \times_S Y'$,
then $f_*$ sends $\gamma_{X',Y',i',j'}(\alpha')$ to $\gamma_{X',Y',i',j'}(\alpha)$,
because by \eqref{e:PDhomcomppush} the image of $\gamma_{X',Y',i',j'}(\alpha') \circ h_{S'}(a_{X'})(i')$
under the inverse of the adjunction isomorphism is $\gamma_{X',Y',i',j'}(\alpha)$.
Considering images of idempotent endomorphisms then shows using \eqref{e:PDhomtwist}
that $f_*(M'(n'))$ is isomorphic to $f_*(M')(n')$ for every $M'$,
and using Proposition~\ref{p:PDhom}~\ref{i:PDhomtrans} that
$f_*(M'{}^\vee)$ is isomorphic to $f_*(M')^\vee(d)$
with $d$ the relative dimension of $S'$ over $S$.
If $b_X:X \times_S S' \to X$ is the projection, the unit is
\[
h_S(b_X)(n):h_S(X)(n) \to f_*f^*(h_S(X)(n)),
\]
because when $X' = X \times_S S'$ and $n = n'$  the image of $h_S(b_X)(n)$ under
\eqref{e:schadj} is the identity.
Since every $b_X$ is surjective, every $h_S(b_X)$ is a section,
so that the unit $M \to f_*f^*M$
is a section for every $M$ in $\sM_{S,k}$.
Similarly if $f$ is finite and \'etale
then every $a_{X'}$ is an open and closed
immersion, so that the counit $f^*f_*M' \to M'$ is a retraction for every $M'$
in $\sM_{S',k}$.

Suppose now that $f$ is finite and \'etale.
Then there exists a finite Galois cover $f_1:S_1 \to S$ of $S$ with $S_1$
connected such that the  set $\Gamma$ of
morphisms from $S_1$ to $S'$ over $S$ is non-empty.
For every $\sigma \in \Gamma$, composing the counit with $\sigma^*$ gives
a natural transformation $c(\sigma):f_1{}\!^*f_* \to \sigma^*$.
The morphism from $f_1{}\!^*f_* M'$ to $\bigoplus_{\sigma \in \Gamma} \sigma^*M'$
with component $c(\sigma)_{M'}$ at $\sigma$ is then an isomorphism for every $M'$
because it is an isomorphism when $M' = h_{S'}(X')(n')$.
In particular the rank of $f_*M'$ is the rank of $M'$
multiplied by the degree of $f$.
Further $f_1{}\!^*f_*$ preserves positive and negative objects.
Since $f_1{}\!^*$ is faithful, it follows that $f_*$ preserves
positive and negative objects.

\begin{prop}\label{p:genfib}
If $S'$ is the spectrum of the local ring of the generic point of $S$
and $f:S' \to S$ is the inclusion,
then $f^*:\sM_{S,k} \to \sM_{S',k}$ is full,
and every morphism in its kernel is tensor nilpotent.
\end{prop}

\begin{proof}
It is enough to show that \eqref{e:Chowpull}
is surjective for every $X$ in $\sV_S$, and that if $\alpha$ lies
in its kernel and $X^n$ is the $n$\nd fold fibre product of $X$ over $S$,
then $\alpha^{\otimes n}$ in $CH(X^n)_k$ is $0$ for some $n$.
The surjectivity of \eqref{e:Chowpull} is clear, because each cycle in $CH(X \times_S S')_k$
is the pullback along $X \times_S S' \to X \times_S U$ of a cycle in
$CH(X \times_S U)_k$ for some non-empty open subscheme $U$ of $S$,
and $CH(X)_k \to CH(X \times_S U)_k$ is surjective.

Let $U \ne S$ be a non-empty open subscheme of $S$.
Then we can see as follows that there is an open subscheme $U'$ of $S$ strictly
containing $U$ such that for every $Y$ in $\sV_S$ and $\beta$ in
$CH(Y)_k$ with restriction $0$ to $Y \times_S U$, the restriction
of $\beta^{\otimes 2}$ to $Y^2 \times_S U'$ is $0$.
The reduced subscheme $W_1$ of $S$ with support $S - U$ has a non-empty connected
open subscheme $W$ which is regular and hence
lies in $\sS$, because $S$ and hence $W_1$ is excellent (\cite{EGAIV2},~7.8.6~(iii)).
Then $U \cup W$ is an open subscheme $U'$ of $S$
which is connected because $S$ is irreducible,
and $W \ne U'$ is a non-empty closed subscheme of $U'$ with complement $U$.
By taking $W$ sufficiently small we may suppose that the conormal
sheaf of $W$ in $U'$ is free.
Suppose that $\beta$ in $CH(Y)_k$ has restriction
$0$ to $Y \times_S U$.
Then \eqref{e:exacthomotopy} applied to the embedding
$i$ of $Y \times_S W$ into $Y \times_S U'$ shows that
$\beta$ has restriction $i_*(\delta)$ to $Y \times_S U'$ for some
$\delta$ in $CH(Y \times_S W)$.
If $j$ is the embedding of $Y^2 \times _S W$ into $Y^2 \times_S U'$,
then  by \eqref{e:basechange} we have
\[
(\pr_r)^*i_*(\delta) = j_*((\pr_r)^*(\delta))
\]
for $r = 1,2$, because the projections are flat.
Thus the restriction of $\beta^{\otimes 2}$ to $Y^2 \times_S U'$ is
\[
j_*((\pr_1)^*(\delta)).j_*((\pr_2)^*(\delta)) =
j_*(j^*j_*((\pr_1)^*(\delta)).((\pr_2)^*(\delta))) = 0
\]
by \eqref{e:selfint}, because the conormal sheaf of
$Y^2 \times _S W$ in $Y^2 \times_S U'$ is free of rank $> 0$.

Now let $\alpha$ be a cycle in the kernel of $CH(X)_k \to CH(X \times_S S')_k$.
Then for some non-empty open subscheme $U_0$ of $S$ the restriction
of $\alpha$ to $X \times_S U_0$ is $0$.
Since $S$ is noetherian, it follows inductively by the above that there is an
ascending chain of
open subschemes $U_0,U_1, \dots ,U_m = X$ of $X$ such that the restriction
of $\alpha^{\otimes 2^r}$ to $X^{2^r} \times_S U_r$ is $0$.
Taking $r = m$ now gives what is required.
\end{proof}

A motive in $\sM_{S,k}$ is called \emph{effective} if it is a direct summand
of $h(X)$ for some $X$ in $\sV_S$.
Since $\I$ is a direct summand of $h(\mathbf{P}^1{}\!_S)$, any direct summand
of $h(X)$ is also a direct summand of $h(X_1)$ with $X_1 = X \times_S \mathbf{P}^1{}\!_S$
of dimension $\dim X + 1$.
The full subcategory of $\sM_{S,k}$ consisting of the effective motives is thus pseudo-abelian,
a hence $k$\nd tensor subcategory, which we denote by $\sM_{S,k}^\mathrm{eff}$.
The motive $\I(-1)$ is effective,
because it is a direct summand of $h(\mathbf{P}^1{}\!_S)$.
If $M$ in $\sM_{S,k}$ is effective, then
\[
\sM_{S,k}(\I,M(i)) = 0
\]
for $i < 0$, because it is isomorphic when $M = h(X)$ to $CH^i(X)_k$.
Given $f:S' \to S$ in $\sS$ with $S'$ connected, $f^*$
preserves effective motives, as does $f_*$ when $f$ is proper and smooth.
If $M$ is a direct summand of $h(X)$ for some $X$ of dimension $\le d$, then $M$ and
$M^\vee(-d)$ are effective.
The converse holds when $S$ is the spectrum of a field, but this will not be needed.

For any integer $i \ge 0$, we define inductively as follows the notion
of effective motive of degree $\ge i$.
Every effective motive in $\sM_{S,k}$ will be said to be of degree $\ge 0$.
If $i > 0$,  a motive $M$ in $\sM_{S,k}$ will be called
\emph{effective of degree $\ge i$} if $M$ is effective and if
\[
\sM_{S,k}(\I,M \otimes L(i-1)) = 0
\]
for every effective motive $L$ in $\sM_{S,k}$ of degree $\ge i-1$.
An effective motive $M$ is of degree $\ge i$ if and only if $M(-1)$ is of degree $\ge i+2$.
Indeed the ``only if'' follows because $M(-1) \otimes L(i+1)$ is isomorphic to
$L \otimes M(i)$, and the ``if'' then follows because $M \otimes L(i-1)$ coincides
with $M(-1) \otimes L(-1)(i+1)$.
It is easily seen by induction on $i$ that $M$ effective of degree
$\ge i+1$ implies $M$ effective of degree $\ge i$, by showing at the same time that
$M$ effective of degree $\ge i$ implies $M(-1)$ effective of degree $\ge i+1$.
Similarly if $M$ is effective of degree $\ge i$ and $N$ is effective then
$M \otimes N$ is effective of degree $\ge i$.
In what follows essential use will be made of this notion only for $i = 0,1,2$.

For $i \ge 0$, a motive $M$ will be called \emph{effective of degree $\le i$}
if $M$ is effective and $M^\vee(-i)$ is effective of degree $\ge i$,
and \emph{effective of degree $i$} if it is effective of degree $\ge i$ and
of degree $\le i$.
For $M$ effective, $M$ of degree $\le i$ implies $M$ of degree $\le i+1$
and $M(-1)$ of degree $\le i+2$.
If $M$ is effective of degree $\ge i+1$ and $N$ is effective of degree $\le i$,
then $\sM_{S,k}(N,M) = 0$.
If both $M$ and $M^\vee(-i)$ are effective, and in particular if $M$
is a direct summand of $h(X)$ with $\dim X \le i$,
then $M$ is of degree $\le 2i$.

\begin{prop}\label{p:pureins}
If $S$ is the spectrum of a field and $f:S' \to S$ is defined by a
purely inseparable extension,
then  $f^*:\sM_{S,k} \to \sM_{S',k}$
is an equivalence of categories,
and it induces an equivalence between
the full subcategories of effective motives.
\end{prop}

\begin{proof}
We may assume that $S$ is of characteristic $> 0$ and $f$ is finite.
For every $X$ in $\sV_S$, the projection from $X \times_S S'$ to $X$
is then proper and bijective,
and hence induces an isomorphism from $CH(X)_k$ to $CH(X \times_S S')$.
Thus $f^*$ is fully faithful.
Since $f$ is defined by a finite purely inseparable extension,
a sufficiently high power $\Fr_{S'}^n$ of the Frobenius endomorphism
$\Fr_{S'}$ of $S'$ factors as
\[
S' \xrightarrow{f} S \to S'.
\]
Let $X'$ be a scheme in $\sV_{S'}$.
Then if $X$ is the pullback of $X'$ along $S \to S'$,
the pullback $\widetilde{X}{}'$ of $X'$ along $\Fr_{S'}^n$ coincides with
the pullback of $X$ along $f$.
By naturality of $\Fr$, the endomorphism $\Fr_{X'}^n$ of $X'$ factors
as a morphism $X' \to \widetilde{X}{}'$ over $S'$ followed by the projection
$\widetilde{X}{}' \to X'$.
Then $X' \to \widetilde{X}{}'$ is a universal homeomorphism.
Hence $h_{S'}(X')$ is isomorphic in $\sM_{S',k}$ to $f^*h_S(X)$.
Since $f^*$ is fully faithful, it follows that any direct summand $M'$ of
$h_{S'}(X')$ is isomorphic to $f^*M$ for some direct summand $M'$ of $h_S(X)$.
Thus $f^*$ is essentially surjective on full subcategories of
effective motives.
\end{proof}

Let $f:S' \to S$ be dominant with $S'$ in $\sS$ connected.
We can see as follows that $f^*:\sM_{S,k} \to \sM_{S',k}$ reflects isomorphisms,
sections and retractions.
Reduce by Proposition~\ref{p:genfib} to the case where $S$ is the
spectrum of a field, then by replacing $S'$ by an open subscheme
to the case where $S'$ is affine,
next by writing $S'$ as the limit of schemes of finite type over $S$
to the case where $f$ is of finite type, and finally replacing $S'$
by the reduced subscheme on a closed point to the case where $S'$ is defined
by a finite extension of fields, which may be assumed separable by
Proposition~\ref{p:pureins}.
It that case the unit $M \to f_*f^*M$ is a section
for every $M$, so that $f^*$ reflects sections.
Taking duals then shows that
$f^*$ also reflects retractions.
Similarly $f^*$ reflects Kimura objects:
again we reduce to the case where $f$ is defined by a finite separable
extension of fields, when $f^*M$ a Kimura object implies $f_*f^*M$ and hence $M$
a Kimura object.

Let $f:S' \to S$ be finite and \'etale with $S'$ connected.
Then $f_*(N'{}^\vee)$ is isomorphic to $f_*(N')^\vee$
and $f_*(N'(n))$ to $f_*(N')(n)$ for every motive $N'$ in $\sM_{S',k}$.
Since $f_*$ is right adjoint to $f^*$ and the units $M \to f_*f^*M$ are sections
and the counits $f^*f_*M' \to M'$ are retractions, it follows that
both $f^*$ and $f_*$ preserve and reflect effective motives of
degree $\ge i$ and effective motives of degree $\le i$.
By Proposition~\ref{p:pureins}, the same holds if $f:S' \to S$ is defined by
a finite extension of fields.
If $f$ is defined by an algebraic extension of fields, then every object
and every morphism in $\sM_{S',k}$ is the pullback from some finite subextension,
so that again $f^*$ preserves and reflects effective motives of
degree $\le i$ and of degree $\ge i$.

Given $X$ in $\sV_S$,
we have the Stein factorisation $X \to X_0 \to S$ of $X \to S$,
where $X_0 \to S$ finite and \'etale and
$p:X \to X_0$ is surjective with geometrically connected fibres.
Then $h(p):h(X_0) \to h(X)$ is a section,
and $h(X_0)$ is self-dual and hence effective of degree $\le 0$.
For any $Y$ in $\sV_S$,
the morphism $p \times_S Y$ is
proper and surjective with geometrically connected fibres, and hence
induces an isomorphism
from $CH^0(X_0 \times_S Y)_k$ to $CH^0(X \times_S Y)_k$.
Thus $h(p) \otimes h(Y)$ induces an isomorphism
\[
\sM_{S,k}(\I,h(X_0) \otimes h(Y)) \iso \sM_{S,k}(\I,h(X) \otimes h(Y)),
\]
so that the cokernel of $h(p)$ is effective of degree $\ge 1$.
Hence $h(p)$ is universal among morphisms in $\sM_{S,k}$
with target $h(X)$ and source an effective motive of degree $\le 0$.
Since every effective $M$ is a direct summand of some $h(X)$,
the embedding into $\sM_{S,k}^\mathrm{eff}$ of the full subcategory of
effective motives of degree $\le 0$ has a right adjoint $\tau_{\le 0}$,
and for each $M$ the counit $\tau_{\le 0}M \to M$ is a section with
cokernel effective of degree $\ge 1$.
An effective motive $M$ is of degree $\ge 1$ if and only if $\tau_{\le 0}M = 0$.
We have $\tau_{\le 0} h(X) = h(X_0)$,
so that the category of effective motives
of degree $0$ is the category of $k$\nd linear Artin motives over $S$,
i.e.\ the pseudo-abelian hull in $\sM_{S,k}$ of the full subcategory
consisting of the $h(X)$ with $X$ non-empty, finite and \'etale over $S$.
If $M$ and $M^\vee(-1)$ are effective, and in particular if $M = h(X)$
with $\dim X \le 1$, then applying $\tau_{\le 0}$ to $M^\vee(-1)$ gives
a decomposition
\begin{equation}\label{e:M1decomp}
M = M_0 \oplus M_1 \oplus M_2(-1)
\end{equation}
with $M_0$ and $M_2$ effective of degree $0$ and $M_1$ effective of degree $1$.
The functor $\tau_{\le 0}$ commutes with pullback along any $S' \to S$,
because the Stein factorisation does.
It follows that pullback preserves effective motives of degree $\ge 1$.

For any effective motive $M$ in $\sM_{S,k}$, we write $\tau_{\ge 1} M$
for the cokernel of the counit $\tau_{\le 0}M \to M$.
Then $M \to \tau_{\ge 1} M$ is universal among morphisms from $M$
to an effective motive of degree $\ge 1$.
If $L$ is an effective motive of degree $\ge 1$, there is a retraction
$h(X) \to L$ for some $X$ in $\sV_S$, which must factor
through a retraction $\tau_{\ge 1} h(X) \to L$.
Thus every effective motive of degree $\ge 1$ is a direct summand of some
$\tau_{\ge 1} h(X)$.
Suppose that $S$ is the spectrum of a separably closed field.
Then every $X$ in $\sV_S$ has an $S$\nd point,
and if $X$ is connected the morphism $p:h(X) \to \I$ defined by
any $S$\nd point of $X$ is left inverse to the counit
$\I = \tau_{\le 0} h(X) \to h(X)$, so that $h(X) \to \tau_{\ge 1} h(X)$ induces an isomorphism
from $\Ker p$ to $\tau_{\ge 1} h(X)$.
Now for $X_1$ and $X_2$ in $\sV_S$ with the same dimension, $h(X_1 \amalg X_2)$ is a direct summand
of $h(X_1 \times X_2)$, because $h(X_1)$ and $h(X_2)$ have the direct summand $\I$.
It follows that for $S$ the spectrum of a separably closed field,
any effective motive in $\sM_{S,k}$ of
degree $\ge 1$ is a direct summand of a motive
$\Ker p$ with $p:h(X) \to \I$ defined by an $S$\nd point of
a connected $X$ in $\sV_S$.

Denote by $\widetilde{S}$ the spectrum of the local ring of the generic point
of $S$ and by $j:\widetilde{S} \to S$ the embedding.
Then for $X$ in $\sV_S$, pullback along $j$ induces a bijection
from the set of connected components of $X$ to the set of connected components
of $X \times_S \widetilde{S}$, and hence an isomorphism from $CH^0(X)_k$ to
$CH^0(X \times_S \widetilde{S})_k$.
For $L$ an effective motive in $\sM_{S,k}$, the functor
$j^*:\sM_{S,k} \to \sM_{\widetilde{S},k}$ thus defines a natural isomorphism
\begin{equation}\label{e:jiso}
\sM_{S,k}(\I,L) \iso \sM_{\widetilde{S},k}(\I,j^*L).
\end{equation}
Taking $L = M \otimes N$ with $M$ and $N$ effective
then shows that if $M$ is effective and $j^*M$
is of degree $\ge 1$ then $M$ is of degree $\ge 1$.

Given $X$ in $\sV_S$ with Stein factorisation $X \to X_0 \to S$,
the reduced and irreducible closed subschemes of $X$ of
codimension $1$ with image in $S$
strictly contained in $S$ are the pullbacks along $X \to X_0$ of the
reduced and irreducible closed subschemes of $X_0$ of codimension $1$.
With $\widetilde{S}$ as above, we thus have a short exact sequence
\[
0 \to CH^1(X_0)_k \to CH^1(X)_k \to
CH^1(X \times_S \widetilde{S})_k \to 0.
\]
Hence for $L$ an effective motive in $\sM_{S,k}$ we have a short exact sequence
\begin{equation}\label{e:jsex}
0 \to \sM_{S,k}(\I,(\tau_{\le 0}L)(1)) \to \sM_{S,k}(\I,L(1))
\to \sM_{\widetilde{S},k}(\I,(j^*L)(1)) \to 0
\end{equation}
which is natural in $L$, where $j$ is the embedding.
If $L = M \otimes N$ with $M$ and $N$ effective and $N$ of degree
$\ge 1$, then $\tau_{\le 0} L = 0$ because $L$ is effective of degree $\ge 1$.
Thus if $M$ is effective and $j^*M$
is of degree $\ge 2$ then $M$ is of degree $\ge 2$.

\begin{prop}\label{p:degtens}
Let $i \ge 0$ and $j \ge 0$ be integers with $i+j \le 2$, and let
$M$ and $N$ be effective motives in $\sM_{S,k}$.
Then $M \otimes N$ is of degree $\le i+j$ if $M$ is of degree
$\le i$ and $N$ is of degree $\le j$, and $M \otimes N$ is
of degree $\ge i+j$ if $M$ is of degree~$\ge i$ and $N$ is of degree $\ge j$.
\end{prop}

\begin{proof}
By definition, an effective motive $L$ is of degree $\le r$ if and only if
$L^\vee(-r)$ is effective of degree $\ge r$.
Thus it is enough to prove that $M \otimes N$ is
of degree $\ge i+j$ if $M$ is of degree~$\ge i$ and $N$ is of degree $\ge j$.
The cases where $i$ or $j$ is $0$ have been seen.
It remains to consider the case where $i = j = 1$.

Suppose that both $M$ and $N$ are of degree $\ge 1$.
It is to be shown that $M \otimes N$ is of degree $\ge 2$.
Write $\widetilde{S}$ and $\overline{S}$ for the respective spectra of the local
ring of the generic point of $S$ and its algebraic closure.
It has been seen that an effective motive in $\sM_{S,k}$ is of degree $\ge 2$
when its pullback along $\widetilde{S} \to S$ is of degree $\ge 2$,
and similarly for pullback along $\overline{S} \to \widetilde{S}$.
Since pullback preserves effective motives of degree $\ge 1$,
we may thus suppose after replacing $S$ by $\overline{S}$ that $S$ is the spectrum
of an algebraically closed field.
We show in this case that if $L_1,L_2,L_3$ are effective motives in $\sM_{S,k}$
of degree $\ge 1$, then
\begin{equation}\label{e:Lcubed}
\sM_{S,k}(\I,L_1 \otimes L_2 \otimes L_3(1)) = 0.
\end{equation}
Taking $L_1 = M$, $L_2 = N$ will then give the required result.
We may suppose that $L_r = \Ker p_r$ for $r = 1,2,3$ with $p_r:h(X_r) \to \I$
defined by an $S$\nd point $x_r$ of a connected $X_r$
in $\sV_S$.
The left hand side of \eqref{e:Lcubed} may be then identified
by means of the natural isomorphism
\[
\sM_{S,k}(\I,h(X_1) \otimes h(X_2) \otimes h(X_3)(1)) \iso
CH^1(X_1 \times_S X_2 \times_S X_3)_k,
\]
with the $k$\nd subspace of
$CH^1(X_1 \times_S X_2 \times_S X_3)_k$
consisting of those elements with restriction $0$ to the subscheme
with $r$th coordinate $x_r$ for $r = 1,2,3$.
By the theorem of the cube, this $k$\nd subspace is $0$ as required.
\end{proof}

Suppose that $S$ is the spectrum of a field.
Call a $Z$ in $\sV_S$ almost abelian if for some $S'$ over $S$ defined by a
finite separable extension $Z \times_S S'$ can
be given a structure of abelian variety over $S'$.
It is equivalent to require that $Z$ can be given a structure
of principal homogeneous space under some abelian variety $A$ over $S$.
Let $X$ in $\sV_S$ be geometrically connected.
Then there exists a morphism
\[
X \to \Alb^1(X)
\]
which is universal among morphisms $X \to Z$ in $\sV_S$ with $Z$ almost abelian,
and its formation commutes with extension of scalars.
We briefly recall its construction.
By Galois descent we may suppose that $X$ has an $S$\nd point $x$.
Write $\sP(X)$ for the reduced subscheme of the Picard scheme of $X$ with support the
identity component.
It is an abelian variety and represents the functor on
reduced connected pointed schemes $(W,w)$ over $S$ that
sends $(W,w)$ to the subgroup of $\Pic(X \times_S W)$ trivial on $x \times W$
and $X \times w$.
If we take $\Alb^1(X) = \sP\sP(X)$, then the morphism $a:X \to \Alb^1(X)$ defined
by the universal element of $\Pic(X \times_S \sP(X))$ has the required universal property.
Also $a$ induces an isomorphism from $\sP(\Alb^1(X))$ to $\sP(X)$,
and hence for any reduced connected pointed scheme $(W,w)$ over $S$ an isomorphism from the
subgroup of $\Pic(\Alb^1(X) \times_S W)$
trivial on $a(x) \times W$ and $\Alb^1(X) \times w$ to the subgroup of $\Pic(X \times_S W)$
trivial on $x \times W$ and $X \times w$.

\begin{thm}\label{t:abKim}
Every effective motive of degree $0$ in $\sM_{S,k}$ is positive and
every effective motive of degree $1$ in $\sM_{S,k}$ is negative.
\end{thm}

\begin{proof}
Pulling back along $S' \to S$ with $S'$ the spectrum of an algebraic
closure of the local ring of the generic point of $S$, we may
assume that $S$ is the spectrum of an algebraically closed field.
An effective motive of degree $0$ is then a direct sum of objects $\I$,
and hence is positive.

Let $N$ be an effective motive of degree $1$.
Since $N^\vee(-1)$ and $N$ are effective of degree $\ge 1$, there exist
connected $X$ and $Y$ in $\sV_S$ and $S$\nd points
$x:S \to X$ and $y:S \to Y$ such that $N^\vee(-1)$ is a direct summand of $\Ker h(x)$
and $N$ of $\Ker h(y)$.
Then we have sections \mbox{$u:N^\vee(-1) \to h(X)$} and $v:N \to h(Y)$ with
$h(x) \circ u = 0$ and $h(y) \circ v = 0$.
If $X$ has dimension $d$ and we write
\[
f = v \circ u^\vee(-1):h(X)(d-1) \to h(Y),
\]
it follows that $h(x) \circ f^\vee(-1) = 0$ and $h(y) \circ f = 0$.
Since $u^\vee(-1)$ is a retraction and $v$ is a section, $S^n N = 0$
if and only if $S^n f = 0$.
Now $f = \gamma_{X,Y,d - 1,0}(\alpha)$ for some $\alpha$
in $CH^1(X \times_S Y)_k$, and by \eqref{e:PDhomcomppull} and
Proposition~\ref{p:PDhom}~\ref{i:PDhomtrans} the restrictions of $\alpha$
to $x \times Y$ and $X \times y$ are $0$.
By Proposition~\ref{p:PDhom}~\ref{i:PDhomtens}, \eqref{e:PDhomcomppush} and
\eqref{e:PDhomcomppull}, $S^n f$ is $0$ if and only if
\begin{equation}\label{e:XKimura}
\sum_{\sigma, \tau \in \mathfrak{S}_n}
\alpha_{\sigma 1, \tau 1}.\alpha_{\sigma 2, \tau 2}.
\cdots .\alpha_{\sigma n, \tau n},
\end{equation}
in $CH^n(X^n \times_S Y^n)_k$ is $0$, where $\alpha_{ij}$ denotes
the pullback of $\alpha$ along the product
$X^n \times_S Y^n \to X \times_S Y$ of the $i$th and $j$th projection.
To prove that $N$ is negative, it thus suffices to show that
\eqref{e:XKimura} is $0$ for some $n$.

Write $A$ for $\Alb^1(X)$ and $B$ for $\Alb^1(Y)$.
We may regard $A$ and $B$ as abelian varieties over $S$ with identities the
images of $x$ and $y$ under the canonical morphisms $a:X \to A$ and $b:Y \to B$.
Factoring $a \times b$ as $(A \times b) \circ (a \times Y)$
shows that pullback along $a \times b$ induces an isomorphism from the subspace
of $CH^1(A \times_S B)_k$ trivial on $0 \times B$ and $A \times 0$ to the subspace
of $CH^1(X \times_S Y)_k$ trivial on $x \times Y$ and $X \times y$.
Thus $\alpha$ is the pullback along $a \times b$ of a $\pi$ in
$CH^1(A \times_S B)_k$ trivial on $0 \times B$ and $A \times 0$.
The pullback of $\pi$ along
the product $A^r \times_S B^s \to A \times_S B$ of the $r$\nd fold addition of
$A$ and $s$\nd fold addition of $B$
is $\sum_{1 \le i \le r,1 \le j \le s} \pi_{ij}$, with $\pi_{ij}$
the pullback of $\pi$ along the product  of the $i$th and $j$th projection.
This can be seen by induction on $r$ and the seesaw principle, by
restricting to subschemes of $A^r \times_S B^s$ with all coordinates
constant but the first and noting that
the restriction $\pi_z$ of $\pi$ to $A \times z = A$ is translation invariant
with $\pi_z + \pi_{z'} = \pi_{z+z'}$.
Now \eqref{e:XKimura} coincides with
\begin{equation}\label{e:inclexcl}
\sum_{I,J \subset [1,n]}(-1)^{|I|+|J|}
(\sum_{i \in I, j \in J} \alpha_{ij})^n.
\end{equation}
Indeed \eqref{e:inclexcl} is a sum of terms
$c\alpha_{i_1j_1}.\alpha_{i_2j_2}. \cdots .\alpha_{i_nj_n}$,
where $c$ is $(-1)^{|I|+|J|}$ summed over
the $I$ that contain $I_0 = \{i_1,i_2, \dots ,i_n\}$
and the $J$ that contain $J_0 = \{j_1,i_2, \dots ,j_n\}$, and hence $c = 1$
when $I_0 = J_0 = [1,n]$ and $c = 0$ otherwise.
Each $n$th power on the right of \eqref{e:inclexcl} is a pullback of $\pi^n$
in $CH^n(A \times _S B)_k$.
Thus for $n$ greater than the dimension of  $A \times_S B$,
\eqref{e:inclexcl} and hence \eqref{e:XKimura} is~$0$.
\end{proof}

\begin{thm}\label{t:deg1ss}
For $i = 0,1$, the full subcategory of $\sM_{S,k}$ consisting of the effective
motives of degree $i$ is semisimple abelian.
\end{thm}

\begin{proof}
Let $M$ and $N$ be effective motives of degree $i$, where $i$ is $0$ or $1$.
It is enough by Theorem~\ref{t:abKim} to show that any morphism
$M \to N$ in $\sM_{S,k}$ with image $0$ in $\overline{\sM}_{S,k}$ is $0$.
If $j:\widetilde{S} \to S$ is the embedding of the spectrum of the
local ring of the generic point of $S$,
then \eqref{e:jiso} and \eqref{e:jsex} with $L$ the effective motive
$N \otimes M^\vee(-i)$ of degree $\ge i$
show that $j^*$ induces an isomorphism from $\sM_{S,k}(M,N)$ to
$\sM_{\widetilde{S},k}(j^*M,j^*N)$.
Since by Proposition~\ref{p:genfib} $j^*$ induces a fully faithful functor from
$\overline{\sM}_{S,k}$ to $\overline{\sM}{}_{\widetilde{S},k}$,
we may thus after replacing $S$, $M$ and $N$ by $\widetilde{S}$,
$j^*M$ and $j^*N$ suppose that $S$ is the spectrum of a field.
Passing to an algebraic closure, we may further suppose that
$S$ is the spectrum of an algebraically closed field.
The case $i = 0$ is clear because $M$ and $N$ are direct sums of copies of $\I$.

Let $L_1$ and $L_2$ be effective motives
of degree $\ge 1$.
We show that any morphism
\begin{equation}\label{e:Lsquared}
\I \to L_1 \otimes L_2(1)
\end{equation}
in $\sM_{S,k}$ with image $0$ in $\overline{\sM}_{S,k}$ is $0$.
The case $i = 1$ will follow by taking $L_1 = N$ and $L_2 = M^\vee(-1)$.
We may suppose that $L_r = \Ker p_r$ for $r = 1,2$ with $p_r:h(X_r) \to \I$
defined by an $S$\nd point $x_r$ of a connected $X_r$ in $\sV_S$.
The $k$\nd vector space of morphisms \eqref{e:Lsquared} may be then identified
by means of the natural isomorphism
\[
\sM_{S,k}(\I,h(X_1) \otimes h(X_2)(1)) \iso CH^1(X_1 \times_S X_2)_k
\]
with the $k$\nd subspace of $CH^1(X_1 \times_S X_2 )_k$
consisting of those elements with restriction $0$ to the subscheme
with $r$th coordinate $x_r$ for $r = 1,2$.
Now for any connected $Z$ in $\sV_S$ the kernel of the projection
$CH^1(Z)_k \to \overline{CH}{}^1(Z)_k$ is the $k$\nd subspace
\mbox{$\Pic^0(Z) \otimes_\Z k$}
of $CH^1(Z)_k = \Pic(Z) \otimes_\Z k$, where $\Pic^0$ denotes the group
of divisor classes algebraically equivalent to $0$ (see e.g.\ \cite{SGA6},~XIII~4.6).
Since $\Pic^0(X_1 \times_S X_2)$ is the product of $\Pic^0(X_1)$ and $\Pic^0(X_2)$,
the required result follows.
\end{proof}

A bilinear form $\varphi:M \otimes N \to \I(n)$ in a Tate $k$\nd pretensor
category $\sC$ will be called non-degenerate if there is a morphism
$\psi:\I(n) \to N \otimes M$ such that $\psi(-n)$ and $\varphi(-n)$ are the unit
and counit of a duality pairing between $M$ and $N$.
Similarly $\psi$ will be called non-degenerate when there exists a $\varphi$.
If $M$ has a dual $M^\vee$, then by uniqueness up to unique isomorphism of duals,
$\varphi$ is non-degenerate if and only if the morphism $N(-n) \to M^\vee$ induced by
$\varphi$ is an isomorphism, and similarly for $\psi$.
When $M = N$, the form $\varphi$ is symmetric
if and only if $\psi$ is symmetric.
We then identify $\varphi$ with the morphism $S^2N \to \I(n)$ through which it
factors, and similarly for $\psi$.

Suppose that $\sC$ is a Kimura $k$\nd tensor category and that $N$ is a negative
object in $\sC$ of rank $-2r$.
Then $\psi:\I(n) \to S^2N$ is non-degenerate if and only if
\begin{equation}\label{e:Pfaff}
\psi^{(r)}:\I(rn) \to S^{2r}N,
\end{equation}
defined as the composite of $\psi^{\otimes r}$ with
$(S^2 N)^{\otimes r} \to S^{2r}N$
given by the algebra structure of $\Sym N$, is an isomorphism.
To see this we may suppose after replacing $\sC$ by $\sC_{\mathrm{red}}$ that
$\sC$ is split.
After modifying the symmetry of $\sC$, it is then equivalent to show that
if $\sC$ is a positive Kimura category and $N$ is of rank $2r$, the form
\mbox{$\psi:\I(n) \to \bigwedge^2N$} is non-degenerate if and only if
$\psi^{(r)}:\I(rn) \to \bigwedge^{2r}N$ defined using the algebra structure of
$\bigwedge N$ is an isomorphism.
As in \ref{ss:Tate}, we may suppose that $\sC$ is a separated Tate $k$\nd tensor category.
By Corollary~\ref{c:posrep} and Lemma~\ref{l:compatibletwist},
there is thus a faithful conservative Tate $k$\nd tensor functor from $\sC$
to the category of finitely generated projective modules over a commutative
$k$\nd algebra $R$ with \mbox{$\I(1) = R$}.
It is enough to prove that $\psi:R \to \bigwedge^2 R^{2r}$
is non-degenerate if and only if $\psi^{(r)}:R \to R$ is an isomorphism.
This is clear because if we identify $\psi$ with a skew-symmetric $2r \times 2r$
matrix, then $\psi^{(r)}$ is $2^rr!$ times the Pfaffian of $\psi$,
while $\psi$ is non-degenerate if and only if the determinant of $\psi$,
which is the square of the Pfaffian, is invertible.

By an abelian scheme over $S$ we mean a proper and smooth group scheme $A$ over $S$
with geometrically connected fibres.
For such an $A$ it can be seen as follows that $h(A)$ is a Kimura object in $\sM_{S,k}$.
Since pullback along a dominant morphism reflects Kimura objects,
we may after replacing $S$ by the spectrum of an algebraic closure
of the local ring of its generic point suppose that $S$ is the spectrum of an
algebraically closed field.
Then there is a surjective morphism $X \to A$ over $S$
with $X$ a product of proper smooth curves $X_i$ over $S$.
Since $h(A) \to h(X)$ is a section, $h(A)$ is a direct summand of $h(X)$.
By \eqref{e:M1decomp} and Theorem~\ref{t:abKim} each $h(X_i)$ is a Kimura object.
Thus $h(X)$ and hence $h(A)$ a Kimura object.

Since the group law on an abelian scheme $A$  over $S$ is commutative, $h(A)$
has a structure of commutative and cocommutative Hopf algebra in $\sM_{S,k}$.
For $n$ an integer, we write the multiplication by $n$ on $A$ as $n_A$.
Then $h(n_A) = n_{h(A)}$.
If $n \ne 0$, then $n_A:A \to A$ is surjective,
so that the endomorphism $n_{h(A)}$ of $h(A)$ is a section, and hence since $h(A)$ is a
Kimura object an isomorphism, in $\sM_{S,k}$.

The following result is essentially equivalent to the one proved by
K\"unnemann \cite{KunAbSch} using the Fourier-Mukai transform for abelian schemes.
We deduce it here from the Hopf theorem for Kimura categories.

\begin{thm}\label{t:abeliansym}
Let $A$ be an abelian scheme over $S$ of relative dimension $g$.
Then the Hopf algebra $h(A)$ in $\sM_{S,k}$ is symmetric.
The component $h^1(A)$ of $h(A)$ on which $n_{h(A)}$ acts as $n$ for every $n$
is a negative object of rank $-2g$,
and there exists a non-degenerate symmetric bilinear form $S^2 h^1(A) \to \I(-1)$.
\end{thm}

\begin{proof}
Since $h(A)$ is a Kimura object in $\sM_{S,k}$ with $n_{h(A)}$ an isomorphism
for $n \ne 0$, it is symmetric Hopf algebra by Theorem~\ref{t:Hopf}.

To prove the final two statements, we may suppose using
Proposition~\ref{p:genfib} that $S$ is the spectrum of a field.
Write $h^i(A)$ for the direct summand of $h(A)$ on which $h(n_A)$ acts as $n^i$
for each $n$.
Then using the algebra structure of $h(A)$ we may identify
$S^i h^1(A)$  with $h^i(A)$ for each $i$.
Since $(n_A)_*(1) = n^{2g}$ in $CH^0(A)_k = k$ for every integer $n$,
and since by Proposition~\ref{p:PDhom}~\ref{i:PDhomtrans} and \ref{i:PDhomgraph}
the transpose $\nu_A$ of the identity of $h(A)$ is $\gamma_{A,S,g,0}(1)$,
it follows from \eqref{e:PDhomcomppush} the composite of $n_{h(A)}(g)$ with
$\nu_A$ is $n^{2g}\nu_A$.
Thus $\nu_A$ is non-zero on the component $h^{2g}(A)(g)$ of $h(A)(g)$
and zero on all other components.
Since the pairing defined by composing the $g$th twist of multiplication of
$h(A)$ with $\nu_A$ is non-degenerate,
it follows that $S^ih^1(A) = h^i(A)$
is non-zero for $i = 2g$ and zero for $i > 2g$.
Thus $h^1(A)$ has rank $-2g$.

Let $\alpha$ be the class in $CH^1(A)_k$ of a symmetric ample divisor,
and write
\[
a = \gamma_{S,A,-1,0}(\alpha):\I(-1) \to h(A)
\]
and $a_i$ for the component $\I(-1) \to h^i(A)$ of $a$.
Then we have $a_0 = 0$ because $\sM_{S,k}(\I(-1),\I) = 0$,
and $a_1 = 0$ because
$(-1)_{h(A)} \circ a = a$ by symmetry of $\alpha$.
Thus
\[
a_2{}\!^{(g)}:\I(-g) \to S^{2g} h^1(A)
\]
as in \eqref{e:Pfaff} is the only non-zero component of  $a^g:\I(-g) \to h(A)$.
By \eqref{e:PDhomtwist}, $a^g = \gamma_{S,A,-g,0}(\alpha^g)$
and  $\nu_A(-g) = \gamma_{A,S,0,-g}(1)$.
Also $\alpha^g$ in $CH^g(A)_k$
has non-zero push forward along \mbox{$CH^g(A)_k \to CH^0(S)_k = k$},
because $\alpha$ is the class of an ample divisor.
Thus by Proposition~\ref{p:PDhom}~\ref{i:PDhomcomp} the endomorphism
$\nu_A(-g) \circ a^g$ of $\I(-g)$ is non-zero and hence an isomorphism.
Since $S^{2g} h^1(A)$ is positive of rank $1$, it follows that $a_2{}\!^{(g)}$
is an isomorphism.
This implies as was seen above that $a_2:\I(-1) \to S^2 h^1(A)$ is non-degenerate.
Thus there exists a non-degenerate $S^2 h^1(A) \to \I(-1)$.
\end{proof}

For $A$ as in Theorem~\ref{t:abeliansym} we write $h^i(A)$ for the
direct summand of $h(A)$ on which $h(n_A)$ acts as $n^i$ for each $n$.
The multiplication of $h(A)$ then defines an isomorphism
$S^i h^1(A) \iso h^i(A)$ for each $i$.
Since $h^1(A)$ is negative of rank $-2g$ and a non-degenerate bilinear form
$S^2 h^1(A) \to \I(-1)$ exists, $h^{2g}(A) = S^{2g}h^1(A)$ is isomorphic to $\I(-g)$.
The generator $\nu_A$ of the $1$\nd dimensional $k$\nd vector space
$\sM_{S,k}(h(A)(g),\I)$ thus factors through an isomorphism
$h^{2g}(A)(g) \iso \I$.
Further $h^1(A)^\vee(-1)$ is isomorphic to $h^1(A)$.
Thus $h^0(A) = \I$ is effective of degree $0$ and $h^1(A)$ is
effective of degree $1$, so that by Proposition~\ref{p:degtens}
$h^i(A)$ is effective of degree $\ge 2$ when $i \ge 2$.

Suppose that $S$ is the spectrum of a field, and let $Z$ be a principal
homogeneous space under an abelian variety $A$ over $S$.
Then $h(Z)$ and $h(A)$ are isomorphic as objects and even as algebras
in $\sM_{S,k}$.
Indeed the class of $Z$ in the Weil--Ch\^atelet group of $A$ is annulled
by some integer $n > 0$.
The push forward $Z'$ of $Z$
along $n_A:A \to A$ is then isomorphic to $A$, while $Z \to Z'$
induces an isomorphism $h(Z') \iso h(Z)$
because $h(Z' \times_S S') \to h(Z \times_S S')$ is an isomorphism
for some $S'$.

For completeness we prove in Theorem~\ref{t:trunc} below
the existence of a right adjoint  $\tau_{\le 1}$ analogous to $\tau_{\le 0}$.
This result will not be required for what follows.
Similar results have been obtained by a different method by Murre \cite{MurSurf}.

\begin{thm}\label{t:trunc}
\begin{enumerate}
\item\label{i:truncexist}
For $i = 0,1$, the embedding into $\sM_{S,k}^\mathrm{eff}$ of the full
subcategory of effective motives of degree $\le i$  has a right adjoint
$\tau_{\le i}$.
\item\label{i:truncsection}
For  $i = 0,1$ and every $M$ in $\sM_{S,k}^\mathrm{eff}$, the counit
$\tau_{\le i} M \to M$ is a section and its cokernel is of degree $\ge i+1$.
\item\label{i:truncstein}
If $X$ is in $\sV_S$ and $X \to X_0 \to S$ is the Stein
factorisation of $X \to S$, then $X \to X_0$ defines an isomorphism
$h(X_0) \iso \tau_{\le 0} h(X)$.
\item\label{i:truncalb}
If $S$ is the spectrum of a field and $X$ in $\sV_S$ is
geometrically connected over $S$,
then the canonical morphism $X \to \Alb^1(X)$ defines an
isomorphism $\tau_{\le 1}h(\Alb^1(X)) \iso \tau_{\le 1} h(X)$.
\end{enumerate}
\end{thm}

\begin{proof}
It has already been seen that
\ref{i:truncexist} and \ref{i:truncsection} for $i = 0$ and
\ref{i:truncstein} hold.
Call a morphism $u:N \to M$ in $\sM_{S,k}^\mathrm{eff}$
$1$\nd universal if $N$ is of degree $\le 1$ and  $u$ is a section
with cokernel of degree $\ge 2$.
Such a $u$ is universal among morphisms $N' \to M$ to $M$ with $N'$ effective of
degree $\le 1$.
To prove  \ref{i:truncexist} and \ref{i:truncsection} for $i = 1$ and \ref{i:truncalb},
it will thus suffice to show the following.
\begin{enumerate}
\renewcommand{\theenumi}{(\arabic{enumi})}
\item\label{i:1truncexist}
If $X$ in $\sV_S$ has
geometrically connected fibres, there is a $1$\nd universal morphism
with target $h(X)$.
\item\label{i:1truncalb}
For $S$ and $X$ as in \ref{i:truncalb}, the composite with $h(\Alb^1(X)) \to h(X)$
of a $1$\nd universal morphism with target $h(\Alb^1(X))$ is $1$\nd universal.
\end{enumerate}
Indeed since $1$\nd universal morphisms are preserved
by finite \'etale push forward, \ref{i:1truncexist} will then hold for an
arbitrary $X$ in $\sV_S$, by the Stein factorisation.

Suppose first that $S$ is the spectrum of a field, and let
$X$ be as in \ref{i:1truncexist} and \ref{i:1truncalb}.
Then $\Alb^1(X)$ is a principal homogeneous
space under an abelian variety $A$ over $S$, and $h(\Alb^1(X))$ is isomorphic
to $h(A)$.
It has been seen that
the embedding
\[
e:\I \oplus h^1(A) = h^0(A) \oplus h^1(A) \to h(A)
\]
is $1$\nd universal.
Hence the composite of $e$ with an isomorphism $h(A) \iso h(\Alb^1(X))$ is
$1$\nd universal.
Thus \ref{i:1truncexist} will follow in this case from \ref{i:1truncalb},
in the stronger form that  there is a $1$\nd universal morphism
$\I \oplus N \to h(X)$ with  $N$ effective of degree $\ge 1$ and $N^\vee(-1)$
isomorphic to $N$.

To prove \ref{i:1truncalb} we may suppose that $S$ is the spectrum of an
algebraically closed field, because pullback along an algebraic extension of fields
preserves and reflects $1$\nd universal morphisms.
Then $X$ has an $S$\nd point $x$.
We may identify $\Alb^1(X)$ with $A$ in such a way that the canonical morphism
$a:X \to A$ sends $x$ to the identity of $A$.
It is then to be shown that
\[
h(a) \circ e:\I \oplus h^1(A) \to h(X)
\]
is $1$\nd universal.
For some $n$ the composite $X^n \to A^n \to A$ defined using
the multiplication of $A$ is surjective.
The composite
\begin{equation}\label{e:hAhXn}
h(A) \to h(A)^{\otimes n} \to h(X)^{\otimes n}
\end{equation}
is then a section.
Write $L$ for the kernel of the retraction $h(X) \to \I$ defined by $x$.
The unit $\I \to h(X)$ and $L \to h(X)$ give a direct sum decomposition
$h(X) = \I \oplus L$,
and $h(a)$ sends $\I$ to $\I$ and $h^1(A)$ to $L$
because $h^1(A)$ is contained in the kernel of $h(A) \to \I$ defined by the
identity of $A$.
Thus we have a commutative diagram
\[
\begin{CD}
\I \oplus L  @>>>  \I \oplus L^{\oplus n}  @>>>
   (\I \oplus L)^{\otimes n} @= h(X)^{\otimes n}  \\
@AAA        @AAA      @AAA      @AAA             \\
\I \oplus h^1(A) @>>> \I \oplus h^1(A)^{\oplus n} @>>>
   (\I \oplus h^1(A))^{\otimes n}  @>>>  h(A)^{\otimes n}
\end{CD}
\]
where for example $L^{\oplus n}$ consists of the $n$ summands of
$(\I \oplus L)^{\otimes n}$ with exactly one factor $L$ and
$L \to L^{\oplus n}$ is the diagonal embedding.
The left vertical arrow of this diagram is $h(a) \circ e$,
and its bottom right leg is $e$ composed with \eqref{e:hAhXn}.
Since both $e$ and \eqref{e:hAhXn} are sections, it follows that $h(a) \circ e$
is a section.

Let $v:M' \to M''$ be a section in $\sM_{S,k}$ with $M'$ and $M''$ effective and
$M'$ of degree $\le 1$.
Then $v$ is $1$\nd universal if and only if the homomorphism
\[
\theta_{v,M}:\sM_{S,k}(\I,M' \otimes M(1)) \to \sM_{S,k}(\I,M'' \otimes M(1))
\]
induced by $v$ is an isomorphism for every effective $M$
of degree $\ge 1$.
Since $e$ is $1$\nd universal and
$\theta_{h(a) \circ e,M} = \theta_{h(a),M} \circ \theta_{e,M}$,
the morphism $h(a) \circ e$ will thus be $1$\nd universal provided that
$\theta_{h(a),M}$ is an isomorphism for every effective $M$ of degree $\ge 1$.
We may suppose that $M = \Ker p$,
where $p:h(Y) \to \I$ is defined by an $S$\nd point $y$ of
a connected $Y$ in $\sV_S$.
We have an exact sequence
\[
0 \to \sM_{S,k}(\I,h(Z) \otimes (\Ker p)(1)) \to CH^1(Z \times_S Y)_k \to CH^1(Z)_k,
\]
natural in $Z$ in $\sV_S$, with the last arrow defined by $y$.
The universal property of the Picard scheme $B$ of $Y$
thus gives an isomorphism
\[
\Hom_S(Z,B) \otimes_\Z k \iso \sM_{S,k}(\I,h(Z) \otimes (\Ker p)(1))
\]
which is natural in $Z$.
The universal
property of $a:X \to A$ applied to the connected components of
$B_\mathrm{red}$ then shows that $\theta_{h(a),\Ker p}$ is an isomorphism.
Thus $h(a) \circ e$ is $1$\nd universal.
This proves \ref{i:1truncalb}.

To prove \ref{i:1truncexist} for arbitrary $S$,
write $\widetilde{S}$ for the spectrum of the local ring of the generic
point of $S$ and $j:\widetilde{S} \to S$ for the inclusion.
It follows from \ref{i:1truncalb} that $h_{\widetilde{S}}(X \times_S \widetilde{S})$
is the target of a $1$\nd universal morphism $\widetilde{u}$ in $\sM_{\widetilde{S},k}$
with source $\I \oplus \widetilde{N}$,
with $\widetilde{N}$ effective of degree $\ge 1$ and $\widetilde{N}^\vee(-1)$
isomorphic to $\widetilde{N}$.
By Proposition~\ref{p:genfib}, there is a section
$u:\I \oplus N \to h_S(X)$ in $\sM_{S,k}$ with $j^*(N)$ isomorphic to $\widetilde{N}$
and  $j^*(u)$ isomorphic as a morphism with target $h_{\widetilde{S}}(X \times_S \widetilde{S})$
in $\sM_{\widetilde{S},k}$ to $\widetilde{u}$.
Now for $M$ in $\sM_{S,k}$ effective and $i = 1,2$, it has been seen that
$M$ is of degree $\ge i$ when $j^*(M)$ is of degree $\ge i$.
The cokernel of $u$ is thus effective of degree $\ge 2$, and
$N$ is effective of degree $\ge 1$, so that $N$ is effective of degree $\le 1$
since $N^\vee(-1)$ is isomorphic to $N$.
Hence $u$ is $1$\nd universal.
This proves \ref{i:1truncexist}.
\end{proof}

It follows from Theorem~\ref{t:trunc}~\ref{i:truncexist} and \ref{i:truncsection} that for
$i = 0,1,2$ the embedding into $\sM_{S,k}^\mathrm{eff}$
of the full subcategory of effective motives of degree $\ge i$  has a left adjoint
$\tau_{\ge i}$, and that for $i = 0,1$ the cokernel of $\tau_{\le i} M \to M$
is $M \to \tau_{\ge i+1} M$.
If $M$ and $M^\vee(-d)$ are effective, and in particular if $M$ is a direct summand
of $h(X)$ with $\dim X \le d$,
then applying $\tau_{\le 0}$ and $\tau_{\le 1}$ to $(\tau_{\ge 2} M)^\vee(-d)$ gives
a decomposition
\[
M = M_0 \oplus M_1 \oplus M_2 \oplus M_3(-d+1) \oplus M_4(-d)
\]
of $M$, with $M_0$ and $M_4$ effective of degree $0$,
$M_1$ and $M_3$ effective of degree $1$,
and $M_2$ and $M_2{}\!^\vee(-d)$ effective of degree $\ge 2$.

It is clear that $\tau_{\le 0}$ and $\tau_{\le 1}$ commute with
push forward along a finite \'etale morphism $S' \to S$,
and it has been seen that $\tau_{\le 0}$ commutes with arbitrary pullback.
Theorem~\ref{t:trunc} implies as follows that $\tau_{\le 1}$
commutes with arbitrary pullback, so that pullback preserves effective motives
of degree $\ge 2$.
It is enough to show that for every $f:S' \to S$ and $X$ the canonical morphism
from $f^*\tau_{\le 1}h_S(X)$ to $\tau_{\le 1}f^*h_S(X)$
is an isomorphism.
Since this morphism is by Theorem~\ref{t:trunc}~\ref{i:truncsection} a
section, it thus suffices to show that ranks $r_+$ and $r_-$ of the positive and negative
summands of $\tau_{\le 1}h_S(X)$ and $r'{}\!_+$ and $r'{}\!_-$ of the positive and negative
summands of $\tau_{\le 1}h_{S'}(X \times_S S')$ coincide.
To do this we may suppose that $S'$ is the spectrum of a field.
Since $\tau_{\le 1}$ commutes with finite \'etale push forward, we may
also suppose that $X$ has geometrically connected fibres.
If $\Alb^1(X \times_S S')$ is a principal homogeneous space under $A'$, then
$\tau_{\le 1}h_{S'}(X \times_S S')$ is isomorphic to $\I \oplus h^1(A')$.
Thus by Theorem~\ref{t:trunc}~\ref{i:truncalb}, $r_+ = r'{}\!_+ = 1$,
while
\[
r'{}\!_- = -2\dim \Alb^1(X \times_S S'),
\]
which is independent of $S'$ and $f$.
Thus we may suppose finally that $f$ is the embedding of the spectrum of the
local ring of the generic point of $S$.
In that case it suffices to note that by Proposition~\ref{p:genfib}
the direct summand $h^1(A')$ of
$h_{S'}(X \times_S S')$ lifts to a direct summand $N$ of $h_S(X)$,
with $N^\vee(-1)$ isomorphic to $N$ because $h^1(A')^\vee(-1)$ is isomorphic
to $h^1(A')$.

Let $\sJ$ be a proper tensor ideal of $\sM_{S,k}$, and $\sM'$ be a
pseudo-abelian hull of $\sM_{S,k}/\sJ$.
Then $\sM'$ is the category of $k$\nd linear motives over $S$
modulo the equivalence defined by $\sJ$.
We may define in the same way as for $\sM_{S,k}$ a notion of effective
motive, and effective motive of degree $\ge i$ and $\le i$, in $\sM'$.
For $i = 0,1,2$, it can be seen inductively that a motive in $\sM'$
is effective of degree $\ge i$ if and only if it is a direct summand
of the image in $\sM'$ of an effective motive in $\sM_{S,k}$ of degree $\ge i$.
Suppose for example that $i$ is $1$ or $2$ and that $M'$ in $\sM'$ is effective
of degree $\ge i$.
Then if $P:\sM_{S,k} \to \sM'$ is the projection, there is an effective $M$
in $\sM_{S,k}$ such that we have a retraction $P(M) \to M'$ in $\sM'$.
The cokernel $M_1$ of the section $\tau_{\le i-1}M \to M$ is effective of
degree $\ge i$, and
\[
P(\tau_{\le i-1}M) \to P(M) \to M'
\]
is $0$ in $\sM'$ because by induction $P(\tau_{\le i-1}M)^\vee(-i+1)$
is effective of degree $\ge i-1$.
Thus $P(M) \to M'$ factors through a retraction $P(M_1) \to M'$,
and $M'$ is indeed a direct summand of an object $P(M_1)$ with $M_1$ effective of
degree $\ge i$.
It follows from this that Proposition~\ref{p:degtens} and Theorems~\ref{t:abKim},
\ref{t:deg1ss} and \ref{t:trunc} hold with $\sM_{S,k}$ replaced by $\sM'$,
and that $P$ commutes with the $\tau_{\le i}$ for $i = 0,1$.
Similarly if $k'$ is an extension of $k$, then $\sM_{S,k} \to \sM_{S,k'}$ commutes
 with the $\tau_{\le i}$ for $i = 0,1$.

\subsection{Abelian motives}\label{ss:Abmot}

We denote by $\sM_{S,k}^\mathrm{ab}$ the strictly full rigid $k$\nd tensor subcategory
of $\sM_{S,k}$ generated by $\I(1)$ and the effective motives of degree either $0$ or $1$.
By Theorem~\ref{t:abKim}, it is a Kimura $k$\nd tensor category.
Equivalently by the decomposition \eqref{e:M1decomp}, $\sM_{S,k}^\mathrm{ab}$
is the strictly full rigid $k$\nd tensor subcategory
of $\sM_{S,k}$ generated those motives $M$ for which $M$ and $M^\vee(-1)$ are effective.
The motives in $\sM_{S,k}^\mathrm{ab}$ will be called \emph{abelian motives}.
If an abelian scheme of relative dimension $> 0$ over $S$ exists, then
$\sM_{S,k}^\mathrm{ab}$ is generated by the effective motives of degree $1$.
Indeed there then exists by Theorem~\ref{t:abeliansym}
an effective motive $N$ degree $1$ and rank $< 0$ in $\sM_{S,k}$.
Then $\I(-1)$ is a direct summand of $N \otimes N^\vee(-1)$ with $N^\vee(-1)$
effective of degree $1$, and any effective motive $M$ of degree $0$
is a direct summand of $(M \otimes N) \otimes N^\vee$ with $M \otimes N$
by Proposition~\ref{p:degtens} effective of degree $1$.

We denote by  $\sA\sS_S$ the full subcategory of the category of
commutative group schemes
over $S$ consisting of the abelian schemes over $S$.
It has a structure of cartesian monoidal category, and discarding the group structure
defines a faithful product-preserving
functor $\sA\sS_S \to \sV_S$.
By Theorem~\ref{t:abeliansym}, the composite of
$(\sA\sS_S)^\mathrm{op} \to (\sV_S)^\mathrm{op}$ with
$h:(\sV_S)^\mathrm{op} \to \sM_{S,k}$ factors through $\sM_{S,k}^\mathrm{ab}$.
Thus we obtain from $(\sM_{S,k},h,\nu)$ a Poincar\'e duality theory
$(\sM_{S,k}^\mathrm{ab},h^\mathrm{ab},\nu^\mathrm{ab})$ with source $\sA\sS_S$.
Factoring out the maximal tensor ideal of $\sM_{S,k}^\mathrm{ab}$ then gives
a Poincar\'e duality theory
$(\overline{\sM}{}_{S,k}^\mathrm{ab},
\overline{h}{}^\mathrm{ab},\overline{\nu}{}^\mathrm{ab})$.

By an \emph{algebra with involution} in a $k$\nd pretensor category $\sC$ we mean
an algebra $R$ in $\sC$ equipped with an automorphism of algebras $\theta:R \iso R$
such that $\theta^2 = 1_R$.
Such a $\theta$ may be identified with a $\Z/2$\nd grading of $R$ in a
pseudo-abelian hull of $\sC$.
A morphism $R \to R'$
of algebras with involution is a morphism $f:R \to R'$ of algebras
which commutes with the involutions.
To any abelian scheme $A$ over $S$ is associated the algebra $h(X)$ with involution
$h^\mathrm{ab}((-1)_A) = (-1)_{h^\mathrm{ab}(A)}$ in $\sM_{S,k}^\mathrm{ab}$.

\begin{thm}\label{t:alginvunique}
Let $\sJ$ be a proper tensor ideal of $\sM_{S,k}^\mathrm{ab}$, and
denote by $P$ the projection
$\sM_{S,k}^\mathrm{ab}/\sJ \to \overline{\sM}{}_{S,k}^\mathrm{ab}$.
Let $R_1$ and $R_2$ be commutative algebras with involution
in $\sM_{S,k}^\mathrm{ab}/\sJ$ such that $P(R_1)$ and $P(R_2)$ are isomorphic
in $\overline{\sM}{}_{S,k}^\mathrm{ab}$ to algebras with involution
associated to abelian schemes over $S$.
Then above any morphism $P(R_1) \to P(R_2)$
of algebras with involution in $\overline{\sM}{}_{S,k}^\mathrm{ab}$
there lies a unique morphism $R_1 \to R_2$ of  algebras with involution in
$\sM_{S,k}^\mathrm{ab}/\sJ$.
\end{thm}

\begin{proof}
Write $P':\sM_{S,k}^\mathrm{ab} \to \sM_{S,k}^\mathrm{ab}/\sJ$ for the projection.
By hypothesis $P(R_i)$ is isomorphic to the algebra with involution $(PP')(h(A_i))$
for some $A_i$ in $\sA\sB_S$.
For $i = 1,2$, write $N_i$ for the effective motive $h^1(A_i)$ of degree $1$,
and $R_i{}\!^-$ and $h(A_i)^-$ for the respective direct summands of $R_i$ and $h(A_i)$
on which the involutions act as $-1$.
By Theorem~\ref{t:abeliansym}, the embedding $N_i \to h(A_i)^-$
induces an isomorphism of commutative algebras $\Sym N_i \iso h(A_i)$
in $\sM_{S,k}^\mathrm{ab}$.
Choosing an isomorphism of algebras with involution $(PP')(h(A_i)) \iso P(R_i)$,
we obtain a morphism
\[
e_i:(PP')(N_i) \to P(R_i{}\!^-)
\]
in $\overline{\sM}{}_{S,k}^\mathrm{ab}$
which induces an isomorphism $\Sym (PP')(N_i) \iso P(R_i)$ of commutative algebras.
Since $\sM_{S,k}^\mathrm{ab}$ is a Kimura $k$\nd tensor category,
$P$ reflects isomorphisms.
Thus any lifting
$P'(N_i) \to R_i{}\!^-$ of $e_i$ induces an isomorphism
of commutative algebras
\[
\Sym P'(N_i) \iso R_i.
\]
If $\Sym P'(N_i)$ is equipped with its canonical involution acting as $-1$
on $P'(N_i)$, then this isomorphism respects the involutions, by the universal property
of $\Sym P'(N_i)$.
We may thus suppose that $R_i = \Sym P'(N_i)$.

Let $f:\Sym P'(N_1) \to \Sym P'(N_2)$ be a morphism of algebras with involution.
Since $N_1$ and $N_2$ are effective of degree $1$,
Proposition~\ref{p:degtens} shows that the restriction of $f$ to $P'(N_1)$ factors through
$\I \oplus P'(N_2)$, and hence through $P'(N_2)$ because $f$ respects the involutions.
Thus $f = \Sym j$ for a unique $j:P'(N_1) \to P'(N_2)$.
Similarly if $\overline{f}:\Sym (PP')(N_1) \to \Sym (PP')(N_2)$ is a morphism
of algebras with involution then $\overline{f} = \Sym \overline{\jmath}$ for a unique
$\overline{\jmath}:(PP')(N_1) \to (PP')(N_2)$.
By Theorem~\ref{t:deg1ss} any $\overline{\jmath}$ lifts uniquely to a $j$.
Thus any $\overline{f}$ lifts uniquely to an $f$.
\end{proof}

\begin{thm}\label{t:habiso}
Let $\sJ$ and $P$ be as in Theorem~\textnormal{\ref{t:alginvunique}},
and $h_1$ and $h_2$ be symmetric monoidal functors
$(\sA\sS_S)^{\mathrm{op}} \to \sM_{S,k}^\mathrm{ab}/\sJ$ such that
$Ph_1 = \overline{h}{}^\mathrm{ab} = Ph_2$.
Then there is a unique monoidal isomorphism $h_1 \iso h_2$ lying above
the identity of $\overline{h}{}^\mathrm{ab}$.
\end{thm}

\begin{proof}
By Theorem~\ref{t:alginvunique}, there is for every $A$ in $\sA\sS_S$
a unique isomorphism of algebras with
involution $\varphi_A:h_1(A) \xrightarrow{\sim} h_2(A)$
with $P(\varphi_A)$ the identity of $\overline{h}{}^\mathrm{ab}(A)$.
The $\varphi_A$ are the components of a monoidal isomorphism $\varphi$.
Indeed by Theorem~\ref{t:alginvunique}
the squares expressing the naturality and compatibility with
tensor products commute because they lie above commutative squares in
$\overline{\sM}{}_{S,k}^\mathrm{ab}$ with all sides morphisms
of algebras with involution.
Then $P\varphi$ is the identity of $\overline{h}{}^\mathrm{ab}$.
Suppose that also $P\varphi'$ the identity of $\overline{h}{}^\mathrm{ab}$.
Then
$\varphi'{}\!_A:h_1(A) \iso h_2(A)$ is an isomorphism
of algebras with involution with $P(\varphi_A)$ the identity of
$\overline{h}{}^\mathrm{ab}(A)$, so that
by Theorem~\ref{t:alginvunique} $\varphi'{}\!_A = \varphi_A$
for every $A$.
Hence $\varphi' = \varphi$.
\end{proof}

\begin{thm}\label{t:abisounique}
Let $\sJ$ and $P$ be as in Theorem~\textnormal{\ref{t:alginvunique}}.
Then $P$ has a right inverse, which is unique up to unique tensor isomorphism
lying above the identity.
\end{thm}

\begin{proof}
Since $\sM_{S,k}^\mathrm{ab}/\sJ$ is a Kimura $k$\nd category, the existence
and uniqueness up tensor isomorphism of a right inverse to $P$ follows
from Theorems~\ref{t:splitting} and \ref{t:uniquelift}.
It remains to show that if $T$ is a right inverse to $P$ and
$\varphi$ is a tensor automorphism of $T$ with $P\varphi$ the identity,
then $\varphi$ is the identity.
Consider the full subcategory $\sM_0$ of $\overline{\sM}{}_{S,k}^\mathrm{ab}$
consisting of those objects $\overline{M}$ for which
$\varphi_{\overline{M}} = 1_{\overline{M}}$.
By naturality of $\varphi$, it is a strictly full subcategory
of $\overline{\sM}{}_{S,k}^\mathrm{ab}$ which is pseudo-abelian,
and by tensor naturality of $\varphi$, it is a rigid
$k$\nd tensor subcategory of $\overline{\sM}{}_{S,k}^\mathrm{ab}$.
It is to be shown that $\sM_0 = \overline{\sM}{}_{S,k}^\mathrm{ab}$.
If $P':\sM_{S,k}^\mathrm{ab} \to \sM_{S,k}^\mathrm{ab}/\sJ$ is the projection,
it will suffice to show that $\sM_0$ contains $(PP')(M)$ for $M$ either
effective of degree $0$ or $1$ or the object $\I(1)$ in $\sM_{S,k}^\mathrm{ab}$,
because such objects generate $\sM_{S,k}^\mathrm{ab}$ as a rigid $k$\nd tensor
category and $PP'$ is surjective on objects.

Since $P$ is bijective on objects, $T$ sends $(PP')(M)$ to $P'(M)$.
If $i = 0,1$, then $PP'$ is faithful on the
full subcategory of $\sM_{S,k}^\mathrm{ab}$ consisting of effective
motives of degree $i$, by Theorem~\ref{t:deg1ss}.
Thus if $i = 0,1$, then $P$ is faithful on the
full subcategory of $\sM_{S,k}^\mathrm{ab}/\sJ$ consisting of the
$P'(M)$ with $M$ effective of degree $i$.
Hence $(PP')(M)$ lies in $\sM_0$ when $M$ is effective of degree $0$ or $1$.
Similarly $(PP')(M)$ lies in $\sM_0$ if $M = \I(1)$.
\end{proof}

\subsection{The motivic algebra}\label{ss:motalg}

We give a brief account here of the motivic algebra of $\sM{}_{S,k}^\mathrm{ab}$,
which is an algebra in an ind-completion of $\overline{\sM}{}_{S,k}^\mathrm{ab}$
describing the structure of $\sM{}_{S,k}^\mathrm{ab}$.
Since this will not be required for what follows, some details are omitted.

Write $\sM$ for $\sM{}_{S,k}^\mathrm{ab}$ and $P:\sM \to \overline{\sM}$ for the projection.
By Theorem~\ref{t:abisounique}, $P$ has a right inverse.
Any such right inverse is bijective on objects, because $P$ is bijective on objects.
Denote by $\sK$ the set of those $k$\nd tensor functors $K:\sM \to \sM$ for which
there exists a tensor isomorphism $\theta:\Id_\sM \iso K$ with $P\theta = \id_P$.
Then $\sK$ forms a group under composition.
If $T$ and $T'$ are right inverses to $P$, then there is a unique $K \in \sK$ such that
$T' = KT$.
Indeed by Theorem~\ref{t:abisounique} there is a unique $\varphi:T \iso T'$ with
$P\varphi$ the identity, and since $T$ is bijective on objects
there is a unique $\theta$ with $\theta T = \varphi$.

Since $\sM$ is a Kimura $k$\nd tensor category, $\overline{\sM}$ is semisimple abelian.
Let $\widehat{\sM}$ be an ind-completion of $\overline{\sM}$.
Then  Lemma~\ref{l:frmod}~\ref{i:frmodes} with $T$ a right inverse to $P$ shows
that there exists a pair $(R,I)$, with  $R$ a commutative
algebra in $\widehat{\sM}$ and
\[
I:\sF_R \iso \sM
\]
a $k$\nd tensor isomorphism, such that $IF_R$ is right inverse to $P$.
Given also $(R',I')$ with $I':\sF_{R'} \iso \sM$ a $k$\nd tensor isomorphism and
$I'F_{R'}$ right inverse to $P$, there is a unique pair $(j,K)$ with
$j:R \iso R'$ an isomorphism of algebras and $K \in \sK$ such that
$I'F_{R'} = KIF_R$ and $I'\sF_j = KI$.
Indeed the existence and uniqueness of a $K \in \sL$ with
$I'F_{R'} = KIF_R$ has been seen,
and by Lemma~\ref{l:frmod}~\ref{i:frmodff} there is a unique $j$ with
$\sF_j = I'{}^{-1}KI$.
Thus $(R,I)$ is determined uniquely up to a unique isomorphism.
We call $(R,I)$, or simply $R$, the \emph{motivic algebra} of $\sM$.
Since $I$ is a $k$\nd tensor isomorphism, the motivic algebra $R$ describes
the structure of $\sM$ using objects defined in (the ind-completion $\widehat{\sM}$
of) $\overline{\sM}$.

Composing the inverse of \eqref{e:frmodhom} with  $I_{F_R(M),\I}$
gives an isomorphism
\begin{equation}\label{e:motalghom}
\widehat{\sM}(M,R) \iso \sM((IF_R)(M),\I)
\end{equation}
which is natural in the object $M$ of $\overline{\sM}$.
Taking $M = \I$ in \ref{e:motalghom} shows that
\[
\widehat{\sM}(\I,R) = k.
\]
The algebra $R$ has a unique augmentation $p:R \to \I$, and with the identification
$\sF_\I = \overline{\sM}$ we have $\sF_p = PI$.
More generally we obtain a bijection between ideals of $R$ and tensor ideals of
$\sF_R$, and hence of $\sM$, by assigning to $J$ the kernel of $\sF_R \to \sF_{R/J}$.
If $\sJ$ is a tensor ideal of $\sM$, and a motivic algebra for
$\sM/\sJ$ is defined in the same way as for $\sM$,
then it is given by $(R/J,I')$, with $J$ the ideal of $R$ corresponding to $\sJ$
and $I':\sF_{R/J} \to \sM/\sJ$ given by factoring the composite of $I$
with $\sM \to \sM/\sJ$.
Similarly the motivic algebra of $\sM$ for arbitrary $k$ can be obtained from
that for $k = \Q$.

There is a unique $\Z$\nd grading on  $\overline{\sM}$ such that for $i = 0,1$
the image in $\overline{\sM}$ of an effective motive of degree $i$ in $\sM$
lies in $\overline{\sM}_i$, and such that $\I(1)$ lies in $\overline{\sM}_{-2}$.
It extends uniquely to a $\Z$\nd grading on
$\widehat{\sM}$, and $\widehat{\sM}_i$ is an ind-completion of $\overline{\sM}_i$.
It follows from \eqref{e:motalghom} that for given $i$, the $k$\nd vector subspace
$CH^r_i(A)_k$ of  $CH^r(A)_k$ on which the $n_A$ act as $n^{2r-i}$ is $0$ for every $A$
and $r$ if and only if the component $R_i$ of $R$ in $\overline{\sM}_i$ is $0$.
(The analogue for arbitrary $S$ of) Beauville's conjecture in \cite{Bea}
that $CH^r_i(A)_k = 0$ for $i < 0$
is thus equivalent to the statement that $R_i = 0$ for $i < 0$.
Similarly the injectivity of $CH^r_0(A)_k \to \overline{CH}{}^r(A)_k$
is equivalent to the statement that $R_0 = \I$.

Suppose that $S$ is the spectrum of a finite field $\F_q$.
If $\Fr_X$ denotes the Frobenius endomorphism of $X$ in $\sV_S$, then
by \eqref{e:PDhomcomppush} and  \eqref{e:PDhomcomppull} there is a
tensor automorphism of the identity functor
of $\sM_{S,k}$ with component at $h(X)(i)$ given by $q^{-i}h(\Fr_X)(i)$.
By restriction we obtain a tensor automorphism $\Fr$ of the identity
of $\sM$.
It lies above a tensor automorphism $\overline{\Fr}$ of the identity of
$\overline{\sM}$, and $\overline{\Fr}$
extends uniquely to a tensor automorphism $\widehat{\Fr}$ of the identity of
$\widehat{\sM}$.
We then have full $k$\nd tensor subcategories $\overline{\sM}_{00} \subset \overline{\sM}_0$
and $\widehat{\sM}_{00} \subset \widehat{\sM}_0$ of $\overline{\sM}$,
consisting of the objects on which $\overline{\Fr}$ and $\widehat{\Fr}$ respectively
act as the identity, and $\widehat{\sM}_{00}$ is an ind-completion of $\overline{\sM}_{00}$.
It can be seen as follows that
\[
R \in \widehat{\sM}_{00}.
\]
Let $M$ be an irreducible object of $\overline{\sM}$ with $\overline{\Fr}_M \ne 1_M$.
Then $N = (IF_R)(M)$ in $\sM$ lies above $M$,
so that $\Fr_N - 1_N$ is an isomorphism, because its image in $\overline{\sM}$ is.
Thus there are no non-zero morphisms $N \to \I$, because by naturality of $\Fr$
the composite of $\Fr_N - 1_N$ with any $N \to \I$ is $0$.
It then follows from \eqref{e:motalghom} and the naturality of
$\widehat{\Fr}$ that $\widehat{\Fr}_R = 1_R$.

Let $X$ in $\sV_S$ be such that $h(X)$ lies in $\sM{}_{S,k}^\mathrm{ab} = \sM$.
Then we have a canonical commutative algebra structure $\mu$ on $h(X)$ in $\sM$.
If $\overline{h}$ is $h$ composed with the projection
$\sM_{S,k} \to \overline{\sM}_{S,k}$,
then $P(\mu)$ is the canonical commutative algebra structure
$\overline{\mu}$ on $\overline{h}(X)$ in $\overline{\sM}$.
We thus have a commutative algebra structure $I^{-1}(\mu)$ on
\[
F_R(\overline{h}(X)) = I^{-1}(h(X))
\]
in $\sF_R$ such that $\sF_p(I^{-1}(\mu)) = \overline{\mu}$,
where $p:R \to \I$ is the augmentation.
Now $\sF_R$ was defined by factoring the $k$\nd tensor functor $R \otimes -$ from
$\overline{\sM}$ to free $R$\nd modules on objects of $\overline{\sM}$ as a strict
$k$\nd tensor functor which is bijective
on objects followed by a $k$\nd tensor equivalence.
Thus $I^{-1}(\mu)$ gives an element $[I^{-1}(\mu)]$ of the pointed set
$\sD(\overline{h}(X),R)$
of deformations of the commutative algebra  $\overline{h}(X)$ parametrised by $R$.
Explicitly, $\sD(\overline{h}(X),R)$ is the set of
commutative $R$\nd algebras with free underlying $R$\nd module
and fibre $\overline{h}(X)$ along $R \to \I$, up to an isomorphism with fibre the
identity along $R \to \I$.
If $(R',I')$ is another choice of motivic algebra for $\sM$, and if $(j,K)$ is the unique
isomorphism as above from $(R,I)$ to $(R',I')$, then it is easily seen that
the bijection from $\sD(\overline{h}(X),R)$ to $\sD(\overline{h}(X),R')$ defined by $j$
sends $[I^{-1}(\mu)]$ to $[I'{}^{-1}(\mu)]$.
The element $[I^{-1}(\mu)]$ coincides with the base point of $\sD(\overline{h}(X),R)$
if and only if for some, or equivalently for every, right inverse $T$ of $P$, the
algebra $h(X)$ is isomorphic to $T(\overline{h}(X))$.

Let $Q$ be a quotient algebra of $R$ such that $Q = \I \oplus Q_1$ with $Q_1$
in $\widehat{\sM}_1$.
Then $Q_1$ is an ideal of square $0$ in $Q$,
so that $\sD(\overline{h}(X),Q)$ may be identified with
\begin{equation}\label{e:Harr}
\widehat{\sM}(\I,H_\mathrm{Harr}^2(\overline{h}(X),\overline{h}(X)) \otimes Q_1),
\end{equation}
where $H_\mathrm{Harr}^2$ denotes the second Harrison cohomology object,
defined in a similar way as for ordinary commutative algebras over $k$.
The image of $[I^{-1}(\mu)]$ in $\sD(\overline{h}(X),Q)$ is thus
a canonical element $\beta_X$ of \eqref{e:Harr}.
Suppose for example that $k = \Q$ and $S = \Spec(\C)$.
Then Abel-Jacobi equivalence of cycles modulo torsion defines a tensor ideal in $\sM$,
and hence an ideal of $R$.
The quotient of $R$ by this ideal is a $Q$ as above.
Now there exist connected curves $X$ in $\sV_S$ of genus $3$ with the following property
(Ceresa \cite{Cer}):
if $e:X \to A$ is an embedding of $X$ into its Jacobian $A$,
then $e_*(1) - ((-1)_A)^*e_*(1)$ in $CH^2(A)$ has image of
infinite order in the quotient  of the intermediate Jacobian $J^2(A)$ by its maximal
abelian subvariety.
If $X$ is such a curve, then the canonical element $\beta_X$ in \eqref{e:Harr} is non-zero.

According to a conjecture of Kimura, $\sM_{S,k}$ is a Kimura $k$\nd tensor category.
Assuming this conjecture, we may define in the same way as above a motivic algebra
describing the structure of $\sM_{S,k}$.
For an account of this, and its relation to other conjectures see \cite{AndPan},
12.2 and \cite{AndBou}, 4.5.
If $\sM'$ is a full $k$\nd tensor subcategory of $\sM_{S,k}$, it is sometimes
possible to show that the quotient of $\sM'$ by some tensor ideal $\sJ$ is a Kimura
$k$\nd tensor category.
A motivic algebra can then be defined for $\sM'/\sJ$,
but without assuming unproved conjectures its uniqueness
can in general be shown only up to possibly
non-unique isomorphism.

\section{Symmetrically distinguished cycles}\label{s:symdist}

\emph{In this section $k$ is a field of characteristic $0$, and $S$ is a non-empty,
connected, separated, regular excellent noetherian scheme of finite Krull dimension.}

\subsection{The splitting theorem}

Recall that $\sA\sS_S$ denotes the cartesian monoidal category
of abelian schemes (equipped with an identity) over $S$.
It has a dimension function given by relative dimension over $S$.
The forgetful functor $\sA\sS_S \to \sV_S$ preserves products and dimensions.
By pulling back along it, we may regard
$CH(-)_k$ as a Chow theory on $\sA\sS_S$.
We then have a canonical isomorphism $\gamma^\mathrm{ab}$ from $CH(-)_k$ on $\sA\sS_S$
to the Chow theory associated to $(\sM_{S,k}^\mathrm{ab},h^\mathrm{ab},\nu^\mathrm{ab})$.
Let $C$ be a quotient of $CH(-)_k$ on $\sA\sS_S$ by a
proper ideal $J$, and $(\sM,h,\nu)$ be the Poincar\'e duality theory
associated to $C$.
As was seen at the end of \ref{ss:uniPoin}, there is a proper tensor ideal $\sJ$ of
$\sM_{S,k}^\mathrm{ab}$ such that $J$ corresponds under $\gamma^\mathrm{ab}$ to $\sJ$.
Then $\gamma^\mathrm{ab}$ induces an isomorphism from $C$ to the Chow theory
associated to the push forward of $(\sM_{S,k}^\mathrm{ab},h^\mathrm{ab},\nu^\mathrm{ab})$
along the projection onto $\sM_{S,k}^\mathrm{ab}/\sJ$.
The $k$\nd tensor functor
\begin{equation}\label{e:MMJ}
\sM \to \sM_{S,k}^\mathrm{ab}/\sJ
\end{equation}
that gives the morphism from $(\sM,h,\nu)$ to this push forward
defined by the universal property $(\sM,h,\nu)$ is then fully faithful.

\begin{thm}\label{t:Chowsplit}
Let $C$ be a quotient of the Chow theory $CH(-)_k$ on $\sA\sS_S$
by a proper ideal.
Then the projection from $C$ to $\overline{C}$ has a unique right inverse.
\end{thm}

\begin{proof}
Denote by $(\sM,h,\nu)$ the Poincar\'e
duality theory associated to $C$.
Then the Poincar\'e
duality theory associated to $\overline{C}$ is the push forward
$(\overline{\sM},\overline{h},\overline{\nu})$ of $(\sM,h,\nu)$
along the projection $P:\sM \to \overline{\sM}$.
{}From \eqref{e:MMJ} it follows that the pseudo-abelian hull of
$\sM$ is a Kimura $k$\nd tensor category.
By Theorem~\ref{t:habiso}, it also follows that if $h_1$ and $h_2$ are
symmetric monoidal functors from $(\sA\sS_S)^{\mathrm{op}}$ to $\sM$
with
\[
Ph_1 = \overline{h} = Ph_2,
\]
then there is a unique monoidal isomorphism
\begin{equation}\label{e:moniso}
h_1 \iso h_2
\end{equation}
lying above the identity of $\overline{h}$.

Since $P$ is the morphism $(\sM,h,\nu) \to (\overline{\sM},\overline{h},\overline{\nu})$
associated to the projection $C \to \overline{C}$, it will suffice by
Theorem~\ref{t:Poinfunctor} to show that $P$ has a unique right inverse in the
category of Poincar\'e duality theories with source $\sA\sS_S$.
Such a right inverse is a $k$\nd tensor functor
$T:\overline{\sM} \rightarrow \sM$ such that
\begin{enumerate}
\renewcommand{\theenumi}{(\arabic{enumi})}
\item\label{i:rightinverse}
$PT = \Id_{\overline{\sM}}$
\item\label{i:funcompat}
$T\overline{h} = h$
\item\label{i:twists}
$T$ is a Tate $k$\nd tensor functor.
\end{enumerate}
Indeed if \ref{i:rightinverse},
\ref{i:funcompat} and \ref{i:twists} hold then also $T\overline{\nu} = \nu$,
because for any $A$ in $\sA\sS_S$ the functor $P$ sends
the generator $\nu_A$ of the $1$-dimensional $k$-vector space of morphisms
$h(A)(\dim A) \to \I$ to the generator $\overline{\nu}_A$
of the $1$-dimensional $k$-vector space of morphisms
$\overline{h}(A)(\dim A) \to \I$.
It is thus to be shown that there is a unique $k$\nd tensor functor
$T:\overline{\sM} \rightarrow \sM$ such that \ref{i:rightinverse},
\ref{i:funcompat} and \ref{i:twists} hold.

To prove the existence of such a $T$, note first that
there exists by Theorem~\ref{t:splitting} a $k$\nd tensor functor $T$
such that \ref{i:rightinverse} holds.
Then $Ph = \overline{h} = PT\overline{h}$, so that
by \eqref{e:moniso} with $h_1 = h$ and $h_2 = T\overline{h}$ there is a
monoidal isomorphism $\xi:h \iso T\overline{h}$ such that
\begin{equation}\label{e:Pxi}
P \xi = id_{\overline{h}}.
\end{equation}
Since $\overline{h}$ is injective on objects, we may define for
every object $M$ of $\overline{\sM}$ an isomorphism $\zeta_M$ in
$\sM$ with target $T(M)$, by taking
\begin{equation}\label{e:zetaAr}
\zeta_{\overline{h}(A)} = \xi_A
\end{equation}
and $\zeta_M = 1_{T(M)}$
for $M$ not in the image of $\overline{h}$.
Then there is a tensor functor
$\widetilde{T}:\overline{\sM} \rightarrow \sM$ such that the $\zeta_M$
are the components of a tensor isomorphism
\[
\zeta:\widetilde{T} \iso T.
\]
By \eqref{e:zetaAr}, we have
$\zeta \overline{h} = \xi$,
so that $\widetilde{T}\overline{h} = h$.
Similarly $P \zeta$ is the identity of
$\mathrm{Id}_{\overline{\sM}}$,
so that $P \widetilde{T} = \mathrm{Id}_{\overline{\sM}}$.
Replacing $T$ by $\widetilde{T}$, we thus obtain a $T$ such that
\ref{i:rightinverse} and \ref{i:funcompat} hold.

Since $P$ is bijective on objects and sends $\I(1)$ to $\I(1)$,
we have $T(\I(1)) = \I(1)$ by \ref{i:rightinverse}.
The hypotheses of Lemma~\ref{l:compatibletwist} are then satisfied
if we take $\sB = \overline{\sM}$, $\sC = \sM$ and $t = 1_{\I(1)}$,
and as $\sB_0$ the set of those objects of $\overline{\sM}$ which lie
in the image of $\overline{h}$.
Thus we obtain a pair $(T',\varphi)$, with $T':\overline{\sM} \rightarrow \sM$
a Tate $k$\nd tensor functor and
$\varphi:T' \iso T$ a tensor isomorphism, such that
$\varphi_{\overline{h}(A)} = 1_{h(A)}$ and $\varphi_{\I(1)} = 1_{\I(1)}$.
Then $\varphi\overline{h} = \id_h$, so that $T'\overline{h} = h$.
If we replace $\sC$ by $\overline{\sM}$ and $T$ by
$\Id_{\overline{\sM}}$ in Lemma~\ref{l:compatibletwist}, the uniqueness
shows that $(PT',P\varphi)$ coincides with $(\Id_{\overline{\sM}},\id)$,
whence that $PT' = \Id_{\overline{\sM}}$.
Replacing $T$ by $T'$, we thus obtain a $T$ such
that \ref{i:rightinverse}, \ref{i:funcompat} and \ref{i:twists} hold.

Suppose \ref{i:rightinverse}, \ref{i:funcompat} and
\ref{i:twists} hold also with $T$ replaced by $T_1$.
Then by \ref{i:rightinverse} and Theorem~\ref{t:uniquelift},
there is a tensor isomorphism $\rho:T \iso T_1$ with
$P\rho$ the identity of $\Id_{\overline{\sM}}$.
By \ref{i:funcompat}, $\rho \overline{h}$ is a monoidal automorphism
of $h$ lying above the identity of $Ph = \overline{h}$.
Since by \eqref{e:moniso} with $h_1 = h_2 = h$ such an automorphism of
$h$ is unique, $\rho \overline{h}$ is the identity of $h$.
Thus $\rho_{\overline{h}(A)}$ is the identity for every $A$.
By \ref{i:twists} we have
\[
T(\I(i)) = \I(i) = T_1(\I(i))
\]
for every $i$.
Thus $\rho_{\I(i)}$ is the identity, because $P$ induces an isomorphism
between the $1$\nd dimensional $k$\nd vector spaces
of endomorphisms of $\I(i)$ in $\sM$ and $\overline{\sM}$.
By \ref{i:twists} the structural isomorphism
\[
T(\overline{h}(A)) \otimes T(\I(i)) \iso T(\overline{h}(A)(i))
\]
of $T$ is the identity for every $A$ and $i$, and similarly for $T_1$.
Since every object $M$ of $\overline{\sM}$ can be written in the form
$\overline{h}(A)(i)$, it thus follows from the tensor naturality of $\rho$ that
$\rho_M$ is the identity for every $M$.
Hence $T = T_1$.
\end{proof}

\subsection{The unique lifting theorem}

According to Theorem~\ref{t:Chowsplit}, there is associated to every $A$ and
$\overline{\alpha}$ in $\overline{C}(A)$ a well-defined $\alpha$ in $C(A)$,
namely the image of $\overline{\alpha}$ under the unique right inverse
$\overline{C} \to C$ to $C \to \overline{C}$.
In Theorem~\ref{t:Chowunique} below, $\alpha$ will be characterised as the
unique symmetrically distinguished element of $C(A)$ lying above $\alpha$,
in the sense of the following definition.

\begin{defn}\label{d:symdist}
Let $C$ be a quotient of the Chow theory $CH(-)_k$ on $\sA\sS_S$
by a proper ideal, $A$ be an abelian scheme over $S$,
and $\alpha$ be a homogeneous element of $C(A)$.
For each integer $m \ge 0$, denote by $V_m(\alpha)$
the $k$-vector subspace of $C(A^m)$ generated by elements of the form
\begin{equation}\label{e:symdist}
p_*(\alpha^{r_1} \otimes \alpha^{r_2} \otimes \dots \otimes \alpha^{r_n}),
\end{equation}
where $n \le m$, the $r_j$ are integers $\ge 0$,
and $p:A^n \rightarrow A^m$ is a closed immersion
with each component $A^n \rightarrow A$ either a projection
or the composite of a projection with $(-1)_A:A \rightarrow A$.
Then $\alpha$ will be called \emph{symmetrically distinguished}
if  for every $m$ the restriction of the projection
$C(A^m) \rightarrow \overline{C}(A^m)$
to $V_m(\alpha)$ is injective.
An arbitrary element of $C(A)$ will be called symmetrically distinguished
if each of its homogeneous components is symmetrically distinguished.
\end{defn}

For any $\alpha$ the $k$\nd vector space $V_m(\alpha)$ of Definition~\ref{d:symdist} is clearly
finite-dimensional.
Taking the $r_j$ zero in \eqref{e:symdist} shows that each $V_m(\alpha)$
contains the $k$\nd vector subspace $V_m(0)$ of $C(A^m)$ generated by the $p_*(1)$.

Any symmetrically distinguished element of $C(A)$ is symmetric, in the sense that
it is unchanged by pullback or push forward along $(-1)_A$,
because every element of $\overline{C}(A)$ is symmetric.
Suppose that $\alpha$ in $C(A)$ is homogeneous and symmetrically distinguished.
Then $\alpha^r = 0$ when the degree of $\alpha^r$ is greater than $\dim A$,
by the injectivity of $V_1(\alpha) \to \overline{C}(A)$.
Thus for any $r$ the push forward along $A \to S$ of $\alpha^r$ lies in $C^0(S) = k$.

Symmetrically distinguished elements
are stable under change of base scheme, in the following sense.
Let $S'$ be a non-empty,
connected, separated, regular excellent noetherian scheme of finite Krull dimension,
and $S' \to S$ be a morphism of schemes.
Let $C$ be the quotient of $CH(-)_k$ on $\sA\sS_S$
by a proper ideal $J$, and $C'$ be the quotient of $CH(-)_k$ on $\sA\sS_{S'}$
by a proper ideal $J'$, such that for any $A$ in $\sA\sS_S$ pullback along
$A \times_S S' \to S$ sends $J(A)$ into $J'(A \times_S S')$.
Then the pullback of any symmetrically distinguished element of $C(A)$
is symmetrically distinguished in $C(A \times_S S')$.
Indeed any element of $C(A^m)$ with image
in $C'(A^m \times_S S')$ numerically equivalent to $0$ is itself
numerically equivalent to $0$.
It is clear also that symmetrically distinguished elements are stable under
extension of the field $k$, and by projection onto a quotient of $C$.

Theorem~\ref{t:Chowsplit} implies the existence of a symmetrically
distinguished element $\alpha$ of $C(A)$ lying above a given element
$\overline{\alpha}$ of $\overline{C}(A)$, namely the image of
$\overline{\alpha}$ under the right inverse to $C \to \overline{C}$.
The uniqueness of such an $\alpha$ will be proved Theorem~\ref{t:Chowunique} below.
This uniqueness, combined with Theorem~\ref{t:Chowsplit}, will then
imply the stability of symmetrically distinguished elements
under the algebraic operations and pullback and push forward.
To prove Theorem~\ref{t:Chowunique}, we first reformulate in Lemma~\ref{l:symdistequiv}
the condition for an element $\alpha$ of $C(A)$ to be symmetrically distinguished
as a condition on a certain Chow theory $C_\alpha$ defined by $\alpha$.
Theorem~\ref{t:Chowunique} will then deduced from
Theorems~\ref{t:uniquelift} and \ref{t:alginvunique}.

Given an abelian scheme $A$ over $S$, denote by $\sE_A$ the (non-full)
subcategory of $\sA\sS_S$ that consists of the objects $A^m$ for $m = 0,1,2, \dots$
and those morphisms $A^n \to A^m$ for which each component $A^n \to A$ is either
a projection or the composite of a projection with $(-1)_A$.
Then $\sE_A$ has a structure of cartesian monoidal category, where the final object is $A^0$
and the product of $A^m$ and $A^n$
is $A^{m+n}$ with projections the projections in $\sA\sS_S$ onto the first
$m$ and the last $n$ factors.
Any map from $\{1,2, \dots ,m\}$ to $\{1,2, \dots ,n\}$ induces a morphism
from $A^n$ to $A^m$ in $\sE_A$, and any \mbox{$A^n \to A^m$} in $\sE_A$ can be written as the
composite of such a morphism with a product of morphisms $(\pm 1)_A$.
A morphism in $\sE_A$ is an isomorphism in $\sE_A$ if and only if
it is an isomorphism in $\sA\sS_S$, and a section in $\sE_A$ if and only if
it is a closed immersion.
Any morphism $A^n \to A^m$ in $\sE_A$ factors as a projection $A^n \to A^p$
followed by a closed immersion $A^p \to A^m$ in $\sE_A$.

The embedding $\sE_A \to \sA\sS_S$ is a product-preserving functor,
so that we obtain by restriction from $\sA\sS_S$ a dimension function on $\sE_A$.
Given a Chow theory $C$ with source $\sA\sS_S$, we denote by $C_A$ the
Chow theory with source $\sE_A$ obtained from $C$ by restriction.
Given an element $\alpha$ of $C(A) = C_A(A)$, we denote by $C_\alpha$ the
Chow subtheory of $C_A$ generated by $\alpha$, i.e.\ the smallest Chow subtheory
$C'$ of $C_A$ such that $C'(A)$ contains $\alpha$.

Let $A$ be an abelian scheme over $S$ of relative dimension $d$
and with identity $\iota:S \to A$.
Then $h(\iota):h(A) \to \I$ is the projection onto $h^0(A) = \I$.
Since $\nu_A:h(A)(d) \to \I$ factors through $h^{2d}(A)(d)$,
it thus follows from \eqref{e:transpose} that \mbox{$h(\iota)^\vee:\I \to h(A)(d)$}
factors through $h^{2d}(A)(d)$.
Hence if $d > 0$ we have
\begin{equation}\label{e:Cherntriv}
\iota^*\iota_*(1) = 0
\end{equation}
in $CH^d(S)_k$, because by  Proposition~\ref{p:PDhom}~\ref{i:PDhomtrans} and
\ref{i:PDhomgraph} and \eqref{e:PDhomcomppull} the image of $\iota^*\iota_*(1)$
under $\gamma_{S,S,0,d}$ is  $h(\iota)(d) \circ h(\iota)^\vee$.
Write $\delta:A^2 \to A$ for the difference morphism, which
sends the point $(a,a')$ of $A^2$ to $a - a'$.
Then factoring $\delta$ as the involution $(\pr_1,\delta)$ of $A^2$ followed by
the projection $\pr_2$ shows that $\delta^*\iota_*(1)$ is the class
$(\Delta_A)_*(1)$ of the diagonal in $CH^d(A^2)_k$.
Thus for any integers $r$ and $s$
\begin{equation}\label{e:diagpull}
(r_A,s_A)^*(\Delta_A)_*(1) =  ((r-s)_A)^*\iota_*(1) = (r-s)^{2d} \iota_*(1)
\end{equation}
in $CH^d(A)_k$, by  Proposition~\ref{p:PDhom}~\ref{i:PDhomtrans} and
\ref{i:PDhomgraph} and \eqref{e:PDhomcomppull}.

\begin{lem}\label{l:pbeta}
Let $C$, $A$ and $\alpha$ be as in Definition~\textnormal{\ref{d:symdist}},
and denote by $\iota:S \to A$ the identity of $A$.
Suppose that
$\alpha$ is symmetric,
and that
$\iota^*(\alpha)$ and the push forward along $A \to S$ of any power of $\alpha$ lie
in $k \subset C(S)$.
Then $C_\alpha(A^m)$
is the $k$\nd vector subspace of $C_A(A^m) = C(A^m)$ generated
by the elements of the form
\begin{equation}\label{e:pbeta}
p_*(\beta_1 \otimes \beta_2 \otimes \dots \otimes \beta_n),
\end{equation}
where $n \le m$, each $\beta_i$ is either $\iota_*(1)$ or a power of $\alpha$,
and $p:A^n \to A^m$ is a closed immersion in $\sE_A$.
\end{lem}

\begin{proof}
We may suppose that $\dim A > 0$.
Write $U_m$ for the $k$\nd vector subspace of $C(A^m)$ generated by the elements
$\beta_1 \otimes \beta_2 \otimes \dots \beta_m$ with each $\beta_i$ either $\iota_*(1)$
or a power of $\alpha$.
Since $\iota_*(1).\iota_*(1) = \iota_*\iota^*\iota_*(1) =  0$ by \eqref{e:Cherntriv}, and since
$\iota_*(1).\alpha = \iota_*\iota^*(\alpha)$ is a scalar multiple of $\iota_*(1)$
by the condition on $\iota^*(\alpha)$, it is clear that $U_m$ is a graded subalgebra of
$C(A^m)$.
Pullback along $v:A^l \to A^m$ for any $v$ in $\sE_A$ thus sends $U_m$ into $U_l$,
because the algebra $U_m$ is generated by pullbacks along projections $A^m \to A$
of elements of $U_1$, and every element of $U_1$ is symmetric.

The $k$\nd vector subspace $W_m$ of $C(A^m)$ generated by the \eqref{e:pbeta}
is the sum of the $p_*(U_n)$ for $n \le m$ and $p:A^n \to A^m$
a closed immersion in $\sE_A$.
It is a graded subspace of $C(A^m)$ because $\alpha$ is homogeneous.
The element $\iota_*(1)$ of $C(A)$ lies in $C_\alpha(A)$ by \eqref{e:diagpull}
with $r = 1$ and $s = -1$.
Thus $U_n \subset C_\alpha(A^n)$ for every $n$,
and hence $W_m \subset C_\alpha(A^m)$ for every $m$.
We show that $v_*$ sends $W_l$ into $W_m$ and $v^*$ sends $W_m$ into $W_l$
for every $v:A^l \to A^m$ in $\sE_A$.
Since the tensor product of two elements of $W_m$ clearly lies in $W_{2m}$,
pulling back along the diagonal $A^m \to A^{2m}$ will then show
that $W_m$ is graded $k$\nd subalgebra of $C(A^m)$.
The $W_m$ will thus define a Chow subtheory of $C_A$, so that $W_m = C_\alpha(A^m)$
for every $m$, as required.

Since push forward along $A \to S$ sends $\iota_*(1)$ to $1$ and $\alpha^r$
into $k \subset C(S)$,
push forward along any projection $A^j \to A^n$ sends
$U_j$ into $U_n$.
Factoring $u'$ as a projection followed by a closed immersion in $\sE_A$
thus shows $W_m$ is the sum of the $u'{}\!_*(U_{n'})$ for $n' \ge 0$ and
$u':A^{n'} \to A^m$ in $\sE_A$.
It follows that $v_*$ sends $W_l$ into $W_m$ for every $v:A^l \to A^m$ in $\sE_A$.

It remains to show that $v^*$ sends $W_m$ into $W_l$
for every $v:A^l \to A^m$ in $\sE_A$.
This is clear when $v$ is an isomorphism, because then $v^* = (v^{-1})_*$.
If $w$ is the projection $A^{j+1} \to A^j$ onto the first $j$ factors,
then $w^* = - \otimes 1$ sends $W_j$ into $W_{j+1}$.
Any projection in $\sE_A$ is a composite of such $w$ and isomorphisms in $\sE_A$,
so that if $v:A^l \to A^m$ is a projection then $v^*$ sends $W_m$ into $W_l$.
Since an arbitrary $v:A^l \to A^m$ in $\sE_A$ factors as a projection followed
by a closed immersion in $\sE_A$, and any closed immersion in $\sE_A$ is a
composite of isomorphisms and
products of the diagonal $A^1 \to A^2$ with some $A^j$, we reduce to showing
that pullback along
\[
e = \Delta_A \times A^{m-2}:A^{m-1} \to A^m
\]
sends $W_m$ into $W_{m-1}$ for every $m \ge 2$.
Now $W_m$ is the subspace of $C(A^m)$ generated by
the $p_*(1) . \theta = p_*p^*(\theta)$ with
$\theta \in U_m$ and $p:A^n \to A^m$ a closed immersion in $\sE_A$,
because such $p$ are sections in $\sE_A$, so that the $p^*$ induce surjective
maps $U_m \to U_n$.
It is thus enough to show that $e^*u_*(1)$ lies in $W_{m-1}$ for
every $m \ge 2$ and $u:A^n \to A^m$ in $\sE_A$.
Write $\widetilde{e}:A^m \to A^{m-1}$ for the projection
onto the last $m-1$ factors.
Then $\widetilde{e} \circ e$
is the identity of $A^{m-1}$,
so that
\[
e^*u_*(1) = \widetilde{e}_*e_*e^*u_*(1)
= \widetilde{e}_*(u_*(1) . e_*(1))
= \widetilde{e}_*u_*u^*e_*(1)
= (\widetilde{e} \circ u)_*u^*e_*(1).
\]
Since $e_*(1)$ is the pullback of $(\Delta_A)_*(1)$ along the projection
$A^m \to A^2$ onto the first two factors, we reduce finally to showing that
$u_0{}\!^*(\Delta_A)_*(1)$ lies in $W_n$ for every
$u_0:A^n \to A^2$ in $\sE_A$.
Factoring $u_0$ as a projection followed by a section, we may suppose
that it is either $\Delta_A$ or $(1_A,(-1)_A)$.
It then suffices to apply \eqref{e:diagpull}.
\end{proof}

The following Lemma, which will be needed for the proof of Lemma~\ref{l:symdistequiv},
is based on the fact that for $V$ a vector space over $k$ of dimension $d$ we have a canonical
non-degenerate pairing $V \otimes_k \bigwedge^{d-1} V \to \bigwedge^d V$.

\begin{lem}\label{l:poscontr}
Let $f:N' \otimes L \to N \otimes L$ be a morphism in a $k$\nd tensor category,
with $L$ positive of integer rank $d$.
Denote by $f_0:N' \to N$ the contraction of $f$ with respect to $L$
and by $a:L^{\otimes d} \to L^{\otimes d}$ the antisymmetrising idempotent.
Then
\[
f_0 \otimes a =
d(N \otimes a) \circ  \xi_N{}\!^{-1} \circ(f \otimes L^{\otimes (d-1)})
\circ \xi_{N'} \circ (N' \otimes  a):
N' \otimes L^{\otimes d} \to N \otimes L^{\otimes d},
\]
where $\xi_M$ denotes the associativity
$M \otimes L^{\otimes d} \iso (M \otimes L) \otimes L^{\otimes (d-1)}$.
\end{lem}

\begin{proof}
Denote by $\eta:\I \to L^\vee \otimes L$ and $\varepsilon:L \otimes L^\vee \to \I$
the unit and counit of a duality pairing for $L$, and by
$\widetilde{\eta}:\I \to L \otimes L^\vee$ and $\widetilde{\varepsilon}:L^\vee \otimes L \to \I$
the composites of $\eta$ and $\varepsilon$ with the symmetries interchanging $L$ and $L^\vee$.
Write
\[
\psi =
(\widetilde{\varepsilon} \otimes L^{\otimes (d-1)}) \circ \xi_{L^\vee} \circ (L^\vee \otimes a):
L^\vee \otimes L^{\otimes d} \to L^{\otimes (d-1)}.
\]
Then with $\zeta:L \otimes (L^\vee \otimes L^{\otimes d}) \iso
(L \otimes L^\vee) \otimes L^{\otimes d}$ the associativity we have
\begin{equation}\label{e:mupsi}
(\varepsilon \otimes a) \circ \zeta = da \circ (L \otimes \psi).
\end{equation}
Indeed by Theorem~\ref{t:GL}, it is enough to verify \eqref{e:mupsi} when
$L$ is the standard representation of $GL_d$ on $k^d$ in $\Rep_k(GL_d)$.
We may then take for $L^\vee$ the contragredient
action of $GL_d$ on $k^d$, and if $e_1,e_2,\dots,e_d$ is the standard basis of $k^d$
we may take for $\varepsilon$ the $k$\nd homomorphism $k^d \otimes_k k^d \to k$ that
sends $e_i \otimes e_j$ to $1$ when $i = j$ and $0$ otherwise.
The equality \eqref{e:mupsi} now follows from the fact that if $\sigma$ is an
element of $\mathfrak{S}_d$, then
for example $a(e_i \otimes e_{\sigma 2} \otimes \dots \otimes e_{\sigma d})$  is $0$
unless $i = \sigma 1$, when it coincides up to the sign of $\sigma$
with $a(e_1 \otimes e_2 \otimes \dots \otimes e_d)$.

To prove the required result we may suppose that the tensor product is strict,
so that $\zeta$ and the $\xi_M$ are identities.
Then
$(L \otimes \widetilde{\varepsilon}) \circ (\widetilde{\eta} \otimes L)  = 1_L$
by the triangular identity.
Thus by bifunctoriality of the tensor product we have
\begin{equation}\label{e:taupsi}
(L \otimes \psi) \circ (\widetilde{\eta} \otimes L^{\otimes d}) = a.
\end{equation}
Now by definition
$f_0 = (N \otimes \varepsilon) \circ (f \otimes L^\vee) \circ (N' \otimes \widetilde{\eta})$,
so that
\[
f_0 \otimes a = (N \otimes \varepsilon \otimes  a)
\circ (f \otimes L^\vee \otimes L^{\otimes d}) \circ
(N' \otimes \widetilde{\eta} \otimes L^{\otimes d}).
\]
The fact that
$(N \otimes L \otimes \psi) \circ (f \otimes L^\vee \otimes L^{\otimes d})
= (f \otimes L^{\otimes (d-1)}) \circ (N' \otimes L \otimes \psi)$
by bifunctoriality of the tensor product, together with \eqref{e:mupsi} tensored
on the left with $N$ and \eqref{e:taupsi} tensored on the left with $N'$,
thus gives the required result.
\end{proof}

\begin{lem}\label{l:symdistequiv}
Let $C$, $A$ and $\alpha$ be as in Definition~\textnormal{\ref{d:symdist}}.
Then $\alpha$ is symmetrically distinguished if and only if
for every $m$ the restriction to $C_\alpha(A^m) \subset C(A^m)$ of
the projection $C(A^m) \to \overline{C}(A^m)$ is injective.
\end{lem}

\begin{proof}
Each element of
$C(A^m)$ of the form \eqref{e:symdist} lies in $C_\alpha(A^m)$.
Thus $\alpha$ is symmetrically distinguished if  for every $m$
the map $C_\alpha(A^m) \to \overline{C}(A^m)$ is injective.

Conversely suppose that $\alpha$ is symmetrically distinguished.
To prove that for every $m$ the map $C_\alpha(A^m) \to \overline{C}(A^m)$ is injective
we may suppose that $\dim A > 0$.
Write $d = 2^{2\dim A - 1}$.
If  $\iota:S \to A$ is the identity of $A$, we show using Lemma~\ref{l:poscontr} that
\begin{enumerate}
\renewcommand{\theenumi}{(\arabic{enumi})}
\item\label{i:zetal}
there exists a
$\zeta \in V_{2d}(\alpha)$ such that $- \otimes \zeta:C(A^j) \to C(A^{j+2d})$
is injective for every $j$
and such that $\iota_*(1) \otimes \zeta$ lies in $V_{2d+1}(\alpha)$
\item\label{i:alphaC0}
$\iota^*(\alpha)$ lies in $C^0(S)$.
\end{enumerate}
The injectivity of $C_\alpha(A^m) \to \overline{C}(A^m)$ will then be deduced from
\ref{i:zetal} and \ref{i:alphaC0} using Lemma~\ref{l:pbeta}.

Denote by $(\sM,h,\nu)$ the Poincar\'e duality theory
associated to $C$, and by $\gamma$ the universal morphism.
Let $\sM'$ be a pseudo-abelian hull of the Tate $k$\nd tensor category $\sM$.
By \eqref{e:MMJ} and Theorem~\ref{t:abeliansym},
the Hopf algebra $h(A)$ in $\sM'$ is symmetric on a
negative object $h^1(A)$ of rank $-2 \dim A$.
Thus $h(A)$ in $\sM'$ has a canonical grading with component $h^j(A)$ of degree
$j$ the symmetric power $S^j h^1(A)$.
Then $h((-1)_A)$ acts as $(-1)^j$ on $h^j(A)$, so that
\[
e = (1 + h((-1)_A))/2
\]
is idempotent with image
the sum $h^+(A)$ of the $h^{2j}(A)$.
Thus $e = v \circ u$, with $v:h^+(A) \to h(A)$ and $u:h(A) \to h^+(A)$
such that $u \circ v = 1_{h^+(A)}$.
The object $h^+(A)$ is positive of rank $d$.

The antisymmetriser of $h(A)^{\otimes d}$ commutes with $e^{\otimes d}$,
and their composite is an idempotent endomorphism $z$ of $h(A)^{\otimes d}$ with image
$\bigwedge^d h^+(A)$ of rank $1$.
Thus $z$ has trace $1$, so that $- \otimes z$ is injective on any hom-space
of $\sM'$, with a left inverse given by contracting with respect to $h(A)^{\otimes d}$.
We have
\[
z = \gamma_{A^d,A^d,0,0}(\zeta)
\]
for a unique $\zeta$ in $C(A^{2d})$.
Then by Proposition~\ref{p:PDhom}~\ref{i:PDhomtens}, the homomorphism
$- \otimes \zeta$ from $C(A^j)$ to $C(A^{j+2d})$
is injective for every $j$.
Also $z$ is a $k$\nd linear combination of automorphisms $h(w)$
of $h(A)^{\otimes d}$ with $w$ an automorphism of $A^d$ in $\sE_A$.
Thus by Proposition~\ref{p:PDhom}~\ref{i:PDhomgraph} $\zeta$ lies in
$V_{2d}(\alpha)$.
To prove \ref{i:zetal} it remains to show that $\iota_*(1) \otimes \zeta$ lies in
$V_{2d+1}(\alpha)$.

Let $l:h(A)^{\otimes (m+1)}(i) \to h(A)^{\otimes (n+1)}$ be a morphism in $\sM'$,
and denote by \mbox{$l_0:h(A)^{\otimes m}(i) \to h(A)^{\otimes n}$}
the contraction of
\[
l^+ = (h(A)^{\otimes n} \otimes u) \circ l \circ (h(A)^{\otimes m}(i) \otimes v):
h(A)^{\otimes m}(i) \otimes h^+(A) \to h(A)^{\otimes n} \otimes h^+(A)
\]
with respect to $h^+(A)$.
Since the identity of $h^+(A)^{\otimes (d-1)}$ is $u^{\otimes (d-1)} \circ v^{\otimes (d-1)}$,
applying Lemma~\ref{l:poscontr} with $N = h(A)^{\otimes n}$, $N' = h(A)^{\otimes m}(i)$,
$L = h^+(A)$, and $f = l^+$,
and then composing with $h(A)^{\otimes m}(i) \otimes u^{\otimes d}$ on the right and
$h(A)^{\otimes n} \otimes v^{\otimes d}$ on the left and using the compatibility
of $u^{\otimes d}$ and $v^{\otimes d}$ with the antisymmetrisers of $h(A)^d$ and $h^+(A)^d$,
shows that
\begin{equation}\label{e:cyclecontr}
l_0 \otimes z
 = d(h(A)^{\otimes n} \otimes z) \circ \xi' \circ (l \otimes h(A)^{\otimes (d-1)})
\circ \xi(i) \circ (h(A)^{\otimes m}(i) \otimes z),
\end{equation}
where $\xi$ and $\xi'$ are the images under $h$ of the appropriate associativities in $\sA\sS_S$.
We have
\[
l = \gamma_{A^{m+1},A^{n+1},i,0}(\lambda)
\]
for a unique $\lambda$ in $C(A^{m+n+2})$, and
\[
l_0 = \gamma_{A^m,A^n,i,0}(\lambda_0)
\]
for a unique $\lambda_0$ in $C(A^{m+n})$.
Suppose that
\begin{equation}\label{e:lambdaincl}
\lambda \in V_{m+n+2}(\alpha).
\end{equation}
Then by Proposition~\ref{p:PDhom}~\ref{i:PDhomtens} and \ref{i:PDhomgraph},
$l \otimes h(A)^{\otimes (d-1)}$ is the image of an element of $V_{m+n+2d}(\alpha)$
under $\gamma_{A^{m+d},A^{n+d},i,0}$.
Since $h(A)^{\otimes m} \otimes z$ for example is a $k$\nd linear combination
of automorphisms $h(w)$
with $w$ an automorphism of $A^{m+d}$ in $\sE_A$,
it follows from \eqref{e:PDhomcomppush} and \eqref{e:PDhomcomppull}
that the right hand side of \eqref{e:cyclecontr} is also the image
of an element of $V_{m+n+2d}(\alpha)$ under $\gamma_{A^{m+d},A^{n+d},i,0}$.
Thus we have
\begin{equation}\label{e:lambda0incl}
\lambda_0 \otimes \zeta \in V_{m+n+2d}(\alpha),
\end{equation}
by \eqref{e:cyclecontr} and Proposition~\ref{p:PDhom}~\ref{i:PDhomtens}.

Take $m=1$, $n=0$, $i=0$, and for $l$ the multiplication $h(A) \otimes h(A) \to h(A)$.
Then since $h(\iota)$ is the projection $h(A) \to \I$ onto the summand
$\I = h^0(A)$ of $h(A)$,
the diagonal entry $h(A) \otimes h^{2j}(A) \to h^{2j}(A)$ of the matrix
of $l^+$ is $h(\iota) \otimes h^{2j}(A)$.
Thus $l_0 = dh(\iota)$, so that $\lambda_0 = d\iota_*(1)$.
Also $l = h(\Delta_A)$, so that $\lambda$ is the
is the push forward of $1$ in $C(A)$ along the diagonal embedding $A \to A^3$,
and hence \eqref{e:lambdaincl} holds.
Thus \eqref{e:lambda0incl} holds, and $\iota_*(1) \otimes \zeta$ lies in
$V_{2d+1}(\alpha)$.
This proves \ref{i:zetal}.

Now take $m = n = d$, for $-i$ the degree of $\alpha$,
and for $l$ the tensor product $z \otimes \gamma_{S,A,i,0}(\alpha) \otimes h(\iota)$.
Since $h(\iota)$ factors through $h^+(A)$, the contraction $l_0$ of $l^+$
with respect to $h^+(A)$ coincides with the contraction
\[
z \otimes (h(\iota) \circ \gamma_{S,A,i,0}(\alpha)) =
\gamma_{A^d,A^d,0,0}(\zeta) \otimes \gamma_{S,S,i,0}(\iota^*(\alpha))
\]
of $l$ with respect to $h(A)$.
Thus $\lambda_0$ is given, up to a symmetry of $A^{2d}$, by
$\iota^*(\alpha) \otimes \zeta$.
Also $\lambda$ is given, up to a symmetry of $A^{2d+2}$, by
$\iota_*(1) \otimes \zeta \otimes \alpha$.
Hence \eqref{e:lambdaincl} holds.
Thus \eqref{e:lambda0incl} holds, and $\iota^*(\alpha) \otimes \zeta^{\otimes 2}$
lies in $V_{4d}(\alpha)$.
If $-i \ne 0$, then the image of $\iota^*(\alpha)$ in $\overline{C}(S)$ and
hence of $\iota^*(\alpha) \otimes \zeta^{\otimes 2}$ in $\overline{C}(A^{4d})$  is $0$.
Since by hypothesis $\alpha$ is symmetrically distinguished,
it follows that  $\iota^*(\alpha) \otimes \zeta^{\otimes 2}$ and hence $\iota^*(\alpha)$
is $0$.
This proves \ref{i:alphaC0}.

Since $\alpha$ is symmetrically distinguished, $\alpha$ is symmetric and the push forward
along $A \to S$ of any power of $\alpha$ lies in $C^0(A)$.
By \ref{i:alphaC0}, the hypotheses of Lemma~\ref{l:pbeta} are thus satisfied.
Hence $C_\alpha(A^m)$ is generated
as a $k$\nd vector subspace of $C(A^m)$ by the elements \eqref{e:pbeta}.
For $n \le m$, the tensor product of a
$\beta_1 \otimes \beta_2 \otimes \dots \otimes \beta_n$ as in \eqref{e:pbeta}
with $\zeta^{\otimes m}$ lies in $V_{n + 2dm}(\alpha)$, by \ref{i:zetal}.
Thus the map $- \otimes \zeta^{\otimes m}$ from $C(A^m)$ to $C(A^{(2d+1)m})$,
which by \ref{i:zetal} is injective, sends
$C_\alpha(A^m)$ into $V_{(2d+1)m}(\alpha)$.
The required injectivity of $C_\alpha(A^m) \to \overline{C}(A^m)$ now
follows from the injectivity of $V_{(2d+1)m}(\alpha) \to \overline{C}(A^{(2d+1)m})$.
\end{proof}

\begin{thm}\label{t:Chowunique}
Let $C$ and $A$ be as in Definition~\textnormal{\ref{d:symdist}}.
Then above any element of $\overline{C}(A)$
there lies a unique symmetrically distinguished element of $C(A)$.
\end{thm}

\begin{proof}
Let $\overline{\alpha}$
be a homogeneous element of $\overline{C}(A)$.
By Theorem~\ref{t:Chowsplit}, the projection $C \to \overline{C}$ has a right
inverse $\tau$.
Then $\tau$ sends $\overline{\alpha}$ to a symmetrically distinguished element
of $C(A)$ lying above $\overline{\alpha}$.
It remains to show that if $\alpha'$ and $\alpha''$ in $C(A)$  are homogeneous
and symmetrically distinguished and lie above
$\overline{\alpha}$, then \mbox{$\alpha' = \alpha''$}.
By Lemma~\ref{l:symdistequiv}, the projection
$\pi:C_A \to \overline{C}_A$ defines by restriction isomorphisms from
$C_{\alpha'}$ and $C_{\alpha''}$ to $\overline{C}_{\overline{\alpha}}$.
The composites $\kappa'$ and $\kappa''$ of their inverses with the
embeddings of $C_{\alpha'}$ and $C_{\alpha''}$
into $C_A$ are then morphisms $\overline{C}_{\overline{\alpha}} \to C_A$ for which
\[
\pi \circ \kappa' = \kappa = \pi \circ \kappa'',
\]
with $\kappa:\overline{C}_{\overline{\alpha}} \to \overline{C}(A)$ the embedding.
Since $\kappa'$ and $\kappa''$ send $\overline\alpha$ in
$\overline{C}_A(A) = \overline{C}(A)$ to
$\alpha'$ and $\alpha''$ in $C_A(A) = C(A)$,
it will suffice to show that $\kappa' = \kappa''$.

Denote by $(\sM_0,h_0,\nu_0)$ and $(\sM_1,h_1,\nu_1)$ the
Poincar\'e duality theories with source $\sE_A$ associated respectively to
$\overline{C}_{\overline{\alpha}}$ and $C_A$.
The push forward of $(\sM_1,h_1,\nu_1)$ along the projection
$P:\sM_1 \to \overline{\sM}_1$ is then the Poincar\'e duality theory
associated to $\overline{C}_A$, and $P$ is the morphism
of  Poincar\'e duality theories associated to $\pi$.
If $K'$ and $K''$ are the morphisms $(\sM_0,h_0,\nu_0) \to (\sM_1,h_1,\nu_1)$ associated
to $\kappa'$ and $\kappa''$ and $K$ is the morphism associated to $\kappa$, then
\[
PK' = K =  PK''
\]
with $K:\sM_0 \to \overline{\sM}_1$ faithful, and
\[
K' h_0 = h_1 = K'' h_0.
\]
We show that $K' = K''$.
By Theorem~\ref{t:Poinfunctor}, this will imply that
$\kappa' = \kappa''$ as required.

If $(\sM,h,\nu)$ is the Poincar\'e duality theory associated to $C$,
then the universal property of $(\sM_1,h_1,\nu_1)$ defines a morphism $\sM_1 \to \sM$
from $(\sM_1,h_1,\nu_1)$ to the restriction of $(\sM,h,\nu)$ to $\sE_A$.
It is fully faithful.
Hence its composite
\[
E:\sM_1 \to \sM_{S,k}^\mathrm{ab}/\sJ.
\]
with \eqref{e:MMJ} is fully faithful.
Also $Eh_1$ is the restriction to $\sE_A$ of the composite $\widehat{h}$
of $h^\mathrm{ab}$ with the projection onto $\sM_{S,k}^\mathrm{ab}/\sJ$.
Thus $E$ sends the algebra $h_1(A^n)$ with involution $h_1((-1)_{A^n})$
to the algebra $\widehat{h}(A^n)$ with involution
$\widehat{h}((-1)_{A^n}) = (-1)_{\widehat{h}(A^n)}$.
Let $\widetilde{\sM}_0$ and $\widetilde{\sM}_1$
be pseudo-abelian hulls of $\sM_0$ and  $\sM_1$.
Since $E$ is fully faithful,
$\widetilde{\sM}_1$ is a Kimura category.
By Theorem~\ref{t:abeliansym}, $(-1)_{\widehat{h}(A^n)}$
splits $\widehat{h}(A^n)$ as the direct sum of a positive object
on which it acts as $1$ and a negative object on which it acts as $-1$.
Since the composite of either $K'$ or $K''$ with $E$ is faithful
and sends $h_0(A^n)$ to $\widehat{h}(A^n)$ and $h_0((-1)_{A^n})$
to $(-1)_{\widehat{h}(A^n)}$,
it follows that ${h_0((-1)_{A^n}})$ similarly splits
$h_0(A^n)$ in $\widetilde{\sM}_0$.
Thus $h_0(A^n)$ is a Kimura object in  $\widetilde{\sM}_0$ for every $n$.
Hence $\widetilde{\sM}_0$ is a Kimura category, because the objects of
$\sM_0$ are the
$h_0(A^n)(i)$.

By Theorem~\ref{t:uniquelift} with $\sD = \sM_0$
and $\sC = \sM_1$, there is a tensor isomorphism
\[
\varphi:K' \iso K''
\]
with $P\varphi$ the identity.
Then $\varphi  h_0$ is a monoidal automorphism of $h_1$.
The automorphism $\varphi_{h_0(A^n)}$ of $h_1(A^n)$ respects
the algebra structure of $h_1(A^n)$ by monoidal naturality of $\varphi h_0$,
and the involution $h_1((-1)_{A^n})$ of $h_1(A^n)$ by naturality of $\varphi h_0$.
Applying $E$ and using
Theorem~\ref{t:alginvunique} thus shows that $\varphi_{h_0(A^n)}$ is the identity
for every $n$.
Also $\varphi_{\I(i)}$ is the identity for every $i$,
because $P$ induces an isomorphism from $\End_{\sM_1}(\I(i)) = k$
to $\End_{\overline{\sM}_1}(\I(i)) = k$.
Since every object of $\sM_0$ is of the form
$h_0(A^n)(i)$ and $K'$ and $K''$
are Tate $k$\nd tensor functors, the tensor naturality of $\varphi$
thus shows that $\varphi$ is the identity.
Thus $K' = K''$ as required.
\end{proof}

\begin{cor}\label{c:symdist}
Let $C$ be as in Definition~\textnormal{\ref{d:symdist}}.
Then the assignment to each $A$ in $\sA\sS_S$ of the set of symmetrically
distinguished elements of $C(A)$ defines a Chow subtheory $C_{sd}$ of $C$, and
the projection $C \to \overline{C}$ induces an isomorphism $C_{sd} \to \overline{C}$.
\end{cor}

\begin{proof}
By Theorem~\ref{t:Chowsplit}, the projection $\pi:C \to \overline{C}$
has a right inverse
$\tau:\overline{C} \to C$.
If $\overline{\alpha}$ is an element of $\overline{C}(A)$, then
$\tau_A(\overline{\alpha})$ is a symmetrically distinguished element of $C(A)$
lying above $\overline{\alpha}$, and by Theorem~\ref{t:Chowunique} it is the
unique such element of $C(A)$.
An element $\alpha$ of $C(A)$ is thus symmetrically distinguished if and only
if $(\tau \circ \pi)_A(\alpha) = \alpha$.
The result follows.
\end{proof}

We note finally a criterion for symmetrically distinguished elements
which follows from Theorem~\ref{t:Chowunique}.
Write $(\sM,h,\eta)$ for the push forward of
$(\sM_{S,k}^\mathrm{ab},h^\mathrm{ab},\nu^\mathrm{ab})$ along the projection
onto a quotient $\sM$ of $\sM_{S,k}^\mathrm{ab}$
corresponding to the quotient $C$ of $CH(-)_k$ on $\sA\sS_S$.
We have a canonical isomorphism
$\gamma$ from $C$ to $\sM(\I,h(-)(\cdot))$,
which induces an isomorphism $\overline{\gamma}$ from
$\overline{C}$ to $\overline{\sM}(\I,\overline{h}(-)(\cdot))$,
where $\overline{h} = Ph$ with $P:\sM \to \overline{\sM}$ the projection.
If $T$ is right inverse to $P$,
then for any $A$ in $\sA\sS_S$ there is by
Theorem~\ref{t:alginvunique} a unique
isomorphism of algebras
\[
u:T(\overline{h}(A)) \iso h(A)
\]
with $P(u)$ the identity
which respects the involutions induced by $(-1)_A$.
Then $\alpha \in C^i(A)$ above $\overline{\alpha}$ in $\overline{C}{}^i(A)$ is
symmetrically distinguished if and only if
\begin{equation}\label{e:TPA}
\gamma_{S,A,-i,0}(\alpha) = u \circ T(\overline{\gamma}_{S,A,-i,0}(\overline{\alpha})).
\end{equation}
By Theorem~\ref{t:abisounique}, this condition
is independent of the choice of $T$.
That for given $\overline{\alpha}$ there is a unique
$\alpha$ above $\overline{\alpha}$ for which
\eqref{e:TPA} holds is clear.
To see that this $\alpha$ is symmetrically distinguished, we may suppose by
Lemma~\ref{l:compatibletwist} that the restriction of $T$ to the full $k$\nd pretensor
subcategory of $\overline{\sM}$ consisting of twists of objects in the image
of $\overline{h}$ is a Tate $k$\nd pretensor functor.
By Theorem~\ref{t:habiso}, $u$ is the component at $A$ of a monoidal isomorphism
$\varphi:T\overline{h} \iso h$ with $P(\varphi)$ the identity.
If we write $\beta$ for \eqref{e:symdist},
$\overline{\beta}$ for the image of $\beta$ in $\overline{C}(A^m)$,
and $r$ for $r_1 + \dots + r_n$, then
\[
\gamma_{S,A^m,-ri,0}(\beta) = \varphi_{A^m} \circ
T(\overline{\gamma}_{S,A^m,-ri,0}(\overline{\beta}))
\]
by Proposition~\ref{p:PDhom}.
The injectivity of $V_m(\alpha) \to \overline{C}(A^m)$ follows.

\subsection{Concluding remarks}\label{ss:conclrem}

Let $C$ be a quotient of $CH(-)_k$ on $\sV_S$ by a proper ideal.
It is natural to consider, instead of the condition of Definition~\ref{d:symdist},
the weaker condition on an $\alpha$ in $C^i(A)$ where
the $p$ in \eqref{e:symdist} run only over those
closed immersions $A^n \to A^m$ for which each component $A^n \to A$ is a projection.
Call such an $\alpha$ distinguished.
Since the involution $(-1)_A$ is no longer used, we have also a notion of distinguished
element in $C^i(X)$ for an arbitrary $X$ in $\sV_S$.
Denote by $D^i(X)$ the set of such elements.
It is non-empty if and only if $0$ is distinguished.
The sets $D^i(X)$ are stable under multiplication
by an element of $k$ and pullback and push forward along
isomorphisms in $\sV_S$, and the $r$th power of an element of $D^i(X)$
lies in $D^{ri}(X)$ and the tensor product of elements
of $D^i(X)$ and $D^j(Y)$ lies in $D^{i+j}(X \times _S Y)$.
Distinguished elements are stable under change of base scheme, in the same sense as
for  symmetrically distinguished elements.
For an abelian scheme $A$ over $S$, the symmetrically
distinguished elements of $C^i(A)$ and their translates are distinguished.

Write $(\sM,h,\eta)$ for the push forward of $(\sM_{S,k},h_{S,k},\nu_{S,k})$
along the projection onto the quotient $\sM$ of $\sM_{S,k}$
corresponding to the quotient $C$ of $CH(-)_k$.
There is a canonical isomorphism
$\gamma$ from $C$ to $\sM(\I,h(-)(\cdot))$,
which induces an isomorphism $\overline{\gamma}$ from
$\overline{C}$ to $\overline{\sM}(\I,\overline{h}(-)(\cdot))$,
where $\overline{h}$ is $h$ composed with the projection.
If $\sM^0$ is the strictly full subcategory of $\sM$ with objects isomorphic to
those in the image of
$\sM_{S,k}^\mathrm{ab}$, then the projection $\sM^0 \to \overline{\sM}{}^0$
has a right inverse $T$.
Denote by $\sV^0_S$ the set of those $X$ in $\sV_S$ with $h(X)$ in $\sM^0$,
and by $\sV^{00}_S$ the set of those $X$ in $\sV^0_S$ with $h(X)$ isomorphic as an
algebra in $\sM^0$ to $T(\overline{h}(X))$.
By Theorem~\ref{t:abisounique}, the last condition is independent of the choice of $T$.
Both $\sV^0_S$ and $\sV^{00}_S$ are stable under products.
If $X$ lies in $\sV^0_S$ and if $\alpha$ in $C^i(X)$ above
$\overline{\alpha}$ in $\overline{C}{}^i(X)$ is such that
\begin{equation}\label{e:TPX}
\gamma_{S,X,-i,0}(\alpha) = v \circ T(\overline{\gamma}_{S,X,-i,0}(\overline{\alpha})),
\end{equation}
for some isomorphism of algebras
\[
v:T(\overline{h}(X)) \iso h(X)
\]
in $\sM^0$ lying above the identity in $\overline{\sM}{}^0$,
then it can be seen in a similar way as for \eqref{e:TPA} above that $\alpha$ is distinguished.
Thus for $X$ in $\sV^{00}_S$, every fibre of the projection
$C^i(X) \to \overline{C}{}^i(X)$ contains at least one distinguished element.

A sufficient condition for $X$ in $\sV_S$ to lie in $\sV^{00}_S$ is that the geometric generic fibre
of $X$ should be an abelian variety.
If $X$ in $\sV_S$ is of relative dimension $1$, then the necessary and sufficient condition
for $X$ to lie in $\sV^{00}_S$ is that the algebra $h(X)$ should have a $\Z$\nd grading
lifting the canonical $\Z$\nd grading of $\overline{h}(X)$.
In particular if $S$ is the spectrum of an algebraic extension of a finite field, then any
curve in $\sV_S$ lies in $\sV^{00}_S$.
If $S$ is the spectrum of an arbitrary field, then any hyperelliptic curve in $\sV_S$
lies in $\sV^{00}_S$.
On the other hand, suppose that $S = \Spec(\C)$, and that $C$ is a finer quotient of $CH(-)_k$
than that defined by algebraic equivalence of cycles modulo torsion.
Then by the result of Ceresa \cite{Cer},
if $X$ is a sufficiently general connected proper smooth curve of genus $\ge 3$ over $S$
and if $A$ is the Jacobian of $X$, the class of $X$ in $C(A)$ does not coincide with its
pullback along $(-1)_A$.
Any such $X$ lies outside $\sV^{00}_S$.

As was mentioned at the end of \ref{ss:splitlift}, it seems very plausible that Theorem~\ref{t:uniquelift}
can be proved in the stronger form where $\sD$ is an arbitrary $k$\nd pretensor category.
If this were so, it would follow that for every $X$ in $\sV^{0}_S$,
the distinguished elements of $C^i(X)$
above $\overline{\alpha}$ in $\overline{C}{}^i(X)$ are exactly the elements $\alpha$ for
which there is an isomorphism of algebras $v$ lying above the identity such that \eqref{e:TPX} holds.
This would imply that the $X$ in $\sV^{00}_S$ are exactly the $X$ in $\sV^{0}_S$ for which
$0$ in $C(X)$ is distinguished.
For an abelian scheme $A$ over $S$, it would imply that the symmetrically distinguished
elements of $C^i(A)$ are exactly those which are symmetric and distinguished.
If we assume for simplicity that $S$ is the spectrum of an algebraically closed field,
it would further imply that any two distinguished elements of $C^i(A)$ lying above the
same element of $\overline{C}{}^i(A)$ are translates of one another.

\end{document}